\documentclass[11pt]{amsart}

\usepackage{amsmath}
\usepackage{amsthm}
\usepackage{amssymb}
\usepackage{graphicx}

\numberwithin{equation}{section}

\newcommand{\loc}{{\scriptsize{loc}}}
\newcommand{\e}{\varepsilon}

\newcommand{\vp}{\varphi}

\newcommand{\R}{{\mathbb R}}
\newcommand{\T}{{\mathbb T}}

\newcommand{\C}{{\mathbb C}}

\newcommand{\calN}{{\mathcal N}}

\newcommand{\calA}{{\mathcal A}}
\newcommand{\calL}{{\mathcal L}}
\newcommand{\calV}{{\mathcal V}}
\newcommand{\calC}{{\mathcal C}}
\newcommand{\calD}{{\mathcal D}}
\newcommand{\calT}{{\mathcal T}}

\newcommand{\calU}{{\mathcal U}}
 
\newcommand{\calH}{{\mathcal H}}

\newcommand{\sign}{\operatorname{sign}}
\newcommand{\dist}{\operatorname{dist}}

\newcommand{\dep}{{\delta_\e}}
\def\rest{\hskip 1pt{\hbox to 10.8pt{\hfill
\vrule height 7pt width 0.4pt depth 0pt\hbox{\vrule height 0.4pt
width 7.6pt depth 0pt}\hfill}}}

\newcommand{\rvec}{\underline} 
\newcommand{\cvec}{\bar} 

\newcommand{\yt}{{y^\tau}}
\newcommand{\yn}{{y^\nu}}

\newcommand{\beq}{\begin{equation}}
\newcommand{\eeq}{\end{equation}}

\newtheorem{theorem}{Theorem}
\newtheorem{lemma}{Lemma}[]
\newtheorem{proposition}{Proposition}[]

\newtheorem{corollary}{Corollary}

\theoremstyle{remark}
\newtheorem{remark}{Remark}[section]

\begin{document}

\author[R.L. Jerrard]{Robert L. Jerrard }

\address{Department of Mathematics, University of Toronto, Toronto, Canada M5S 2E4}
\email{rjerrard@math.toronto.edu}

\thanks{The author was partially supported by the National Science and
Engineering Research Council of Canada under operating Grant 261955.}
\title[semilinear wave equations and timelike minimal surfaces]
{Defects in semilinear wave equations and timelike minimal surfaces in
Minkowski space}

\begin{abstract} We study semilinear wave
equations with Ginzburg-Landau type nonlinearities 
multiplied by a factor $\e^{-2}$, where $\e>0$ is a small 
parameter. We prove that for suitable initial data, solutions exhibit energy concentration sets that evolve
approximately via the equation for timelike Minkowski minimal surfaces, as long as the minimal surface remains smooth. This gives a proof of predictions made, on the basis of formal asymptotics and other heuristic arguments, by cosmologists studying cosmic strings and domain walls, as well as by applied mathematicans.
\end{abstract}
\maketitle

\today


\section{introduction}

In this paper we prove that if $\Gamma$ is a timelike minimal surface in Minkowski space
$\R^{1+N}$ of codimension $k=1$ or $2$, smooth in a time interval $(-T,T)$, then for suitable initial data, solutions $u:\R^{1+N}\to \R^k$, $N>k$ of the equation
\begin{equation}
\Box u
+ \frac 1{\e^2} f(u)  = 0,  \quad\quad 0<\e \ll 1 
\label{slw}\end{equation}
exhibit an energy concentration set that approximately follows $\Gamma$, at least up to time $T$.
Here the model nonlinearity is
$f(u) = (|u|^2-1)u$ in low dimensions; in 
higher dimensions, we take $f$ to be a qualitatively similar nonlinearity satisfying growth conditions
that leave the equation \eqref{slw} globally well-posed; see \eqref{scalarF}, \eqref{vectorF} for precise assumptions.

Our main motivation for this work comes from 
the very rich mathematical
literature on corresponding questions about elliptic and parabolic
analogs of \eqref{slw}, which
have been studied in great detail for about the past 30 years. In the elliptic case,
these past results establish deep connections between
energy concentration sets  in solutions $u:\Omega\subset \R^N\to \R^k$ of the equation
\beq
-\Delta u + \frac 1{\e^2} f(u)  = 0, \quad\quad\quad0<\e\ll 1
\label{sle}\eeq
and (Euclidean) minimal surfaces of codimension $k$ in $\Omega$. 
Similarly, the parabolic equation
\beq
u_t-\Delta u + \frac 1{\e^2} f(u)  = 0, \quad\quad\quad0<\e\ll 1, \quad\quad
u:(0,T)\times\R^N\to \R^k
\label{slp}\eeq
is related to the geometric evolution problem of codimension $k$ motion by mean curvature. 
Our results address the natural question of whether any parallel results hold, relating the
semilinear for wave equation \eqref{slw} to the timelike Minkowski minimal surface problem, which is a geometric wave equation.


It turns out that this question is also relevant to the
description of cosmological domain walls ($k=1$) and 
strings ($k=2$) ; see Kibble \cite{kibble} for a seminal early paper
and Vilenkin  and Shellard \cite{vs} for an in-depth survey of a  large body of work on related 
questions. 
The questions we study
have also been addressed in the applied math literature 
by  Neu \cite{neu}, with some generalizations considered by
Nepomnyashchy and Rotstein \cite{rn}. 
We will not say any more about any of these applications in this paper,
except to note that our main results can be described as giving a rigorous derivation, 
in the relatively simple and physically unrealistic setting of a
scalar particle described by \eqref{slw}, of the laws of motion
for cosmic strings and domain walls,
deduced formally by cosmologists over 30 years ago.


%

\subsection{mathematical background}\label{S:bg}


We first review results about the elliptic and parabolic equations
\eqref{sle} and \eqref{slp}. Throughout this discussion we consider the model nonlinearity $f(u) = (|u|^2-1)u$.

In the elliptic case, and when $k=1$ (so that \eqref{sle} is
a scalar equation), the general heuristic principle underlying essentially every work
we know of is that
\beq
u \ \approx \  q(\frac{d} \e)
\label{pwd}\eeq
where $q:\R\to \R$  solves
\beq
\mbox{ $-q''+ f(q) =0$, \quad\quad\quad $q(0)= 0,\ \  q(x)\to \pm 1$ as $x\to \pm \infty$}
\label{q.def}\eeq
and $d:\Omega \to \R$
is the signed distance function to a {\em minimal} hypersurface $\Gamma\subset \Omega$, so that
$d$ is characterized near $\Gamma$
by the properties
\beq
\mbox{$d = 0$ on $\Gamma, \quad \quad |\nabla d|^2 = 1$ near $\Gamma$}
\label{d1}\eeq
and $\Gamma$ satisfies
\beq
\mbox{  (Euclidean) mean curvature }=0.
\label{ems}\eeq
There are a vast number of results establishing various forms of these assertions.
Roughly speaking, these fall into two families. The first  (see for example
Modica \cite{mod} or Hutchinson and Tonegawa \cite{ht}) employ variational
and measure theoretic methods, together with elliptic estimates, to characterize the
limiting behavior of sequences of solutions as $\e\to 0$. These proofs generally establish some
form of what is called equipartition of energy, which can be viewed as a weak
form of the description \eqref{pwd}.
The second family of proofs (see for example Pacard and Ritor\'e \cite{pr}) employ Liapunov-Schmidt reduction and related
arguments, relying ultimately on the the implicit function theorem and
control of the spectrum of some linearized operator.
These arguments yield existence results that give very precise descriptions,
in the spirit of \eqref{pwd}, of the solutions that are constructed. 

In the $k=1$ scalar case of the parabolic equation \eqref{slp}, more or less the same 
heuristic \eqref{pwd}, \eqref{q.def} holds,
except that now $d$ is a function of $t$ and $x$, and for every $t$, $d(t,\cdot)$ is 
the signed distance function from a hypersurface $\Gamma_t$, so that
\[
\mbox
{$d(t,\cdot) = 0$ on $\Gamma_t$ and $|\nabla_x d(t, )|^2 = 1$ near $\Gamma_t$},
\]
with $\Gamma := \cup_{t>0} \{ t\} \times  \Gamma_t\subset (0,T)\times \R^N$ satisfying 
\beq
\mbox{velocity} = \mbox{ mean curvature}.
\label{mcf}\eeq
Different versions of this result have been established by a variety of proofs, including linearization
techniques (see de Mottoni and Schatzmann \cite{dms}) which establish a strong form of \eqref{pwd}, but are valid only
locally in $t$; maximum principle arguments  which ultimately rely
on an ansatz based on
\eqref{pwd} to build sub- and super-solutions (see \cite{xc, ess}), or which employ 
a change of variables motivated by \eqref{pwd} and techniques for
weak passage to limits \cite{bos}; and 
measure theoretic methods combined with parabolic estimates as in Ilmanen \cite{ilm}, in which
\eqref{pwd} appears in the weak form of assertions about equipartition of energy.
The maximum principle and measure theoretic arguments give weaker
descriptions that are however valid globally in $t$, with \eqref{mcf} understood in
a weak sense.

In vector-valued $k=2$ case, for both the elliptic  \eqref{sle} and parabolic \eqref{slp} systems, 
we do not know of any characterization as precise as \eqref{pwd}; obstacles to such results include the
difficulty of describing
rotational degrees of freedom, and the related  poor behavior of the spectrum of certain linearized operators.
But there are a number of results showing in various degrees of generality for solutions of
\eqref{sle} (including among others  \cite{lr1, bbo, abo})  and \eqref{slp} (see  \cite{as, lr2, bos} for example)
 with suitable energy bounds, that energy concentrates around a codimension $2$ submanifold $\Gamma$  satisfying \eqref{ems}, respectively \eqref{mcf}.  These results generally employ elliptic or parabolic estimates, some of which are extremely delicate,  in combination with
measure theoretic arguments, and they provide information, customarily
phrased in the language
of varifold convergence, about the precise way in which energy concentrates around the codimension 
2 surface $\Gamma$.

All results about \eqref{sle} and \eqref{slp} rely very heavily on tools that  are not available for hyperbolic equations, such as maximum principles (in the scalar case) and elliptic or parabolic regularity. Thus they do not give much  indication of how to proceed for the nonlinear wave equation \eqref{slw}.
We know of only two partial exceptions to this rule. First, there is no abstract reason that linearization
arguments should be impossible in the hyperbolic setting; they appear however to be  hard to carry through. Second, a number of papers, starting with \cite{bk}, study \eqref{slp} using weighted energy estimates. In 
particular, we mention an argument presented by Soner in a 1995 lecture series \cite{soner} for the scalar parabolic equation \eqref{slp}, and developed in \cite{jso, lin1} for parabolic systems.
This argument relies on a rather straightforward but remarkable computation of
$
\frac d{dt} \int_{\R^N} \zeta e_\e(u) dx
$,
where $e_\e(u)$ is a natural energy density associated with a solution $u$ of \eqref{slp}, and $\zeta$ is a smooth function
such that $\zeta(t,x) = \frac 12 \dist(x, \Gamma_t)^2$ near $\Gamma_t$, where the latter solves \eqref{mcf}. This calculation certainly uses the parabolic character of \eqref{slp}, but it is not clear if it uses it in a really essential way. Indeed, 
our main proofs originated as an attempt to develop an analog of this argument in the hyperbolic setting.

Much less work has been done on 
the hyperbolic equation \eqref{slw} than on its elliptic and parabolic counterparts.
The few papers that we are aware of mostly study situations rather different from those we consider here, including:
\begin{itemize}
\item
works  \cite{j-wave, l-wave} that characterize the behavior of
solutions of \eqref{slw} in the limit $\e\to 0$ in the case $N=k=2$, for the model nonlinearity
$f(u) = (|u|^2-1)u$.
\item
a paper of Gustafson and Sigal \cite{gs} that
studies the Maxwell-Higgs model, in which \eqref{slw}, with the model nonlinearity
$f(u) = (|u|^2-1)u$, is coupled to a electromagnetic field,
when $N=k=2$ and $0< \e \ll 1$.
\item work of Stuart \cite{stu} studying an equation of the form \eqref{slw}
on a Lorentzian manifold and
with a focussing nonlinearity, for $0<\e\ll 1$, see also \cite{stu2}.
\end{itemize}
In all these papers, energy concentrates around points, know as vortices or quasiparticles depending 
on the situation, and these points evolve according to an ODE.
These results are valid only as long as the points remain separated from each other. The fact that points
are geometrically very simple objects 
makes the analysis easier in some ways than in the problems we consider here, where the same role is
now played by submanifolds of dimension $n\ge 1$.
An additional significant simplifying factor in all the papers cited above, except those of Stuart, 
is that they study a scaling in 
which vortices move at subrelativistic velocites, that is, velocities that tend to $0$ as $\e\to 0$. 

It is also worth mentioning work \cite{sc} of Cuccagna that studies \eqref{slw} in $\R^{1+3}$ with $\e=1$
and establishes scattering for initial data
$(u, u_t)|_{t=0}$  a small, very smooth perturbation of  $(q(x^3), 0)$. This can be seen as an analog for \eqref{slw} of results \cite{lind, br}
that establish scattering for solutions of the timelike Minkowski minimal surface problem
with initial data that is a small,  perturbation of a motionless hyperplane.

As far as we know, the only work of rigorous mathematics that addresses exactly the questions
we consider here is a  recent preprint of Bellettini, Novaga, and Orlandi \cite{bno}. Its main result identifies some conditions that, { if} they could be verified, would suffice to imply that a varifold
obtained from a sequence of solutions $(u_\e)$ of \eqref{slw} satisfying natural energy bounds
is stationary with respect to the Minkowski inner product structure. These conditions include lower
density bounds as well as, roughly speaking, some quite strong constraints on the limiting tangent space.  
The results we obtain here are stronger than those projected in \cite{bno}, as discussed in Remark \ref{R:varifold}

\subsection{ new results}

In many ways our results follow the pattern described above.
In the case $k=1$ of a scalar equation,  as in earlier work on the elliptic and parabolic problems,
we obtain, for suitable initial data, a description of solutions of \eqref{slw} parallel to \eqref{pwd}, \eqref{q.def}, \eqref{d1}, \eqref{ems},
with the Euclidean metric replaced by the Minkowski metric in the last two identities. 
And  in the case $k=2$, we
prove that for solutions of \eqref{slw} with suitable initial data, energy concentrates around a codimension $2$ surface 
$\Gamma$ that satisfies \eqref{ems}, again the Euclidean metric replaced by the Minkowski metric.
We also give a precise description of the way that this concentration occurs; in fact we obtain
this description in the case $k=1$ as well.

The strongest results (for example Bethuel, Orlandi and Smets \cite{bos}) on the parabolic equation \eqref{slp} hold globally for $t>0$, and
assume only natural energy bounds on the initial data.
Our results, by contrast, are valid only locally in $t$ --- that is, as long as the surface $\Gamma$ remains smooth --- and require rather special initial data. We note however that results like those we obtain are almost certainly {\em not true} globally in $t$ or for general initial data.

In all our results, we take the timelike minimal surface 
$\Gamma$ to have the topology of $(-T,T)\times \T^n$, where $n=N-k$.
This covers the important example of a closed string in $\R^3$, when $k=2$.
In fact, we view the global topology of $\Gamma$ as relatively unimportant, since
our results are in some sense local, and since both the semilinear wave equation \eqref{slw}
and the timelike minimal surface equation enjoy finite propagation speed, 
In any case, our methods should extend to $\Gamma\cong (-T,T)\times M$
for more general $M$.

Quite general results in Milbredt \cite{mil} imply in particular the local existence of smooth
timelike minimal surfaces $\Gamma$ given smooth data at $t=0$.

In the scalar case, we assume that the nonlinearity $f$ in \eqref{slw} has the
form $f = F'$, where
$F:\R\to \R$
is a smooth  function such that
\beq
F(\pm 1) = 0, \quad\quad
c(1- |s|)^2\le F(s) 
\label{scalarF}\eeq
We also assume that $f$ grows sufficiently slowly that \eqref{slw} is globally
well-posed in $\dot H^1\times L^2$. We may take $f(u) = (u^2-1)u$ if $N\le 4$.

In the statement of our results we  use the notation 
\beq
e_\e(u) :=  \frac 12(u_t^2 + |\nabla u|^2) + \frac 1{\e^{2}}F(u)
\label{eep.def}\eeq
and
\beq
\kappa_1 \  :=  \ \  \int_{-1}^1 \sqrt{2 F(s)} \ ds.
\label{kappa.def0}\eeq
One can think of $\kappa_1$ as a constant related to the surface tension of an interface.
Our main results in the scalar case can be summarized as follows:

\begin{theorem}
Let $\Gamma\subset (-T,T)\times \R^N$ be a smooth timelike minimal hypersurface.
Let  $\Gamma\cap (\{t\}\times \R^N)  := \Gamma_t$, and assume that for
every $t\in (-T,T)$, $\Gamma_t$ is diffeomorphic to the torus $\T^{n}$, $n=N-1$.

Then given $T_0<T$, there exists  a neighborhood $\calN$ of  $\Gamma$ in $(-T_0,T_0)\times \R^N$ in 
which there exists a smooth solution $d:\calN\to \R$ of the problem
\beq
\mbox{$d = 0$ on $\Gamma, \quad \quad -d_t^2 +  |\nabla d|^2 = 1$ near $\Gamma$}.
\label{d1a}\eeq
(In other words, $d$ is the
signed {\em Minkowski} distance to $\Gamma$, compare \eqref{d1}.)
Moreover, there exists a solution $u$ of \eqref{slw} (with $f$ as described above)
such  that  for any  $T_0<T$,
\beq
\| u - q(\frac d \e) \|_{L^2(\calN)} \le C \sqrt \e,
\label{t1.c1}\eeq
where $q$ solves \eqref{q.def}, and
\beq
\int_{ \calN} 
d^2\  e_\e(u) \ 
dt\ dx  \ + \ 
\int_{[ (-T_0,T_0)\times \R^N] \setminus \calN} 
 e_\e(u) \ 
dt\ dx \le  C \e
\label{t1.c2}\eeq
In addition, if $\calT_\e(u) = (\calT_{\e, \beta}^\alpha(u))_{\alpha,\beta=0}^N$ and $\calT(\Gamma) = 
 (\calT_{\beta}^\alpha(\Gamma))_{\alpha,\beta=0}^N$
denote the energy-momentum tensors for $u$ and $\Gamma$
(defined in \eqref{emt1.def}
and \eqref{emt2.def} respectively) then
\beq
\left\| \frac \e{\kappa_1} \calT_\e(u) - \calT(\Gamma) \right\|_{W^{-1,1}((-T_0,T_0)\times \R^N)} \le C \e
\label{t1.c3}\eeq
In all these conclusions, $C= C(T_0, \Gamma)$ is
independent of $\e$.
\label{T1}\end{theorem}

\begin{remark}
The definitions imply that $\calT^0_{\e,0}(u) = e_\e(u)$, and that $\calT^0_0(\Gamma)$
is a measure supported on $\Gamma$ and defined
by
\[
\int f(t,x)  d\calT^0_0 = \int_{-T}^T \int_{\Gamma_t}  f(t,x ) (1-V^2)^{-1/2} \calH^{n}(dx) \ dt
\]
where $V(t,x)$ denotes the (euclidean) normal velocity of $\Gamma$ at a point $(t,x)\in \Gamma$.
We can denote this measure by $(1-V^2)^{-1/2} (\calH^{n}\rest \Gamma_t)\otimes dt$.
The conclusion \eqref{t1.c3} thus implies in particular that
\beq
\left\| \frac \e{\kappa_1} e_\e(u)  \ - \ 
(1-V^2)^{-1/2} (\calH^{n}\rest \Gamma_t)\otimes dt \right\|_{W^{-1,1}((-T_0,T_0)\times \R^N)} \le C \e.
\label{t1.c4}\eeq
A parallel remark holds for
conclusion \eqref{t2.c1} of Theorem \ref{T2} below.
\end{remark}

\begin{remark}\label{R:data}
Our assumptions are satisfied for example if 
\beq
u(0, x)  =  q( \frac{ d(0,x)}{\e} ),\quad\quad\quad u_{t}(0, x) = \frac 1\e q'(\frac{d(0,x)}\e) d_t(0,x)
\label{u.data1}\eeq
in a neighborhood $\calN_0$ of $\Gamma_0$, and if
\beq
\int_{\{0\} \times (\R^N\setminus \calN_0)}  e_\e(u)  \  dx \le \e.
\label{u.data2}\eeq
See Lemma \ref{L.z0} for details. 
\end{remark}


In the vector case, we can again take $f(u) = (|u|^2-1)u$ if $N\le 4$,
or in other words, $f =\nabla_u F$, for $F(u) = \frac 14 (|u|^2-1)^2$.
More generally, we require of $f$ only that the equation \eqref{slw}
be globally well-posed in $\dot H^1\times L^2$, and that $f = \nabla_uF$, where 
\beq
c (1-|u|)^2 \le F(u) \le C(1-|u|)^2
\quad\mbox{ for }|u|\le 2\quad\quad\mbox{ and }
F(u) \ge c>0\mbox{ for }|u|\ge 2.
\label{vectorF}\eeq
We summarize our results in the vector $k=2$ case
in the following:

\begin{theorem}
Let $\Gamma\subset (-T,T)\times \R^N$ be a smooth timelike minimal surface of codimension $k=2$.
Let  $\Gamma\cap (\{t\}\times \R^N)  := \Gamma_t$, and assume that for
every $t\in (-T,T)$, $\Gamma_t$ is diffeomorphic to the torus $\T^{n}$, $n=N-2 \ge 1$.

Then there exists a solution for \eqref{slw} (with $k=2$) such that for any $T_0<T$,
there is a constant $C$ such that
\beq
\int_{(-T_0,T_0)\times \R^N} \tilde d^2 e_\e(u)  \ dt \ dx \le C
\label{t2.c2}\eeq
where $\tilde d(t,x) = \min \{ 1, \dist( (t,x), \Gamma ) \}$,
and
\beq
\left\| \frac 1{\pi|\ln\e|} \calT_\e(u) - \calT(\Gamma) \right\|_{W^{-1,1}((-T_0,T_0)\times \R^N)} \le C |\ln \e|^{-1/2}
\label{t2.c1}\eeq
where $\calT_\e(u)$ and $\calT(\Gamma)$
denote the energy-momentum tensors for $u$ and $\Gamma$
(defined in \eqref{emt1.def}
and \eqref{emt2.def} respectively).
In all these conclusions, $C= C(T_0, \Gamma)$ is independent of
independent of $\e$.
\label{T2}\end{theorem}

\begin{remark}
In Lemma \ref{L.z0} we give an explicit construction of 
initial data for which the conclusions of the theorem hold.\label{R.data2}
\end{remark}

\begin{remark}
The proof  shows that the solutions $u$ from Theorem \ref{T2}  have a defect near $\Gamma$;
see \eqref{iterate.c5} for a precise, if opaque, version of this assertion.
\end{remark}

\begin{remark}
In both the above theorems, the constants $C$ in the conclusions are at least
exponential in $T_0$. That is, our proofs yield constants of the form $C = a e^{b T_0}$,
where $a,b$ themselves depend on $\Gamma$ and $T_0$, and may blow up as $T_0\nearrow T$.
\end{remark}

\begin{remark}\label{R:varifold} Our results imply in particular that if we fix $\Gamma$ as in either of the theorems above, then there exists a sequence $( u_\e )$ of solutions of \eqref{slw} such that the energy-momentum tensors
$\dep \calT_\e(u_\e)$ converge weakly as measures  in $(-T,T)\times \R^N$to  $\calT(\Gamma)$
if the scaling factor $\dep = \dep(k)$ is chosen correctly. This can be seen as a form of varifold convergence, analogous to results
proved in \cite{ilm,  as, bbo,bos} for elliptic and parabolic equations, and discussed in the hyperbolic
case in \cite{bno}.

By providing quantitative estimates of $\|\dep \calT_\e(u) - \calT(\Gamma)   \|_{W^{-1,1}}$, however, our results are 
sharper than simple convergence results.
This sharpening is significant, because convergence results strictly analogous to known results in the elliptic or parabolic cases  {\em can fail} in the hyperbolic setting. That is, in our setting (but {\em not} for elliptic
or parabolic problems) there exist 
sequences of solutions $(u_\e)$ such that $\dep \calT_\e(u_\e)$ converges to a 
measure-valued tensor $\calT$ supported on a codimension $k$ set, but such that $\calT$ is not the energy-momentum tensor for any timelike minimal surface $\Gamma$ --- in other words, $\calT$ is not weakly stationary;
see Section \ref{S:ex} below for explicit examples. 
\end{remark}

\begin{remark}
If we fix $\Gamma$ and consider an associated sequence 
$(u_\e)$ of solutions as found in Theorem \ref{T2} with $\e\to 0$,
the uniform energy bounds \eqref{t2.c2} away from $\Gamma$ combined with a
classical argument of Shatah \cite{sh} imply that after passing to a subsequence,
$u_\e$ converges weakly in $H^1_\loc([(-T,T)\times \R^N]\setminus \Gamma)$
to a wave map into $S^1$.
\end{remark}

\begin{remark}\label{R:k1k2}
In both theorems, we ultimately rely on energy estimates in a frame that moves with $\Gamma$.
These estimates (summarized in Theorem \ref{T3}) assert more or less that energy remains  concentrated around $\Gamma$
on the same scale for $0<t<T$ as it is at $t=0$.
The hypotheses for Theorem \ref{T3} are 
\begin{itemize}
\item small energy away from $\Gamma_0$, see \eqref{idata1};
\item a defect near $\Gamma_0$, see \eqref{idata4}; and
\item small energy, {\em given the presence of the defect}, near $\Gamma_0$, {\em in a frame that moves with
$\Gamma$,} see \eqref{idata2} and \eqref{idata3}.
\end{itemize}
Theorems \ref{T1}, \ref{T2} follow from the special case of Theorem \ref{T3} in
which the energy is, roughly speaking, as concentrated as possible around $\Gamma_0$.
The fact that our results for $k=1$ are considerably stronger than for $k=2$
stems ultimately from the fact that when $k=1$, for initial data that is nearly energetically optimal --- essentially, \eqref{u.data1}, \eqref{u.data2}
or suitable small perturbations thereof --- the energy is very sharply concentrated around $\Gamma_0$,
whereas when $k=2$, for the model initial data, energy is quite spread out. A 
more precise expression
of this fact appears in \eqref{ek1k2}.

\end{remark}

\subsection{about the proofs}

A main issue in the analysis of \eqref{slw} 
is to establish some kind of stability property of the moving defect  --- that is, the
interface ($k=1$) or ``string'' ($k=2$).
The relativistic invariance of the equation suggests that
a defect   should acquire extra energy when it accelerates (and this is confirmed by our results, for example \eqref{t1.c4}), so we must rule out this extra energy as a potential source of instability. 
Our analysis  starts from the observation that, for a solution that behaves as predicted in the formal arguments of \cite{vs, neu, rn} and others, a moving defect will always appear to be energetically optimal  in the frame of reference of an observer who is moving with the defect.


\subsubsection{ change of variables}
Motivated by this, we begin by  rewriting the equation in a 
frame that follows the timelike minimal surface $\Gamma$, where the defect is expected
to remain. In these variables, 
our task is to show that the solution is approximately constant, and 
we expect the defect to have some optimality property that we can exploit.

To define the change of variables, we start with a map  $H$ defined on $(-T,T)\times \T^n$ and 
parametrizing $\Gamma\subset (-T,T)\times \R^N$,
and we extend $H$ to a diffeomorphism  $\psi$ between, essentially,
a neighborhood in $(-T,T)\times \T^n\times \R^k$ of $(-T,T)\times \T^n$ and a neighborhood of $\Gamma$
in $\R^{1+N}$.
We write $\psi$ as a function
of variables $y = (y^0,\ldots, y^N) = (\yt, \yn)$, where $\yt = (y^0,\ldots, y^n)$ are variables
tangent to $\Gamma$, and $\yn = (y^{n+1},\ldots, y^N)$ corresponds to directions normal to $\Gamma$.
We always arrange that $y^0$ is a timelike coordinate and that all other coordinates are spacelike.

We will also write for example
$D_\tau = (\partial_{y^0},\ldots, \partial_{y^n})$
and $\nabla_\nu = (\partial_{y^{n+1}},\ldots, \partial_{y^N})$. We generally write $D$ for a
space-time gradient, and $\nabla$ for a gradient involving space-like variables only.

We then define $v = u \circ \psi$ on the domain of $\psi$. We find it
convenient to write the equation satisfied by  $v$ 
(that is,  equation \eqref{slw}, expressed in terms of the $y$ variables)
in the form
\beq
- \partial_{y^\alpha}( g^{\alpha\beta} \partial_{y^\beta} v) - b \cdot D v + \ \frac 1{\e^2} f(v)=0,\quad\quad
\quad\quad b^\beta :=
\frac{\partial_{y^\alpha}\sqrt{-g}}{\sqrt{-g}} g^{\alpha\beta}.
\label{v.eqn0}\eeq
Here $G = (g_{\alpha\beta})$ is the expression in the $y$ coordinates of the Minkowski metric, $(g^{\alpha\beta})
= (g_{\alpha\beta})^{-1}$,  $g = \det(g_{\alpha\beta})$, and we implicitly sum over repeated indices. 
The equation \eqref{v.eqn0} enjoys certain useful properties, which are
summarized in Proposition \ref{P.transformation}.
Some of these follow from the specific form we choose for the map $\psi$, 
and the fact that $\Gamma$ is a timelike minimal
surface implies a key property of the coefficient $b$ of the first-order term: 
\beq
|b^\nu| \le C| \yn|\quad\mbox{ at }y  = (\yt, \yn), \mbox{ for }
b^\nu := (b^{n+1}, \ldots, b^N).
\label{bgood}\eeq
We emphasize that the verification of \eqref{bgood} is {\em  the only place} in our analysis where we explicitly 
invoke the fact that $\Gamma$ is a minimal surface. 

\subsubsection{energy estimates}
We now focus on $v$ solving \eqref{v.eqn0} on, say,  
$(-T_1,T_1)\times \T^n\times B_\nu(\rho_0)$ for some $T_1<T$ and $\rho_0>0$, where $B_\nu(\rho_0) := \{ \yn\in \R^k_\nu: |\yn|< \rho_0\}$.
We will use the notation
\beq
e_{\e,\nu}(v) :=
 \frac 12 |\nabla_\nu v|^2 + \frac 1{\e^2}F(v).
\label{eenu.def}\eeq
We introduce a scaling factor $\dep = \dep(k)$, see \eqref{dep.def1}, chosen so that, heuristically,
\beq
 \dep \int_{\{ \yn\in \R^k_\nu : |\yn|\le \rho\}}  e_{\e,\nu}(v)(\yt, \cdot)  d\yn \ge 1 - o_\e(1),\quad
\mbox{if $v(\yt, \cdot)$ has a defect near $\yn=0$}
\label{dep.def}\eeq
for every fixed $\rho_1$. This is made precise below.
One of our goals is to show that if 
\[
 \zeta_3(s) :=
\dep \left. \int_{ \T^n\times W_\nu(s)}  |D_\tau v|^2 + |\yn|^2 e_{\e,\nu}(v)\  dy^1\cdots dy^N\right|_{y^0=s} 
\]
is small when $s=0$, say, then it remains small for a range of positive $s$. Here  $W_\nu(s)$
is a neighborhood of the origin in $\R^k_\nu$ that may depend on the parameter $s$
but will always contain a ball of fixed radius $\rho$.
The smallness of $\zeta_3$ is consistent with $v$ having a large amount of energy, as long as it
involves mostly the normal energy $e_{\e,\nu}(v)$ and is concentrated very near the codimension
$k$ surface $\{ \yn = 0\}$.

Our strategy is to define some quantity $\zeta_1(s)$
such that
\beq
\zeta_1'(s) \le C\zeta_3(s)\label{z1g1}\eeq
and such that, under  suitable additional assumptions,
\beq
\zeta_1(s) \ge c \zeta_3(s) - o_\e(1).
\label{z1g2}\eeq
A main task will then be to show that these additional assumptions are 
preserved by the equation \eqref{v.eqn0}. If we can do this, we
can easily use Gr\"onwall's inequality to control the growth of $ \zeta_3$.

For the verification of \eqref{z1g1}, we define the  approximately\footnote{The {\em exact} law expressing conservation of energy for \eqref{slw} can
of course be transposed to the $y$ coordinates. As far as we know this is not 
useful for our problem, since it does not distinguish any good property
of equation \eqref{v.eqn0} resulting from the fact that the change of 
variables is built around a parametrization
of a {minimal} surface.} conserved energy density
\beq
e_\e(v) = \frac 1 2  
 a^{\alpha \beta } v_{y^\alpha} v_{y^\beta}
+ \frac 1{\e^2} F(v)
\label{eep.def1}\eeq
where $a^{\alpha\beta}$ is a positive definite matrix related to $g^{\alpha\beta}$, see 
\eqref{A.def}. (When we want to avoid any possibility of confusion, 
we will write $e_\e(v; G)$ for the above quantity,
and $e_\e(u;\eta)$ for the energy defined in \eqref{eep.def}, with $\eta$ denoting the
expression in the original coordinates of the Minkowski metric.)
We further define
\[
\zeta_1(s) := \dep \left. 
\int_{ \T^n\times W_\nu(s)} (1+ \kappa_2|\yn|^2)e_\e(v) \ dy^1\cdots dy^N
\right|_{y_0=s} - 1
\]
where 
$\kappa_2$
is a constant to be selected in a moment. (It will turn out later that we can take $\kappa_2=1$ in the scalar case.) 
We hope to show that $\zeta_1$ satisfies properties \eqref{z1g1}, \eqref{z1g2} above.

Indeed, as long as the sets $W_\nu(s)$ are chosen to shrink rapidly enough, we show in Section 
\ref{S:energy1} that the verification of \eqref{z1g1} follows quite easily from  the differential inequality
\beq
\frac \partial{\partial y^0} e_\e(v) \le  \sum_{i=1}^N\frac \partial{\partial y^i}\vp^i + C (| D_\tau v|^2 + |\yn|^2 |\nabla_\nu v|^2)
\label{iec}\eeq
for some vector $\vp = (\vp^1,\ldots, \vp^N)$.  
The differential inequality \eqref{iec} in turn follows easily
from \eqref{v.eqn0}, see Lemma \ref{L.eflux}. The key point in \eqref{iec} is the factor 
$|\yn|^2$, which follows from \eqref{bgood} and hence from the fact that $\Gamma$ is a minimal surface.

To check \eqref{z1g2}, we first note that some of the good properties of \eqref{v.eqn0}
alluded to above imply that if $\kappa_2$ is chosen in a suitable way, see \eqref{omega3.def}, then
\[
(1 + \kappa_2|\yn|^2) e_\e(v) \ge \ c |D_\tau v|^2 + (1+|\yn|^2) e_{\e,\nu}(v).
\]
With this choice of $\kappa_2$, 
\[
\zeta_1(s) \ge c \zeta_3(s) + 
\int_{ \T^n}\left.\left(   \dep \int_{W_\nu(s)} e_{\e,\nu}(v)  d\yn - 1\right) dy^1\cdots dy^n
\right|_{y_0=s} .
\]
Thus, in view of the choice \eqref{dep.def} of $\dep$, we can deduce \eqref{z1g2} as long as we can check that
$v(s,\cdot)$ has a defect confined near
$ \{ (y^1,\ldots, y^N)\in \T^n\times \R^k_\nu:  \yn = 0\}$. (This is the additional assumption
mentioned before \eqref{z1g2}.)

\subsubsection{a certain stability property}
We therefore introduce a ``defect confinement functional'' 
$\calD: H^1(\T^n\times B_\nu(\rho_0))\to \R$
that is designed to have two properties. (This functional takes quite different forms in the two cases $k=1,2$ that we consider, see \eqref{calD.k1def} and \eqref{L1|||}.)
First, we require that
\beq
\calD(v(s,\cdot))\mbox{ small }\Rightarrow \mbox{ ``defect is confined'' }
\Rightarrow
\mbox{ lower energy bounds } \Rightarrow \eqref{z1g2}\mbox{ holds}.
\label{calD0}\eeq
This sort of argument will eventually lead to an inequality of the simple form
\beq
\zeta_3(s) \le C[\zeta_1(s) + \zeta_2(s)] + o_\e(1),
\label{calD1}\eeq
where
\[
\zeta_2(s) = \calD(v(s)).
\]
Second, we need $\calD$ to be such that
\beq
\mbox{ changes in $\zeta_2(s)$ can be controlled by $\zeta_3(s)$.}
\label{calD2}\eeq
Concrete versions of \eqref{calD0} and \eqref{calD2} are established in Section \ref{S:energy1} for $k=1$
and Section \ref{S:vector_ee} for $k=2$.
Heuristically, \eqref{calD2} should hold because, if the
defect strays away from $\yn=0$, then it should carry with it 
concentrations of energy that can be detected by $\zeta_3$.
In the case $k=1$, \eqref{calD2} will take the simple form
$\zeta_2(s) \le 2\zeta_2(0) + C \int_0^s \zeta_3(\sigma) d\sigma$.
The corresponding estimate for $k=2$ is similar but slightly more complicated.
In both cases, however, by combining \eqref{calD1} and a specific
concrete version of \eqref{calD2} with \eqref{z1g1}, we obtain control
over $\zeta_i(s)$ for $i=1,2,3$. This gives us a good deal of information
about the behavior of $v$, from which all of our main conclusions are ultimately deduced.

One can view \eqref{calD1}, \eqref{calD2} as a weak stability property of 
states $w$ for which $\calD(w)$ is small and for which the the inequality in \eqref{calD1} is almost saturated.

The difference in the strength of our conclusions in the cases $k=1,2$,
discussed in Remark \ref{R:k1k2}, stems from the fact that
for optimal initial data, 
\beq
\mbox{for }i=1,2,3,\quad
\zeta_i(0) \approx \begin{cases}
\e^2&\mbox{ when }k=1\\
|\ln \e|^{-1}&\mbox{ for } k=2.
\end{cases}
\label{ek1k2}\eeq
(See Lemma \ref{L.z0}.)
This reflects sharper energy concentration around $\{\yn=0\}$ in the case $k=1$.

\subsubsection{some other issues}
The change of variables that we employ is defined only in a neighborhood of $\Gamma$.
We must therefore combine estimates of $v$ near $\Gamma$
with estimates of $u$ away from $\Gamma$, and then iterate. We verify in Section \ref{S:iterate}  that this
can be done in such a way as to genuinely yield estimates valid up to $(-T_0,T_0)\times \R^N$ for arbitrary $T_0<T$.

Spacelike hypersurfaces of the form $\{y^0 = \mbox{constant}\}$
play a distinguished role in our argument, as it is along these surfaces that
the defect structure is nearly energetically optimal for the solutions
$v$ that we consider. This near-optimality is manifested for example in the fact that
inequality \eqref{calD1} is nearly saturated.
In general our change of variables $\psi^{-1}$ maps the hypersurface
$\{ (t,x)\in \R^{1+N} : t=0\}$, on which we assume the data for the solution $u$
of \eqref{slw} is given, onto a hypersurface that is smooth and spacelike but otherwise can be quite arbitrary. 
So a certain amount of work is needed to obtain control
of $v$ on a suitable portion of some hypersurface $\{y^0 = \mbox{constant}\}$. 
This is done in Sections \ref{S:nonzero} and \ref{S:vinitial}, and  involves
mainly technical adjustments to our basic energy estimates as outlined above.
This means that we carry out our main energy estimates twice, once
in a simpler form that can easily be iterated, and once to deal with
complications caused by the geometry  of the initial hypersurface in the transformed variables.
This and the similarity between the cases $k=1,2$ leads to a certain amount of redundancy, which 
however enables us 
to present out argument first in a relatively simple setting, in Section \ref{S:energy1}; we believe
this makes the main ideas easier to grasp.

The technical work of Section \ref{S:nonzero}
could be avoided if we insisted on prescribing data only on spacelike
hypersurfaces that have the form $\{y^0=\mbox{constant}\}$ near $\Gamma_0$, but we feel that 
this would be unnecessarily restrictive.

Finally, we extract all the conclusions of the main theorems from control over quantities such
as $\zeta_1,\zeta_2, \zeta_3$ above. This is done in Section 6. In the vector case, these arguments require a useful
recent estimate of Kurzke and Spirn \cite{ks}, without which we would not be able to establish 
the full energy-momentum tensor estimate \eqref{t2.c1}.

\subsection{some examples}\label{S:ex}

It is well-known that the timelike minimal surface equation for $1+1$-dimensional surfaces in $\R^{1+N}$
is explicitly solvable for every $N\ge 2$. In particular, if $a: \R\to \R^N$ and $b:\R\to \R^N$
are smooth maps such that $|a'| = |b'| = 1$, then the function
\[
X(s,t) := (t, x(s,t)), 
\quad\quad x(s,t): = \frac 12 (a(s+t) + b(s-t))
\]
parametrizes a surface that satisfies the timelike minimal surface equation wherever it is smooth.
(See for example the exposition in \cite{vs}, chapter 6.)
This implies in particular that 
if $g:\R\to \R^{k}$ is any smooth function (where $k=N-1$), then
\beq
\Gamma := \{ (t , s, g(s-t)) : t,s\in \R\}
\label{rigid}\eeq
is a $1+1$-dimensional minimal surface in $\R^{1+N}$. For a timelike minimal surface $\Gamma$ of this very simple form,
it turns out that there are corresponding solutions of the nonlinear wave equation \eqref{slw} that
{\em exactly} follow $\Gamma$. Indeed, if $q:\R^k\to \R^k$ is any smooth solution of
\[
- \Delta q + (q^2-1)q = 0
\]
then writing $x\in \R^N= \R^{1+k}$ as $(x^1, x^\nu)\in \R\times \R^k$,
\beq
u(t,x) :=  q(\frac {x^\nu -g(x^1-t)}\e)
\label{rigid2}\eeq
solves \eqref{slw} in all of $\R^{1+N}$.

In particular, consider a family of surfaces $( \Gamma^\e )_{\e\in 0,1]}$ of the form
\eqref{rigid} associated with a sequence of smooth rapidly oscillating functions $(g_\e)$
converging weakly  in $H^1$,  to a limiting function $g_0$.
Although $\Gamma^\e$ converges in the Hausdorff distance to the minimal surface 
$\Gamma_0$ associated via \eqref{rigid} with the function $g_0$, one can arrange the
oscillation in such a way that  $\calT(\Gamma^\e)$ converges weakly to a
limiting measure that is {\em not} equal to $\calT(\Gamma_0)$.
(This is a simple special case of the phenomenon known in the cosmology literature as ``wiggly strings'', see again \cite{vs} Chapter 6. Related issues are also discussed in \cite{neu}.)

To illustrate this in detail, let us for simplicity assume that $k=1$ and that $g_0 = 0$.
One can check that if $u_\e$ is the solution of the form \eqref{rigid2} associated with $g_\e$, then (using notation defined in Section \ref{S:emt})
\[
\calT_\e(u_\e) = \frac 1{\e^2} q'^2 \left( 
\begin{array}{ccc}
1+ g_\e'^2 &-g_\e'^2 & g_\e'\\
g_\e'^2 & 1- g_\e'^2&g_\e' \\
-g_\e'&g_\e' &0
\end{array}
\right),\quad\quad\quad\mbox{ and }\ 
\calT(\Gamma_0) =  \left( 
\begin{array}{ccc}
1&0&0\\
0&1&0\\
0&0&0\end{array}
\right)\calH^{1+1}\rest \Gamma_0.
\]
From these it is easy to see that
unless $g_\e\to g_0 = 0$ strongly in $H^1_{loc}(\R)$,
$ \frac{\e}{\kappa_1}  \calT_\e(u_\e)$ converges to a limit that does not equal
$\calT(\Gamma_0)$. 
One can further check that this limit in general is not the energy-momentum tensor for any smooth string.

\subsection{acknowledgments}

The author is very grateful to Alberto Montero for numerous useful discussions, including 
conversations which provided the initial impetus for this project, and for carefully reading and commenting on drafts of large parts of this paper.

\section{notation and assumptions}

\subsection{general notation}

We will write $B(\rho)$  to denote an open ball of radius $\rho$ centered at the origin.

In order to emphasize the parallels between the two cases we consider, we will use the same
notation for $k=1,2$, normally without 
indicating the dependence on $k$. For example we 
will write
\beq
\dep := \begin{cases}\e/{\kappa_1}&\mbox{ when $k=1$, for $\kappa_1$ defined in \eqref{kappa.def0}}\\
(\pi |\ln\e|)^{-1}	&\mbox{ for }k=2.
\end{cases}
\label{dep.def1}\eeq
Similarly, $\calD$ and $\calD_\nu$ will have different meanings in the cases $k=1,2$, see \eqref{calD.k1def}-\eqref{calDnu.k1def} and \eqref{L1|||}-\eqref{calDnu.k2def}.

Throughout this work we  consider $1+n$-dimensional submanifolds 
in $1+N$-dimensional Minkowski space. We will always write $k = N-n$ for the codimension of the manifold.
The same number $k$ is also the dimension of the target space for the semilinear wave equation \eqref{slw}.


A parametric $(1+n)$-dimensional submanifold $\Gamma$ of $\R^{1+N}$ is a submanifold described as the image of a smooth map $H: U\to \R^{1+N}$ where $U$ is an open subset of $\R^{1+n}$. We will 
generally  assume that this map $H$ is injective.
Given a map $H$  parametrizing a surface $\Gamma$, we will often define a map 
$\psi:U\times (\mbox{small ball in }\R^k)\to \R^{1+N}$ that parametrizes a neighborhood of $\Gamma$ and agrees with $H$ on $U\times \{0\}$. In this situation, we will typically write points in $U\times \R^k\subset \R^{1+N}$ in the
form 
\beq
\mbox{$y = (\yt, \yn)$, \ \  with $\yt = (y^0,\ldots, y^n) \in U$ and $\yn = (y^{n+1},\ldots, y^N)\in \R^k$.}
\label{tn.notation}\eeq 
The superscripts stand for ``tangential'' and ``normal'' respectively. 
We will also sometimes use the alternate notation
\beq
\yn = (y^{\nu, 1},\ldots, y^{\nu, k})
\label{tn.notation2}\eeq 
for $\yn$.
We will always
arrange that $y^0$ is a timelike coordinate, and 
we will often write ${\yt}' = (y^1,\ldots y^n)$ and $y' := ({\yt}', \yn)$, so that a ``prime''
denotes spatial variables only. 

For notational consistency, we may sometimes write $\yt$  to denote a point 
$ (y^0,\ldots, y^n)\in U\subset\R^{1+n}$, even when there are no normal $\yn$ variables present.
We may also write for example $\R^k_\nu$ to denote a copy of $\R^k$ that should be thought of
as being in the normal $\yn$ variables, and
we will write
$
B_\nu(\rho) :=  \{ \yn \in \R^k_\nu \ : |\yn|<\rho\},
$
where $k$ should be clear from the context.
We will generally write $\nabla$ to denote the gradient in
spatial directions only, and $D$ to denote the spacetime gradient,
so that $D = (\partial_t,\nabla)$.
When using the notation \eqref{tn.notation},
we will similarly write $D = (D_\tau, \nabla_\nu) = (\partial_{y^0}, \nabla_\tau, \nabla_\nu)$,
where for example $\nabla_\nu = (\partial_{y^{n+1}},\ldots, \partial_{y^N})$.

We write $\eta = (\eta_{\alpha\beta}) = (\eta^{\alpha\beta})$ to denote the diagonal matrix $\mbox{diag}(-1,1,\ldots, 1)$.

We normally follow the convention that Latin indices $i,j,k$ run from $1$ to $N$ and  Greek indices $\alpha,\beta,\gamma$ run from
$0$ to $N$, and we sum over repeated upper and lower indices. When summing implicitly over the $(t,x)$ variables, we will identify $x^0$ with $t$.




%
\subsection{assumptions and notation related to timelike minimal surfaces}\label{S:tms}

A parametric submanifold is said to be {\em timelike} if  $\gamma(DH) := \det (DH^T\,  \eta \  DH) < 0$ at every point of $U$. 
The Minkowski area of a timelike parametric submanifold is defined to be
\beq
\calL(H) := \int_U \sqrt{-\gamma}
\label{Lpara.def}\eeq
A timelike submanifold  $\Gamma = \mbox{Image}(H)$ is said to be a {\em timelike minimal surface}
if $H$ is a critical point of $\calL$.
(The terminology, although standard, is misleading, as a minimal surface
$\Gamma$ is in general not a minimizer or local minimizer of $\calL$.)

Our main results all involve a
timelike minimal surface $\Gamma$ that is 
the image of a smooth, injective  map $H: (-T,T)\times \T^n \to (-T,T)\times \R^N$ of the form
\beq
H(y^0,\ldots, y^n) = \ ( {y^0}, {h(y^0,\ldots, y^n)})
\quad\quad\mbox{ for some smooth $h: (-T,T)\times \T^n \to \R^N$.} 
\label{Gamma.h1}\eeq
where $\T^n$ denotes the $n$-dimensional torus, thought of as the periodic unit cube (so that
$\calH^n(\T^n)=1$).
%
We will require that our parametrization satisfies\footnote{ Assumption \eqref{Gamma.h2} does not entail any
loss of generality. Indeed,
for $H$ of the form \eqref{Gamma.h1}, we  can always achieve \eqref{Gamma.h2} by replacing $h$ by  a function $\tilde h$
of the form $\tilde h(y_0,\ldots, y_n) = h(y_0, \Psi(y_0,\ldots, y_n))$ for a suitable $\Psi: (-T,T)\times \T^n\to (-T, T)\times\T^n$.  The suitable $\Psi$ can be found by making the ansatz $\tilde H(y) = (y_0,\tilde h(y))$
for $\tilde h$, and substituting into \eqref{Gamma.h2}. This yields an
ordinary  differential equation for $\Psi$ that we can supplement with the initial conditions
$\Psi(0, y') = y'$ and then solve by appealing to standard theory.}
\beq
H_{y_0}^T\  \eta \  H_{y_i}  = h_{y_0}\cdot h_{y_i} = 0\quad\mbox{ for }i>0
\label{Gamma.h2}\eeq
where here and throughout, we view $H$ and $h$ as column vectors. 
One can easily check that if $\Gamma$ is a timelike parametric submanifold given as the image of a map $H$ satisfying  \eqref{Gamma.h1} and \eqref{Gamma.h2}, then for any $T_1<T$, there exists some $\alpha>0$ such that
\beq
H^T_{y_0} \ \eta \ H_{y_0} = -1 + |h_{y_0}|^2 \le -\alpha,  \quad\quad 
\nabla H^T \nabla H \ge  \alpha I_n\quad\quad\mbox{ for all }\yt\in (-T_1, T_1)\times \T^n.
\label{Gamma.h3}\eeq

\subsection{energy-momentum tensors}\label{S:emt}

Among other results, we establish a relationship between the energy-momentum tensors for a
codimension $k$
timelike Minkowski minimal surface in $\R^{1+N}$ and its counterpart for the
the semilinear wave equation \eqref{slw} 
for a function  $\R^{1+N}\to \R^k$  with $0<\e \ll1$. We recall the definitions:
if $u$ solves \eqref{slw}, then $\calT_\e(u)$ is defined to be the tensor whose components are
\beq
 \calT^\alpha_{\e, \beta}(u) := 
 \delta^\alpha_\beta \left( \frac 1 2 \eta^{\gamma\delta}u_{x^\gamma}\cdot u_{x^\delta} + \frac 1{\e^2} F(u)\right) - \eta^{\alpha \gamma} u_{x^\gamma}\cdot u_{x^\beta}.
\label{emt1.def}\eeq
Here $(\eta^{\alpha\beta}) = \mbox{diag}(-1,1,\ldots, 1)$ as usual. (We deviate from convention
in taking $\calT_\e(u)$ and $\calT(\Gamma)$ to be  tensors of type $(1,1)$ rather than of type $(0,2)$; to recover the
standard definition one must lower an index.)

And if $\Gamma$ is a timelike minimal surface, then we define $\calT(\Gamma)$ to be the tensor whose components are the signed measures
\beq
\calT^\alpha_\beta(\Gamma)(A) := \int_A P^\alpha_\beta(t,x) \ d\lambda_\Gamma,
\label{emt2.def}\eeq
where $\lambda_\Gamma$ denotes the Minkowski area density of $\Gamma$, and
where $P(t,x) = (P^\alpha_\beta(t,x))$ is the matrix corresponding to Minkowski orthogonal projection onto $T_{(t,x)}\Gamma$, for $\lambda_\Gamma$ a.e. $(t,x)\in \Gamma$.
That is, if $H: U\subset\R^{1+n}\to \calU\subset \R^{1+N}$ is a smooth injective map
such that   $\Gamma = H(U)$,
then $\lambda_\Gamma$
denotes the measure on $\calU$ defined by
\[
\int_{\R^{1+N}} f(x) \  d\lambda_\Gamma \   := \  \int_U f(H(\yt)) \sqrt{-\gamma (\yt)} \ d\yt.
\]
where as before $\gamma = \det(DH^T \, \eta\, DH)$. (It is easy to check that $\lambda_\Gamma$ depends only on $\Gamma$.) And $P = P(t,x)$ is characterized by
\[
P^\alpha_\beta v^\beta = v^\alpha \quad\mbox{ for }v\in T_{(t,x)}\Gamma, 
\quad\quad
\quad\quad
P^\alpha_\beta w^\beta = 0 \quad\mbox{ if }w^T\eta v = 0\mbox{ for all }v\in T_{(t,x)}\Gamma. 
\]
For both models, the energy-momentum tensor may be obtained by considering variations
of the relevant action functional with respect to suitable one-parameter families of
diffeomorphisms. We recall this in some detail for $\calT(\Gamma)$, as we will need to refer to this later:
%

\begin{lemma}
Suppose that $H: U\subset\R^{1+n}\to \calU\subset \R^{1+N}$ is a smooth injective map
whose image $\Gamma := H(U)$ is a timelike surface.
Given $\tau\in C^\infty_c(\calU; \R^{1+N})$, define
$\Phi_\sigma(x) := x +\sigma \tau(x)$. 
Then
\beq
\left. \frac d{d\sigma} \calL(  \Phi_\sigma\circ H) \right|_{\sigma=0}
\ = \ 
 \  \int_\calU
\tau^\beta_{x^\alpha}(x) \, \ P^\alpha_\beta \, d\lambda_\Gamma
 \  = \ 
 \  \int_\calU
\tau^\beta_{x^\alpha}(x) \, \ d\calT^\alpha_\beta(\Gamma).
\label{fv}\eeq
\label{EMT2}\end{lemma}

Note that \eqref{fv} exactly parallels the well-known first variation formula in the Euclidean case, in which $\lambda_\Gamma$
is replaced by the restriction to $\Gamma$ of Hausdorff measure of the suitable dimension,  and $P^\alpha_\beta$
is replaced by orthogonal projection with respect to the Euclidean inner product.

Exactly parallel to \eqref{fv}, $\calT_\e(u)$ arises from domain variations of the action functional, say $\calA_\e$, whose Euler-Lagrange equation is \eqref{slw}, see for example \cite{ss} for the proof.
Thus the results \eqref{t1.c3}, \eqref{t2.c1} assert that the first variation of $\calA_\e$ (with respect to domain variations) at the critical point $u$ is close (in a weak
topology, and after suitable rescaling) to the first variation of $\calL$ at the associated timelike minimal surface $\Gamma$. 

We present the standard calculation that leads to \eqref{fv}, since we will need it later:

\begin{proof}[Proof of Lemma \ref{EMT2}] 
We will write $H_\sigma := \Phi_\sigma\circ H$, and 
\[
\gamma_{\sigma, ab} = H_{\sigma, y^a}^T \eta H_{\sigma,y^b} = H^\alpha_{\sigma,y^a} \eta_{\alpha\beta}H^\beta_{\sigma,y^b},
\quad\quad
(\gamma_\sigma^{ab}) = (\gamma_{\sigma, ab})^{-1}
\quad\quad
\gamma_\sigma = \det  (\gamma_{\sigma,ab}),
\]
where indices $a,b$ run from $0$ to $n$ and $\alpha,\beta$ as usual run from $0$ to $N$.
Using the fact that $\frac d{d\sigma} \gamma_\sigma = \gamma_\sigma  \gamma^{ab}_\sigma\frac d{d\sigma}\gamma_{\sigma, ab}$ 
we find that
\begin{align*}
\left. \frac d{d\sigma} \calL( H_\sigma) \right|_{\sigma=0} \ 
\ = \ 
\left. \frac d{d\sigma} \int_U \sqrt{-\gamma_\sigma} \right|_{\sigma=0} \ 
&= 
 \  \int_U
 (\tau^\beta \circ H) _{y^a} \eta_{\beta \delta} H^\delta_{y^b} \gamma^{ab} \sqrt{-\gamma}\ d\yt \\
&= 
 \  \int_U
(\tau^\beta_{x^\alpha}\circ H)  H^\alpha_{y^a} \eta_{\beta \delta} H^\delta_{y^b} \gamma^{ab} \sqrt{-\gamma} \ d\yt \\
&=
\int_\calU \tau^\beta_{x^\alpha}(t,x) P^\alpha_\beta(t,x)\ d \lambda_\Gamma
\end{align*}
where
\[
P^\alpha_\beta(H(\yt)) := H^\alpha_{y^a}(\yt) \gamma^{ab}(\yt) H^\delta_{y^b}(\yt) \eta_{\delta \beta}.
\]
Note that $P^\alpha_\beta$ is defined for $\lambda_\Gamma$ a.e. $(t,x)$, so the above integral makes sense.
In order to complete the proof, we must check that $P^\alpha_\beta(t,x)$ is the orthogonal projection
onto $T_{(t,x)}\Gamma$. To see this, first note that at any $\yt\in \R^{1+n}$,
\[
(P H_{y^c})^\alpha = P^\alpha_\beta H^\beta_{y^c} =  
H^\alpha_{y^a}\gamma^{ab} H^\delta_{y^b} \eta_{\delta \beta} H^\beta_{y^c}  = 
H^\alpha_{y^a}\gamma^{ab}\gamma_{bc} = H^\alpha_{y^a} \delta^a_c = H^\alpha _{y^c}.
\]
Thus $P H_{y^c} = H_{y^c}$.
And if $v$ is orthogonal to $H_{y^b}$ for all $b$, then
\[
(P v)^\alpha = P^\alpha_\beta v^\beta  = 
H^\alpha_{y^a}\gamma^{ab} H^\delta_{y^b} \eta_{\delta \beta} v^\beta  = 0
\]
since the orthogonality of $v$ means exactly that $H^\delta_{y^b} \eta_{\delta \beta} v^\beta= 0$
for every $b$.
Since $T_{(t,x)}\Gamma$  at $(t,x)=H(\yt)$ is spanned by $\{ H_{y^b}(\yt) \}_{b=0}^n$, the above calculations exactly state that
$P(t,x)$ is the matrix corresponding to orthogonal projection onto $T_{(t,x)}\Gamma$.
\end{proof}

\subsection{change of variables}\label{S:cv}

We next define the change of variables that, as mentioned earlier, is the starting point of our
argument. We will use the notation \eqref{tn.notation}.

We assume as always that $\Gamma$ is a smooth timelike minimal surface,
given as the image\footnote{All the results of this section are local, so the topology of $\Gamma$,
that is, the fact that $H$ is defined on $(-T,T)\times \T^n$, is irrelevant here. But it is convenient
to keep the same set-up as in the rest of the paper.}  of a smooth injective map $H:(-T,T)\times\T^n\to \R^{1+N}$ satisfying \eqref{Gamma.h1}, \eqref{Gamma.h3}. For this section, we allow $k = N-n$ to be an arbitrary positive integer, since all the proofs for $k=2$ apply without change to $k\ge 3$. (The case $k=1$ is simpler.)
Although we do not use them in this paper, the results for $k\ge 3$ may be useful
for problems such as the dynamics of defects in certain nonabelian gauge theories.

First, we  fix smooth maps
$\bar \nu_i :(-T,T)\times \T^n\to \R^{1+N}$ for $i=1,\ldots, k$  such that
\beq
\bar \nu_i^T \, \eta \ \bar \nu_j = \delta_{ij},
\quad\quad
\quad
H_{y^\alpha}^T\  \eta \  \bar \nu_i = 0 \quad\mbox{ in $(-T,T)\times \T^n\to \R^{1+N}$}
\label{nu1}\eeq
for all $i,j\in \{1,\ldots, k\}$ and $\alpha\in \{0,\ldots, n\}$. (Here and throughout the paper, we are thinking of $\bar \nu_i$ as a column vector.) This states that $\{ \bar \nu_1(\yt),\ldots, \bar \nu_k(\yt)\}$ form
an orthonormal  basis  for the normal space to $\Gamma$ at $H(\yt)$, where words like ``normal'' and ``orthonormal'' are
understood with respect to the Minkowski inner product and $\yt$ denotes a generic pont in $(-T,T)\times \T^n$. Note that when $k=1$, \eqref{nu1} determines $\bar \nu_1$ up to a sign,
whereas for $k\ge 2$ there are rotational degrees of freedom that we have not specified (and will not specify). 

Next, we define (using the notation \eqref{tn.notation})
\beq
\psi(y) \ := \ H(\yt) + \sum_{i=1}^k \bar \nu_i(\yt)\,  y^{n+i} .
\label{psi.def}\eeq
It is clear that $\psi(\yt,0) = H(\yt)$ for all $\yt\in (-T,T)\times \T^n$. 

Recall that the statement of Theorems \ref{T1}, \ref{T2} involve a number $T_0<T$. We henceforth fix $T_1\in (T_0,T)$, and we let $\rho_0>0$ be so small that 
\beq
\psi \left( \{ -T_1\}\times \T^n\times B_\nu(\rho_0) \right) \subset\subset (-T, -T_0)\times \R^N,
\quad\quad
\psi \left( \{ T_1\}\times \T^n\times B_\nu(\rho_0) \subset (T_0,T)\right) \times \R^N,
\label{rho0.1}\eeq
and
\beq
\mbox{$\psi$ is injective, with smooth inverse $\phi$, on $(-T_1, T_1)\times \T^n \times B_\nu(\rho_0) $.}
\label{T1rho}\eeq
The latter condition can be satisfied due to the inverse function theorem, as we will check below that
$D\psi(\yt,0)$ is invertible for $\yt\in (-T_1,T_1)\times \T^n$. We next define
\beq
(g_{\alpha\beta})_{\alpha,\beta=0}^N = G := D\psi^T \ \eta \ D \psi
\label{G.def}\eeq
so that $G$ represents the Minkowski metric in the $y$ coordinates. 
We further define
$g := \det G$
and 
$
(g^{\alpha\beta})_{\alpha,\beta=0}^N := G^{-1}$,
and we finally define  $(a^{\alpha\beta})_{\alpha,\beta=0}^N$ by
\beq
a^{ij} = g^{ij}\ \mbox{ if }i,j\ge 1,\quad
a^{00} = - g^{00},\quad
a^{i0} = a^{0j} = 0\ \mbox{ for }i,j = 1,\ldots, N.
\label{A.def}\eeq
When we write \eqref{slw} in terms of the $y$ coordinates as in \eqref{v.eqn0}, $(g^{\alpha\beta})$ and $g$ appear in
the coefficients, and $(a^{\alpha\beta})$ appears in a natural associated energy density
$e_\e(v) = e_\e(v;G)$, defined in \eqref{eep.def1}.
We summarize properties of $g$ and $(g^{\alpha\beta})$ that we will use:

\begin{proposition} 
Let $\psi, g, (g^{\alpha\beta})$ be the functions on $ (-T_1,T_1)\times \T^n\times B_\nu(\rho_0)$  defined above.
Then, after
taking $\rho_0$ smaller if necessary, there  exist positive constants $c\le C$   such that
\begin{equation}
\| g^{\alpha\beta}\|_{W^{1,\infty}} \le C
\quad\quad\quad
g^{\alpha\beta}_{y^0}\xi_\alpha\xi_\beta \le C(|\xi_\tau|^2 + |\yn|^2\, |\xi_\nu|^2)
\label{coeffs1}
\end{equation}
\begin{equation}
\frac{\partial_{y^\alpha}\sqrt{-g}}{\sqrt{-g}} g^{\alpha\beta} \xi_\beta\xi_0 \le
C(|\xi_\tau|^2 + |\yn|^2\, |\xi_\nu|^2),
\label{coeffs2}
\end{equation}
\begin{equation}
|g^{\alpha\beta} \xi_\beta|  \le  C(|\xi_\tau| + |\yn| |\xi_\nu|) \quad\quad\mbox{ if }\alpha\le n,
\label{coeffs3}
\end{equation}
and
\beq
c |\xi_\tau|^2 + (1-C |\yn|^2)|\xi_\nu|^2
\  \le  \
a^{\alpha\beta}(y) \xi_\alpha \xi_\beta   
\ \le  \ \  C |\xi_\tau|^2 + (1+C |\yn|^2) |\xi_\nu|^2
\label{pos1}\eeq
for all $y = (\yt, \yn)\in (-T_1,T_1)\times \T^n\times B_\nu(\rho_0)$ and $\xi = (\xi_\tau, \xi_\nu)\in \R^{1+N} \cong \R^{1+n}\times \R^k$.
In addition,
\beq
\psi^0_{y^0} \ge c \quad\mbox{ in }(-T_1,T_1)\times \T^n\times B_\nu(\rho_0).
\label{ps0lbd}\eeq
\label{P.transformation}\end{proposition}

We will use the notation
\beq
\calN :=\psi\left( (-T_1,T_1)\times \T^n\times B_\nu(\rho_0)\right) \cap [ (-T_0,T_0)\times \R^N]
\label{calN.def}\eeq

For future use, it is convenient to fix a constant $\kappa_2\ge 1$ such that
\beq
(1+\kappa_2|\yn|^2) e_\e(v) \ge     \frac \lambda 2  |D_\tau v|^2 + (1+|\yn|^2)e_{\e,\nu}(v) 
\label{omega3.def}\eeq
everywhere in $(-T_1,T_1)\times \T^n\times B_\nu(\rho_0)$, for all $v\in H^1$,
where $\e_{\e,\nu}$ was defined in \eqref{eenu.def}.
This is possible due to \eqref{pos1}.

When $\Gamma$ is a hypersurface, we have slightly better behavior:

\begin{proposition} 
Suppose that $k=1$, and let $\psi, g, (g^{\alpha\beta})$ be as   defined above. 
Then, after
taking $\rho_0$ smaller if necessary,
\beq
g^{\alpha N} = g^{N\alpha} = \begin{cases}1&\mbox{ if }\alpha=N\\0&\mbox{ if not}.\end{cases}
\label{k1coeffs}\eeq
\beq
\lambda|\xi_\tau|^2 + |\xi_\nu|^2
\  \le  \ 
a^{\alpha\beta}(y) \xi_\alpha \xi_\beta   
\ \le  \ \ \Lambda|\xi_\tau|^2 +  |\xi_\nu|^2.
\label{pos1a}\eeq
everywhere in $(-T_1,T_1)\times \T^n\times B_\nu(\rho_0)$.
\label{k1.trans}
\end{proposition}

Conclusion \eqref{pos1a} is not essential but will allow us to simplify our notation, for example
by taking $\kappa_2=1$ in \eqref{omega3.def} and everywhere else that this constant occurs (for $k=1$).

We defer the proofs of Propositions \ref{P.transformation} and \ref{k1.trans} to an Appendix,
see Section \ref{S:geom.lemmas}.

For a solution $u:\R^{1+N}\to \R^k$  of \eqref{slw}, we will define $v: (-T_1,T_1)\times \T^n \times B_\nu(\rho_0) \to \R^k$
by $v = u \circ \psi$.  Then $v$ satisfies
\begin{equation}
\Box_{G} v
+ \frac 1{\e^2} f(v)=0
\label{v.eqn}\end{equation}
on its domain.
Here 
\[
\Box_G v = - 
\frac 1{\sqrt{-g}} \partial_{y^\alpha}( \sqrt{-g} g^{\alpha\beta} \partial_{y^\beta} v).
\]
As noted earlier, we find it convenient to write  \eqref{v.eqn} in the form \eqref{v.eqn0}.
We now derive a key differential inequality for the energy density $e_\e(v)$ from \eqref{eep.def1}.

\begin{lemma}
Suppose that $v:(-T_1,T_1)\times \T^n\times B_\nu(\rho_0)\to \R^k$ is a smooth solution of \eqref{v.eqn}, with coefficients satsfiying \eqref{coeffs1}.
Then 
\beq
\frac \partial {\partial y^0} e_\e(v) 
\le 
C (|D_{\tau }v|^2  + |\yn|^2 \ |\nabla_\nu v|^2)
+ \nabla \cdot \vp
\label{eep.prime1}\eeq
with
\begin{equation}
\vp := (\vp^1,\ldots, \vp^N),
\quad\quad
\quad
\vp^i := 
\ g^{i\alpha }  v_{y^\alpha} \cdot v_{y^0}.
\label{vp.def}\end{equation}
\label{L.eflux}\end{lemma}

\begin{proof}
Multiply \eqref{v.eqn0} by $v_{y^0}$ and rewrite to find that
\[
- \partial_{y^\alpha}( g^{\alpha\beta} v_{y^\beta} \cdot v_{y^0})
+ g^{\alpha\beta} v_{y^\beta}  \cdot v_{y^0 y^\alpha}
+\frac 1{\e^2}F(v)_{y^0} = 
- ( b \cdot D v ) \cdot v_{y^0}.
\]
We rewrite 
$g^{\alpha\beta} v_{y^\beta}\cdot  v_{y^0 y^\alpha}$
as $ \frac 12 \partial_{y^0}( g^{\alpha\beta} v_{y^\beta} \cdot v_{y^\alpha})
-\frac 12 g^{\alpha\beta}_{y^0} v_{y^\beta}\cdot  v_{ y^\alpha}$.
Gathering all the terms of the form $\partial_{y^0}[\ldots]$ on the left-hand side,
we find that
\[
\partial_{y^0} \left[- g^{0\beta} v_{y^\beta}\cdot  v_{y^0} + \frac 12 g^{\alpha\beta}v_{y^\alpha}\cdot v_{y^\beta}
+ \frac 1{\e^2}F(v) \right] = 
 \partial_{y^i}( g^{i \beta} v_{y^\beta}\cdot  v_{y^0}) 
-( b \cdot D v ) v_{y^0} + \frac 1 2 g^{\alpha\beta}_{y^0} v_{y^\beta}\cdot  v_{ y^\alpha}.
\]
The definition \eqref{A.def} of $a^{\alpha\beta}$ implies that  left-hand side is just $ \partial_{y^0} e_\e(v)$.
To complete the proof, we use 
\eqref{coeffs1} and \eqref{coeffs2} to check that the non-divergence  terms on the right-hand side
are bounded by
$C (|D_{\tau }v|^2  + (\yn)^2 |\nabla_\nu v|^2)$.
\end{proof}

As an easy consequence of
Proposition \ref{k1.trans}, we obtain a quite explicit description of the
signed Minkowski distance function 
defined by the  eikonal equation \eqref{d1a} in the case $k=1$.

\begin{corollary}
Assume that $k=1$ and define $\psi$ as above, and let $\phi = (\phi^0,\ldots, \phi^N)$ denote the inverse of 
$\psi$. Then $\phi^N$ solves the eikonal equation \eqref{d1a} on Image$(\psi)$.
\label{L.d}\end{corollary}

In particular, the Corollary shows that it makes sense to speak of the signed distance function in the set $\calN$ defined in \eqref{calN.def}.

\begin{proof}
Fix a point in the image of $\psi$, say $(t,  x) = \psi( y)$.  Then since $\eta = \eta^{-1}$,
\[
(g^{\alpha\beta})(y) =   [D\psi ^T(y) \ \eta \ D\psi(y)]^{-1} = (D\psi)^{-1}(y) \ \eta\  (D\psi)^{-T}(y)
= D\phi(t,x)\ \eta\  D\phi^T(t,x) .
\]
Thus, according to \eqref{k1coeffs}, 
\[
1 = g^{NN}(y) = -(\phi^N_t)^2 + |\nabla \phi^N|^2,
\]
so that \eqref{d1a} holds.
And it is clear that $\phi^N(t,x) = 0$  for $(t,x)\in \Gamma$.
\end{proof}

In fact the curves $s\mapsto H(\yt) + s \nu(\yt) = \psi(\yt,s)$ are exactly characteristic curves for the eikonal equation \eqref{d1a}.

The eikonal equation \eqref{d1a} determines the distance function $d$ only up
to a sign; we will always choose to identify $d$ with $\phi^N$ (so that our choice of a
sign is ultimately determined by our choice of the sign for the unit normal $\nu$.)
Then it follows that
\beq
d(\psi(y)) = y^N\quad\mbox{ for }y\in (-T_1,T_1)\times \T^n\times B_\nu(\rho_0),
\label{d.inv}\eeq

\subsection{initial data}

In this section we describe our general assumptions on the initial data. 

We will eventually combine estimates  for $v= u\circ \psi$ on $(-T_1,T_1)\times \T^n\times B_\nu(\rho_0)$,
which we use control to the behavior of $u$ near $\Gamma$, with standard energy estimates for
\eqref{slw} away from $\Gamma$. We start by making a number of smallness assumptions, in
all of which a parameter $\zeta_0$ appears. We will prove below that one can find data for which
$\zeta_0 \approx \e^2$ when $k=1$, and $\zeta_0 \approx |\ln \e|^{-1}$ when $k=2$.
Although we omit the proof, it is in fact true that one cannot find data satisfying our assumptions
with $\zeta_0 \ll \e^2$ (for $k=1$) or $\zeta_0 \ll |\ln\e|^{-1}$. We therefore will assume that
\beq
\zeta_0 \ge \e^{2}\quad \mbox{ if }k=1,
\quad\quad\quad\quad
\zeta_0 \ge |\ln \e|^{-1}\quad \mbox{ if }k=2.
\label{z0gdep}\eeq
This is convenient, as it will enable us to absorb small error terms into
expressions of the form $C\zeta_0$.

Our first assumption is that the energy is small away from $\Gamma_0$:
\beq
\dep
\left. \int_{ \{ x\in \R^N : (0,x)\not\in \mbox{\scriptsize{\ image}}(\psi)\} } e_\e(u) dx\right|_{t=0} \le  \zeta_0
\label{idata1}\eeq
where $e_\e(u) = e_\e(u;\eta)$ is defined in \eqref{eep.def} and $\dep = \dep(k)$ is defined in \eqref{dep.def1}.

Near $\Gamma_0$, it is convenient to state our assumptions in terms of $v = u \circ \psi$.
Note that initial data for $u$ at $t=0$ corresponds to data for $v$ on a  hypersurface
that does {\em not} in general have the form $\{ y_0 = \mbox{const} \}$. 
This hypersurface is described in the following

\begin{lemma}
There exists a Lipschitz function $b: \T^n\times B_\nu(\rho_0)\to \R$ such that
for $y = (y_0, y')  \in (-T_1,T_1)\times  \T^n\times B_\nu(\rho_0)$, 
\beq
\psi(y_0, y') \in \{0\}\times \R^N \mbox{ if and only if } y_0 = b(y').
\label{b.def0}\eeq
Moreover, $\|\nabla b\|_\infty \le C$.
\label{b.def}
\end{lemma}

\begin{proof}
Fix $y'\in \T^n\times B_\nu(\rho_0)$, and for  $s\in (-T_1,T_1)$, let $y(s) := (s, y')$ and let $X(s) = \psi(y(s))\in \R^{1+N}$. 
To prove that $\psi^{-1}(\{0\}\times \R^N)$ is the graph of a function, we need to show that $y(s)$
intersects  $\psi^{-1}(\{0\}\times \R^N)$ exactly once, or equivalently, that $X(s)$ intersects
$\{0\}\times \R^N$ for exactly one value of $s$.
To prove this, note that the definition of $G$ and \eqref{Gamma.h3} imply that, after taking $\rho_0$ smaller if necessary, 
\[
X'(s)^T \eta X'(s) = y'(s)^T G(y(s)) y'(s) =  g_{00}(y(s)) < 0
\]
for every $s$. Thus $s\mapsto X(s) = (X^0(s), X'(s))$ is a timelike curve, from which the
claim is obvious.
It follows that there exists a function $b$ satisfying \eqref{b.def0}. Then
by differentiating the identity $\psi^0(b(y'), y') = 0$, we find that 
$\psi^0_{y^0}(b(y'),y')\nabla b(y') + \nabla\psi^0(b(y'),y') = 0$. We know from \eqref{ps0lbd} that $\psi^0_{y^0}$ is bounded away from $0$,
and this together with the smoothness of $\psi^0$ implies that $ \| \nabla b \|_\infty \le C$.
\end{proof}

Using the lemma, we define 
\beq
\mbox{
$v_0(y') := v(b(y'), y')$ for $y'\in \T^n\times B_\nu(\rho_0)$.}
\label{v0.def}\eeq
Our next assumptions specify that the energy near $\Gamma_0$ is small, in the frame that moves
with $\Gamma$:
\beq
\dep \int_{\T^n\times B_\nu(\rho_0)} (1+ \kappa_2 |\yn|^2) e_\e(v_0;G) d y' -1\  \le \  \zeta_0,
\label{idata2}\eeq
\beq
\dep\int_{\T^n\times B_\nu(\rho_0)} \  \   \,( |v_{y^0}|^2 + |v_{y^0}| \ |\nabla_\nu v_0|)(b(y'), y') \, dy' \ \le  \ \zeta_0.
\label{idata3}\eeq

Finally,  using notation discussed in the Introduction and defined in  \eqref{calD.k1def} for $k=1$ and   \eqref{L1|||},  \eqref{|||.def} for the case $k=2$, we require that
\beq
\calD(v_0; \rho_0) \le \zeta_0.
\label{idata4}\eeq
This specifies  that the initial profile possesses a defect --- that is, an interface or vortex --- near $\Gamma_0$. 

Note that conditions \eqref{idata1}, \eqref{idata2}-\eqref{idata4} are always satisfied if we define $\zeta_0$ to be the
maximum of the left-hand sides of these inequalities. The smallest possible values of $\zeta_0$ depend on
$k$ and, as mentioned earlier, account for the fact that our conclusions for $k=1$ are stronger than for
$k=2$.

\begin{lemma}
In the scalar ($k=1$) case, there exists initial data $(u, u_t)|_{t=0} \in \dot H^1\times L^2(\R^N)$
for \eqref{slw}
satisfying conditions \eqref{idata1} -- \eqref{idata4} with $\zeta_0 = C \e^2$, and such that
\beq
\int_{\calN_0} \left(u(0,x) - q(\frac{d(0,x)}\e)\right)^2 \le C \e,\ \  \mbox{ where }\ \calN_0 = \{ x\in \R^N : (0,x)\in \calN\}.
\label{idata5}\eeq
And  in the vector ($k=2$) case,
there exists initial data $(u, u_t)|_{t=0}\in \dot H^1\times L^2(\R^N;\R^2)$ for \eqref{slw}
satisfying conditions \eqref{idata1} -- \eqref{idata4} with $\zeta_0 = C |\ln\e|^{-1}$.
\label{L.z0}\end{lemma}

Although we do not prove it, these scalings for $\zeta_0$ are in fact optimal.

\begin{proof} 
In both cases $k=1,2$, we define a function $U$ in $\mbox{Image}(\psi)$ such that
\beq
U\circ \psi = \tilde q(\frac \yn \e)
\label{U.def}\eeq 
where $\tilde q:\R^k\to \R^k$ is a nearly-optimal profile. We then 
require that 
\beq
\mbox{$u(0,x) = U(0,x)$ and $u_t(0,x) = U_t(0,x)$ in $\calN_0$}
\label{u0nearG0}\eeq
and we verify \eqref{idata2} - \eqref{idata4}. (Note that \eqref{d.inv} then implies that
$u(x,0) = \tilde q(\frac d \e)$ when $k=1$, which will make \eqref{idata5} obvious.)
Finally, we argue that  $u(0,\cdot)$ can be
extended to $\R^N\setminus \calN_0$ such that \eqref{idata1} holds.

{\bf k=1}: 
By integrating the equation \eqref{q.def}
solved by $q$ and using the boundary conditions at $\pm \infty$, one finds that $ q' = \sqrt{2 F(q)}$, and hence that
\beq
\int_\R \frac 12 q'^2 + F(q) dx \ = \ \int_\R \sqrt {2 F(q)} q'(s) ds =  \int_{-1}^1 \sqrt{2 F(s)} \ ds \ = \ \kappa_1.
\label{qlb}\eeq
Using  \eqref{q.def} and \eqref{scalarF}, standard ODE arguments
show that 
\[
|q'(s)| + |q(s) - \sign(s)| \le C e^{-c|s|} \quad\mbox{ for all }s.
\]
for suitable constants. It follows that, given $\e>0$, we can find a function $\tilde q$
such that $\tilde q(\frac s \e) = q(\frac s \e)$ if $|s|< \frac 12 \rho_0$, and
\[
\tilde q(\frac s \e) = q(\frac s \e)\mbox{ if }|s|< \frac 13 \rho_0, 
\quad\quad
\tilde q(\frac s \e) = \sign(s)\mbox{ if }|s|> \frac 23 \rho_0,
\quad\quad
\|\tilde q - q\|_{W^{1,\infty}} \le C e^{-c/\e}
\]
and
\beq
\kappa_1 < \int_{-\rho_0/ \e}^{\rho_0/ \e} \frac 12 \tilde q'^2 + F(\tilde q) dx \le \kappa_1 + C e^{-c/\e}.
\label{tqlb}\eeq
Now define $U$ as in \eqref{U.def}, and define $u|_{t=0}$ near $\Gamma_0$ by \eqref{u0nearG0}.
Then by construction $v_0$ as defined in \eqref{v0.def} is given by $v_0(y) = \tilde q(y^N/\e)$,
and $v_{y_0} = 0$. The latter fact immediately implies that \eqref{idata3} holds, and 
\eqref{idata1}, \eqref{idata3} are easily verified. For example, the explicit form of $v_0$
and  \eqref{pos1a} imply that 
$e_\e(v_0;G) = \frac 12 \tilde q'^2(\frac \yn \e ) + \e^{-2}F(\tilde q(\frac \yn \e))$.
Then recalling that $\dep = \frac \e  {\kappa_1} $, we infer from \eqref{tqlb} and a change of
variables $y^N/\e \mapsto y^N$  that
\begin{align*}
\dep \int_{\T^n\times B_\nu(\rho_0)} (1+ \kappa_2 |\yn|^2) e_\e(v_0;G) d y' -1\  
&\le
C\e^2 \int_{-\rho_0/\e}^{\rho_0/\e} [\frac{ \tilde q'^2}2 + F(\tilde q)] ( y^N)^2 dy^N + C e^{-c/\e}.
\end{align*}
The exponential decay of $q$ implies that $\int_R [\frac 12 \tilde q'^2 + F(\tilde q)]  (y^N)^2 dy^N \le C$
independent of $\e$, and \eqref{idata2} follows with $\zeta_0 = C\e^2$. The verifications of \eqref{idata3} and \eqref{idata4} are similar and a little easier.

Finally, on $\R^N \setminus \calN_0$, we set $u_t(0,\cdot)\equiv 0$, and we require that $u(0,\cdot) = \pm 1$ and that $u$ is continuous (hence smooth)
across $\partial \calN_0$. This can be done, since $\R^N\setminus \Gamma_0$ consists of two components,
one of which meets $\calN_0$ where $d = \rho_0$ (and hence $u=1$), and the other  where $d = -\rho_0$.  (Here we have used the fact that $\rho_0$ is sufficiently small, see \eqref{rho0.1}.)

{\bf k=2}:
In this case we may define 
$\tilde q(s) = s \min\{  1 ,\frac 1{|s|}\}$ for $s\in \R^2$, and we go on to
make the definitions \eqref{U.def}, \eqref{u0nearG0} as above, so that $v_0(y) = \tilde q(\yn/\e)$. 
Then an easy calculation shows that
\[
\frac 1{\pi |\ln \e|} \int_{B_\nu(\rho_0/\e)} \frac 12 |\nabla\tilde q|^2 + F(\tilde q) ds \le 1 + C |\ln\e|^{-1}.
\]
This plays a role analogous to \eqref{tqlb} above and allows us to verify along the above lines (but using \eqref{pos1} in place of \eqref{pos1a}) that \eqref{idata2} holds with $\zeta_0 = C|\ln \e|^{-1}$. As above, \eqref{idata3} follows from the fact that $v_{y^0}(b(y'), y')=0$ in $\T^n\times B_\nu(\rho_0)$. 
One can check \eqref{idata4} directly from the definitions (see Section \ref{S:vector_ee}), noting that $J_\nu v_0( \yt, \yn) = \begin{cases}
\e^{-2}&\mbox{ if }|\yn| < \e,\\
0&\mbox{ if }|\yn|>\e .\end{cases}$

It remains to show that $u_0 = U(0,\cdot)$, as defined in $\calN_0$ by \eqref{u0nearG0}, can be extended
to a function in $H^1(\R^N)$ satisfying \eqref{idata1}. It is clear that we can extend $u_0$ by a finite-energy map in a neighborhood $\calV$ of $\calN_0$. Next we point out that since $\Gamma_0$
is a smooth, compact, oriented codimension 2 submanifold without boundary of $\R^N$,
results in \cite{abo0}
imply that we may find a function $w\in H^1_{loc}(\R^N\setminus \Gamma;\C)$ with $\int_{\R^N\setminus \calN_0} |\nabla w|^2 <\infty$, such that
$|w|=1$ a.e., and in addition such that $Jw = J(\frac{u_0}{|u_0|})$ in $\Gamma_0$, where $J(\cdots)$ denotes
the distributional Jacobian of $(\cdots)$. This implies that there exists a real-valued function
$\theta \in H^1_{loc}(\calV\setminus \Gamma_0; \R)$ such that  $u_0 = |u_0| w e^{i\theta}$
in $\calV$. Thus we define $u(0,\cdot)$ globally in $\R^N$ by setting 
\[
u(0,\cdot) =\begin{cases}
| u_0| w e^{i \chi \theta} &\mbox{ in }\calV\\
w &\mbox{ in }\R^N\setminus \calV
\end{cases}
\]
where $\chi\in C^\infty_c(\calV)$ and $\chi\equiv 1$ in $\calN_0$.  We may set
$u_t(0,x) = 0$ outside of $\calN_0$.
\end{proof}

%


%

\section{basic  energy estimates, $k=1$ }\label{S:energy1}

The main result of this section, Proposition \ref{localest} below,  
contains the simplest case of our main estimate.

In this section and the next, we restrict our attention to the case $k=1$, so that\footnote{Although here
there is not much point in writing $\yn$ and $\nabla_\nu$ instead of $y^N$ and $\partial_N$, this notation  will prove useful when we consider the vector case, and we  use it here to emphasize the parallels.} 
$N = n+1$,  $\yn = y^N\in \R$, and $\nabla_\nu = \partial_N$.
Thus in this section, $B_\nu(\rho)$
denotes the interval  $(-\rho, \rho)$ along the $y^N$ axis. We also follow other conventions for $k=1$, so that for example $\dep = \frac \e {\kappa_1} $, see \eqref{dep.def1}.

Throughout this section, we let $\psi$ denote the change of variables from Section
\ref{S:cv}, in
the case $k=1$.
We also use the notation $g, g_{\alpha\beta}, g^{\alpha\beta}$ etc from the previous section.

In the Introduction we discussed a ``defect confinement'' functional $\calD$.
In the case $k=1$ we define it to be
\beq
\calD(v; \rho) := \int_{\T^n\times B_\nu(\rho)} |\yn| \ |v - \sign(\yn)|^2 dy'
\label{calD.k1def}\eeq
for $v:\T^n\times B_\nu(\rho) \to \R$. We will also write
\beq
\calD(v; \rho) = \int_{\T^n} \calD_\nu(v(\yt'); \rho) d\yt'.
\label{intDnu}\eeq
where $v(\yt')(\yn) = v(\yt', \yn)$ and
\beq
\calD_\nu(w; \rho) := \int_{B_\nu(\rho)} |\yn| \ |w - \sign(\yn)|^2 d\yn
\quad\quad
\mbox{ for }
w: B_\nu(\rho)\to \R.
\label{calDnu.k1def}\eeq

Let $c_*$ be a constant
such that 
\beq
| g^{N\alpha}(y) \xi_\alpha \xi_0 |  \ =\  | a^{N\alpha}(y) \xi_\alpha \xi_0| \ \le \  \frac 12 c_*\ a^{\alpha\beta} \xi_\alpha \xi_\beta
\label{cstar.def}\eeq
$\mbox{ for all }\xi \in \R^{1+N}\mbox{ and }y\in (-T_1,T_1)\times \T^n\times B_\nu(\rho_0)$.


The main result of this section is

\begin{proposition}
Let $v:(-T_1,T_1)\times \T^n\times B_\nu(\rho_0)\to \R$ satisfy \eqref{v.eqn}, where $f = F'$ and $F$ satisfies
\eqref{scalarF}. 
Recalling that $\dep = \frac \e {\kappa_1}$, where
$\kappa_1$ is defined in \eqref{kappa.def0}, assume that there exist
some $s_1\in (-T_1,T_1)$, $\rho_1\in (0,\rho_0)$, and $\zeta_0 \ge \e^2$ such that
\begin{align}
\dep\int_{\{s_1\}\times \T^n\times B_\nu(\rho_1)} (1 + (\yn)^2)\, e_\e(v) dy' - 1 \ 
&\le \zeta_0
\label{Ploc.h1}\\
\calD(v(s_1), \rho_1/2)
&\le \zeta_0.
\label{Ploc.h2}\end{align}
Then there exists a constant $C$, independent of $v$ and of $\e\in (0,1]$, such that
\[   
\dep
\int_{  \{s_1+s \}\times \T^n\times B_\nu(\rho_1 - c_* s) }
|D_\tau v|^2 +  (\yn)^2 \left[    |\nabla_\nu v|^2 + \frac 1{\e^2}F(v) \right] \ dy' 
 \le  C \zeta_0
\]  
\[   
\dep
\int_{  \{s_1+s \}\times \T^n\times B_\nu(\rho_1 - c_* s) }
(1+(\yn)^2) e_\e(v) \ 
 dy' \ - \ 1 \ \le C \zeta_0
\]  
and
\[   
\calD(s_1+s; \rho_1/2)
\le C \zeta_0
\]  
for all $s\in [0, \rho_1/ 2c_*]$ such that $s_1+s < T_1$.
\label{localest}\end{proposition}

Our first Lemma will be needed to establish requirement \eqref{calD0} as discussed in the Introduction.
In the statement and proof we take all the $\yt$ variables to be frozen, and we consider a function $v$ of a single real variable $\yn$.

\begin{lemma}
Let $B_\nu(\rho) := (-\rho,\rho)\subset \R_\nu$ be an interval as above. Then there exists a constant $\kappa_3 = \kappa_3(\rho)$ such that if $v\in H^1(B_\nu(\rho))$ and 
if
\beq
\calD_\nu(v; \rho) \ \le  \  \kappa_3
\label{wL2}\eeq
then
\beq
\dep \int_{B_\nu(\rho)}   e_{\e,\nu}(v) \  d\yn \ge 1 - C e^{-C/\e}.
\label{mm2c1}\eeq
Moreover, there exists a constant $\kappa_4 = \kappa_4(\rho)$ such that if \eqref{wL2} holds and if
\beq
\dep \int_{B_\nu(\rho)} e_{\e,\nu}(v) \  d\yn \le 1 +\zeta_0 \quad\mbox{ for some }\zeta_0\in (0,\kappa_4),
\label{mm2h2}\eeq
then
\beq
 \int_{B_\nu(\rho)} \left|\frac \e 2  v_{\yn}^2  - \frac 1 \e F(v) \right| d\yn \le  C (\sqrt{\zeta_0} +  e^{-c/\e}).
\label{mm2c2}\eeq
\label{L.mm2}\end{lemma}

The proof of Proposition \ref{localest} uses only the first conclusion \eqref{mm2c1} of this lemma. The
other conclusion \eqref{mm2c2} is used in the proof of Theorem \ref{T3}, when we deduce control
over the full energy-momentum tensor from simpler energy estimates like those of Proposition \ref{localest}.

\begin{proof}[Proof of Lemma \ref{L.mm2}]
{\bf 1}. Note that \eqref{mm2c1} is obvious if \eqref{mm2h2} fails, so it suffices to show that if \eqref{wL2} and
\eqref{mm2h2} hold, then both conclusions \eqref{mm2c1}, \eqref{mm2c2} follow.

First, we define 
$Q(s) := \int_0^s \sqrt{2 F(\sigma)} d \sigma$,
and for any function $w\in H^1(B_\nu(\rho))$,
we estimate
\[
\e e_{\e,\nu}(w)  = \frac \e 2 w_{\yn}^2 + \frac 1 {\e} F(w)
\ge\sqrt{2 F(w)} |w_{\yn}| = | \partial_{\yn}  (Q\circ w)|.
\]
Thus since $\dep = \frac \e{\kappa_1}$, 
\beq
\dep \int_{B_\nu(\rho)}   e_{\e,\nu}(w) \ \ge \frac 1 {\kappa_1}\int_{B_\nu(\rho)}  | \partial_{\yn}  (Q\circ w)|,
\label{mm.basic}\eeq
and for any $w$, to obtain lower bounds for the left-hand side, it suffices to show that $\yn\mapsto Q(w(\yn))$
has large total variation on $B_\nu(\rho) = (-\rho, \rho)$.

{\bf 2}. Next, fix $\alpha>0$ so that $F' = f$ is decreasing on $(1- \alpha, 1)$; this is possible
as $F$ is $C^2$ and attains its minimum at $1$ with $F''(1)>0$.

Let $v^+ :=  \sup_{\yn \in (\frac \rho 4, \frac{3\rho}4)}v(\yn)$.
If $v^+\le 1$, then \eqref{wL2} implies that
\[
\kappa_3 \ge \int_{\rho/4}^{3\rho/4} \yn |1 - v(\yn)|^2  d \yn \ge  C \rho^2(1-v^+)^2 .
\]
Thus by choosing $\kappa_3$ small enough we can arrange that 
$v^+ \ge 1 -  \theta\alpha$ for some $\theta\in (0,\frac 12)$ to be chosen below.
It then follows by the same argument that 
$v^- :=  \inf_{\yn \in (- \frac {3\rho} 4, -\frac \rho 4)}v(\yn) \le -1+\theta\alpha$.

{\bf 3}. We next claim that, once $\kappa_3$ and $\kappa_4$ are fixed in a suitable way,
our hypotheses imply that 
\beq
v \ge 1-\alpha\quad\mbox{ in }(\frac {3\rho}4, \rho) 
\quad\quad\quad\mbox{ and } \ \ 
v \le -1+\alpha\quad\mbox{ in }(-\rho, -\frac {3\rho}4).
\label{ec1}\eeq
This follows from \eqref{mm.basic} and Step 2 --- the latter implies lower
bounds on the total variation of $Q\circ w$ if \eqref{ec1} fails, and these lower bounds can be made to contradict \eqref{mm.basic} and \eqref{mm2h2}.

In more detail, let us suppose toward a contradiction that the first  inequality  in \eqref{ec1} fails.
Then using Step 2, there exist points $\yn^1< \yn^2< \yn^3$ such that 
$v(\yn^1) < -1+\theta\alpha, v(\yn^2)>1-\theta\alpha$, and $v(\yn^3) < 1-\alpha$.
Hence using the fact that
$Q$ is nondecreasing (as the antiderivative of the positive function $\sqrt {2F}$),
\begin{align*}
\kappa_1(1 +\kappa_4)
&\overset{\eqref{mm2h2}, \eqref{mm.basic}}\ge
\int_{B_\nu(\rho)}  | \partial_{\yn}  (Q\circ v)|\\
&\ge 
|\int_{\yn^1}^{\yn^2}  \partial_{\yn}  (Q\circ v)d\yn| +
|\int_{\yn^2}^{\yn^3}  \partial_{\yn}  (Q\circ v)d\yn| \\
&
\ge 
 |Q(1-\theta\alpha) - Q(-1+\theta\alpha)|+ |Q(1-\alpha) - Q(1-\theta\alpha)| 
\\
&\ge
 |Q(1-\theta\alpha) - Q(-1+\theta\alpha)|+  2\kappa_1 \kappa_4 
\end{align*}
if we choose $\kappa_4 := \frac 1{2\kappa_1}|Q(1-\alpha) - Q(1- \frac\alpha 2)|$,
since we have said that $\theta\le \frac 12$. This inequality is false when $\theta = 0$,
since $\kappa_1 = Q(1)- Q(-1)$, and so it also fails for sufficiently small $\theta\in (0,\frac 12)$. Hence we can choose $\kappa_3$ small enough to obtain a contradiction.

{\bf 4}. We now replace $v$ on the interval $(\frac {3\rho}4,\rho)$ by the 
minimizer of the functional
\[
w\mapsto \int_{\frac {3\rho}4}^\rho e_{\e,\nu}(w)  \  d\yn
\]
subject to the boundary conditions $w(\frac {3\rho}4) = v(\frac {3\rho}4)$ and $w(\rho) = v(\rho)$. Let $v_1$ denote the function
obtained in this way. Standard maximum principle arguments\footnote{The point is that
one can easily check that $w(\yn) := 1 - \alpha\frac{\cosh( b(\yn-(7\rho/8))/\e)}{\cosh(b\rho/8\e)}$ satisfies
$-w'' + \e^{-2}f(w) \le 0$ in $(\frac {3\rho}4, \rho)$, if $b$ is fixed small enough (depending on $F$).  Then 
in view of \eqref{ec1} and the fact that $f$ is decreasing on $(1-\alpha, 1)$, 
one can use the the maximum
principle to find  that  $ v_1 > w$ in $(\frac {3\rho}4, \rho)$.} imply that
$ v_1(\frac 78 \rho) \ge 1 - C e^{-c/\e}$.
In a  similar way, we can modify $ v_1$ on $(-\rho, -\frac {3\rho}4)$ to obtain a function
$v_2$ with less energy than that of $v_1$, and such that $v_2(-\rho) = v(-\rho)$, and $ v_2(-\frac 78 \rho)\le -1 + C e^{-c/\e}$.

Thus $|Q(v_2(\frac 78 \rho)) - Q(1)| \le C e^{-c/\e}$,
and similarly $|Q(v_2(-\frac 78 \rho)) - Q(-1)| \le C e^{-c/\e}$.
As a result, using \eqref{mm2h2} and \eqref{mm.basic} as in Step 3 and recalling
that $\kappa_1 = Q(1)-Q(-1)$, we obtain
\begin{align*}
\kappa_1 (1+ \zeta_0)
&\overset{\eqref{mm2h2}}\ge \int_{B_\nu(\rho)}  \e\, e_{\e,\nu}(v) d\yn  
\ge  \int_{B_\nu(\rho)} \e \,e_{\e,\nu}(v_2) d\yn
\\
&\ge 
|Q(v_2(-\rho)) - Q(v_2(-\frac 78 \rho))| + 
| Q(v_2(-\frac 78 \rho)) - Q(v_2(\frac 78\rho))| + 
|  Q(v_2(\frac 78\rho)) 
-Q(v_2(\rho) |
\\
&\ge
 |Q(v_2(-\rho)) - Q(-1)| + \kappa_1 +
|  Q(1)) 
-Q(v_2(\rho) |
- C e^{c/\e}.
\end{align*}
This implies \eqref{mm2c1}. Also, since $v_2 = v$ at $\pm\rho$,
the above implies that
\begin{align}
Q(v(\rho)) - Q(v(-\rho)) 
&= 
\kappa_1 + Q(v(\rho)) - Q(1)  - [Q(v(-\rho)) - Q(-1)]
\nonumber\\
&\ge
\kappa_1 -  |Q(v(\rho)) - Q(1)|  - [Q(v(-\rho)) - Q(-1)]
\nonumber\\
&\ge 
\kappa_1(1 - \zeta_0) - C e^{-c/\e}.
\label{Q.intest}\end{align}

{\bf 5}. We now use \eqref{Q.intest} to prove \eqref{mm2c2}.
First note that
\begin{align*}
 \int_{B_\nu(\rho)} &\left|\frac \e 2  v_{\yn}^2  - \frac 1 \e F(v) \right| d\yn 
\le
 \int_{B_\nu(\rho)}  \left|\sqrt{\e}  v_{\yn}  - \frac {\sqrt{2F(v)}} {\sqrt\e} \right| \ \  \left|\sqrt{\e}  v_{\yn}  + \frac {\sqrt{2F(v)}} {\sqrt\e} \right|  \ d\yn
 \\
&\le
C\left( \int_{B_\nu(\rho)}  \left|\sqrt{\e}  v_{\yn}  - \frac {\sqrt{2F(v)}} {\sqrt\e} \right|^2 \ d\yn \right)^{1/2}
\left( \int_{B_\nu(\rho)}\ \e\  e_{\e,\nu}(v)  \ d\yn \ \  \right)^{1/2} \\
\end{align*}
Expanding the square and recalling that $\sqrt {2F} = Q'$, we see that
\begin{align*}
\int_{B_\nu(\rho)}\frac 12  \left|\sqrt{\e}  v_{\yn}  - \frac {\sqrt{2F(v)}} {\sqrt\e} \right|^2 \ d\yn
&=
\int_{B_\nu(\rho)} e_{\e,\nu}(v) d\yn - \int_{B_\nu(\rho)} Q'(v) v_{\yn}  \ d\yn \\
&=
\int_{B_\nu(\rho)} e_{\e,\nu}(v) d\yn - [ Q(v(\rho)) - Q(v(-\rho)) ]
\\
&\overset{\eqref{mm2h2}, \eqref{Q.intest}} \le
\kappa_1(1 +\zeta_0) - (\kappa_1(1 - \zeta_0) - C e^{-c/\e})
\\
&\ \  \ \ \le C \zeta_0 + C e^{-c/\e}.
\end{align*}
Combining these inequalities and again appealing to \eqref{mm2h2}, we arrive at \eqref{mm2c2}.
\end{proof}

The next Lemma is used to establish requirement \eqref{calD2} as discussed in the Introduction.
In this lemma we write $v$ as a function of two variables, $y^0$ and $\yn$.

\begin{lemma}Let $B_\nu(\rho)\subset \R$ be an interval as above,
and let $v\in H^1((0,\tau)\times B_\nu(\rho))$ for some $\tau>0$.
Then there exists a constant $C$, depending on $\rho$ but independent of $\tau$ and of $\e\in (0,1]$, such that
\[
\int_{B_\nu(\rho)} |\yn| \ |v(0,\yn) - v(\tau, \yn)|^2 \ d \yn\le C \int_{(0,\tau)\times B_\nu(\rho)} \frac \e 2 v_{y^0}^2 + \frac {(\yn)^2}\e F(v) \ d\yn\ dy^0.
\]
\label{L.mm1}\end{lemma}


\begin{proof} 
{\bf 1}. For  $Q:\R\to \R$ as above such that $Q'(s) = \sqrt{2 F(s)}$,
\[
\frac \e 2 v_{y^0}^2 + \frac {y^2}\e F(v) \  \  \ge \ \   |\yn|\sqrt { 2 F(v)} |v_{y^0}|   \ \ = \ \   |\yn|\, |Q(v)_{y^0}|.
\]
By integrating this inequality,
we find that
\begin{align*}
 \int_{(0,\tau)\times B_\nu(\rho)} \frac \e 2 v_{y^0}^2 + \frac {y^2}\e F(v) \ d\yn \ dy^0 \ \ 
& \ge \int_{B_\nu(\rho)}   \int_0^\tau |\yn |\ \left| Q(v)_{y^0}  \right|\, dy^0\, d\yn \\
&  \ge \ \int_{B_\nu(\rho)} |\yn| \ |Q(v(\tau,\yn)) - Q(v(0,\yn)) | d\yn.
\end{align*}
Finally, our assumption \eqref{scalarF} that $F(s) \ge (1 - |s|)^2$ and elementary calculus imply that 
\[
|Q(b) - Q(a)| \ge c (b-a)^2,
\] 
and the lemma follows.
\end{proof}

Now we can give the 

\begin{proof}[Proof of Proposition \ref{localest}]
Since the equation is well-posed in $H^1\times L^2$, and since all the quantities in the statement are continuous in $H^1\times L^2$, we may prove the Proposition for $v$ smooth.

In the proof we will write simply $\calD(\cdot)$ instead of $\calD(\cdot; \rho_1/2)$.

{\bf Step 1}. 
We may assume that $s_1=0$. We will use the notation  $s_{\max} := \min \{ \rho_1/2c_*, T_1\}$ and 
\[
W_\nu(s) :=B_{\nu}(\rho_1-c_* s ),
\quad\quad\quad\quad
W(s) :=   \T^n \times W_\nu(s).
\]
We define
\begin{align*}
\zeta_1(s) 
&:=
\dep  \int_{  \{s\}\times W(s)}   (1 +  (\yn)^2)\,  e_\e(v) \ dy' -1\\
\zeta_2(s) 
&:= \calD(v(s))\\
\zeta_3(s) 
&:=
\dep \int_{  \{s\}\times W(s) }
  |D_\tau v|^2 +(\yn)^2\left[  |\nabla_\nu v|^2  + \frac 1{\e^2} F(v) \right] \ dy' .
\end{align*}
We first claim that
\beq
\zeta_1(s)  \le C \zeta_0 + C \int_0^s \zeta_3(\sigma) d\sigma
\quad\quad\quad\mbox{ for }s\in (0,s_{max}].
\label{zp.est}\eeq
Towards this end we   compute
\begin{align*} 
\zeta_1'(s)= I_1 -&c_* I_2,\quad\mbox{ where }\\
I_1 &:=
\dep \int_ {  \{s\}\times W(s)}
 (1 +  (\yn)^2)\, \frac \partial {\partial y^0} e_\e(v) 
 \ dy' \\
I_2& =  \dep\int_{\{s \}\times  \T^n \times \partial W_\nu(s) }
 (1 +   (\yn)^2)\, e_\e(v)  \ d{\yt}' .
\end{align*}
To estimate $I_1$, we use Lemma \ref{L.eflux} and integrate by parts in the spatial variables.
From \eqref{vp.def} we  easily see that
$|\yn|  \, |\vp^\nu| \le C (|D_{\tau }v|^2  + (\yn)^2 |\nabla_\nu v|^2)$.
Thus we arrive at
\begin{align*} 
I_1&\le
C\dep \int_{  \{s\}\times W(s)} (|D_{\tau }v|^2  + (\yn)^2 |\nabla_\nu v|^2)
 \ dy' 
\\ &\ \ \ 
\quad\quad\quad\quad+\dep \int_{ \{ s \} \times \T^n \times \partial W_\nu(s)  }
 (1+ (\yn)^2) \  |\vp^N| \ d{\yt}' .
\end{align*}
Our choice \eqref{cstar.def} of $c_*$ exactly guarantees that 
$|\vp_N| \le c_* e_\e(v)$, so that the boundary term above is dominated by $- c_*I_2$.
It follows that $\zeta_1' \le C\zeta_3$. Since it is clear from \eqref{Ploc.h1} that $\zeta_1(0) \le \zeta_0$, we conclude that  \eqref{zp.est} holds.

{\bf Step 2}. Next, we estimate $\zeta_2$.  Using the hypotheses  and  Lemma  \ref{L.mm1}
we find that
\begin{align}
\zeta_2(s)
&\le
2 \calD(v(0))
+ 2\int_{  \T^n\times B_\nu(\rho_1/2)} |\yn| \  |v(s, y') - v(0, y')|^2 \  \ dy' \nonumber\\
&\le
2\zeta_0 + C\int_{\T^n} \left( \int_0^s \int_{ \T^n\times B_\nu(\rho_1/2)} \frac \e 2 |v_{y^0}|^2  + \frac {(\yn)^2}\e
F(v) \ d\yn \ dy^0\right) d{\yt}'
\nonumber\\
&\le 2\zeta_0 + C \int_0^s \zeta_3(\sigma) d\sigma
\label{z2.est}\end{align}
for $s\le s_{max} $.
We have changed the order of integration and used the fact that $ \T^n\times B_\nu(\rho_1/2) = W_\nu(\rho_1/2c_*)\subset W_\nu(s)$
for $s\le s_{max} \le \rho_1/2c_*$.

{\bf Step 3}.
Finally, we claim that 
\beq
\zeta_3(s) \le
C\left( \zeta_1(s)  \ + \   \zeta_2(s) \ + \  e^{-C/\e}   \right)
\label{zeta3.est}\eeq 
for every $s\in (0,s_{max}]$. 
We fix such an $s$, and we often write $v(\cdot)$ instead of $v(s,\cdot)$.
Note that \eqref{pos1a} implies that
\[
(1+(\yn)^2) e_\e(v) \ge     \frac 1 2 \lambda |D_\tau v|^2 + (1+(\yn)^2)e_{\e,\nu}(v) .
\]
It follows from this  and the definitions of $\zeta_1,\zeta_3$  that
\[
\zeta_1(s) \ge  c \  \zeta_3(s)
+
\ \ \dep \int_{ \{s\}\times W(s)} \,  e_{\e,\nu}(v)  \ dy' \ - \ 1.
\]
Thus it suffices to show that 
\beq
1 - \dep \int_{ \{s\}\times W(s)} \,  e_{\e,\nu}(v)  \ dy'  \ \ \le \ \ C \zeta_2(s) \ + \ C e^{-c/\e}.
\label{z3.sts}\eeq
To do this, let us say that a point ${\yt}' \in \T^n$ is {\em good } if 
\[
\calD_\nu(v(\yt')) \le \kappa_3
\]
and bad otherwise, for $v(\yt')(\yn) := v(\yt', \yn)$. Then Chebyshev's inequality implies that
\beq
| \{  {\yt}'\in \T^n\ : {\yt}'\mbox{ is bad } \}| \le \frac 1{\kappa_3}\int_{\{ s \}\times  \T^n}
\calD_\nu(v(\yt')) d\yt' = C \calD(v(s)) = C \zeta_2(s).
\label{badpts}\eeq
Thus $ \left|\{  {\yt}'\in \T^n\ : {\yt}'\mbox{ is good} \} \right| \ge 1 - C \zeta_2(s)$, and
so Lemma \ref{L.mm2} implies that
\begin{align}
&\dep\int_{\{ s \}\times W(s)}  e_{\e,\nu}(v) \ d y'
\nonumber \\
&
\quad
\quad\quad\quad \ \ 
\ge 
\int_{\{  (s,{\yt}') : {\yt}'\in \T^n \mbox{\scriptsize is good} \} }   \left(\dep \int_{W_\nu(s)}  e_{\e,\nu}(v) \ d\yn \right) d{\yt}' \nonumber \\
&\quad
\quad\quad\quad 
\overset{\eqref{mm2c1}}\ge
\left(1 - C \zeta_2(s) \right)( 1 - C e^{-c/\e} ).
\label{z3final}
\end{align}
This proves \eqref{z3.sts}, and hence \eqref{zeta3.est}.

{\bf Step 4}.
By combining the previous few steps and recalling that $\zeta_0\ge \e^2$, we see that
\[
\zeta_3(s) \le C \zeta_0  + C \int_0^s\zeta_3(\sigma) d\sigma,
\]
so Gronwall's inequality implies that that there exists some $C$ such that $\zeta_3(s)\le C\zeta_0$
for all $s\in (0,\rho_1/2c_*)$. Then \eqref{zp.est} and \eqref{z2.est} imply that $\zeta_1(s), \zeta_2(s) \le C \zeta_0$.
These estimates imply all the conclusions of the proposition.
\end{proof}


%

\section{initial energy estimates, $k=1$.}\label{S:nonzero}

In this section we indicate how to modify the above arguments to obtain control over $v$ on a portion of  a hypersurface of the form $\{ y^0 = \mbox{constant}\}$, starting from our assumptions \eqref{idata1}--\eqref{idata4} about 
$u$ at $t=0$, which translate to 
information about $v$ on a hypersurface of the form $\{ (b(y'), y') : y'\in \T^n\times B_\nu(\rho_0)\}$, with $b$ in general a non-constant function.  (Recall that the function $b$ was found in Lemma \ref{b.def}).
This is in general needed before we can start to iterate Proposition \ref{localest}.

We note that if we assume that the minimal surface $\Gamma$ has velocity $0$ at time $t=0$,
then it is easy to check that  $b(y') \equiv0$. As a result, 
the hypotheses  \eqref{Ploc.h1}, \eqref{Ploc.h2} of Proposition \ref{localest}
follow immediately in this case 
from our general assumptions \eqref{idata1}, \eqref{idata2}-\eqref{idata4} on the initial data. So 
the reader who is wiling to accept this restriction on $\Gamma$ can skip this section (and Section \ref{S:vinitial}) 
without any loss.

We continue to follow the notational conventions for the case $k=1$, summarized at the beginning of Section \ref{S:energy1}.
We will prove

\begin{proposition}
Assume that $v:(-T_1, T_1)\times\T^n\times B_\nu(\rho_0)\to \R$ is a solution of \eqref{v.eqn} on  
with data that satisfies \eqref{idata2}--\eqref{idata4} 
on the hypersurface $\{ (b(y'), y') : y'\in \T^n\times B_\nu(\rho_0)\}$.

Then there exists  some $s_1>0$ and $\rho_1>0$  for which $v$ satisfies the hypotheses \eqref{Ploc.h1}, \eqref{Ploc.h2} of Proposition \ref{localest}, with $\zeta_0$ replaced by $C\zeta_0$, and such that in addition
\[
\dep 
\int_{\{y\in (-T_1, s_1)\times \T^n \times B_\nu(\rho_1): \psi^0(y)>0\}  } 
\left[ |D_\tau v|^2 + |\yn|^2( |\nabla_\nu v|^2 + \frac 1{\e^2}F(v)) \right] \ dy
\le
C \zeta_0.
\]
\label{glocalest}\end{proposition}

If we simply tried to repeat our earlier arguments, we would have to worry about the
way in which a cone with slope $c_*$ intersects the initial hypersurface, and these considerations
would force us to impose unnatural restrictions on the initial velocity of the surface $\Gamma$.
We therefore exploit finite propagation speed in a different and sharper way than in our earlier arguments.
(We could have done this earlier, but we wanted to present our basic estimate in a relatively simple setting.)
This and other considerations force us  to introduce a certain amount of notation.

We start by defining
\beq
\calC := \{ (t,x)\in \R^{1+N} :  \dist(x, \Gamma_0) < \tau - t,  t>0  \}.
\label{calC.def}\eeq
where $\dist$ denotes the  Euclidean distance function, $\Gamma_0 = \{ H(0, {\yt}') : {\yt}'\in \T^n \}$, and $\tau>0$ is chosen so small that 
\beq
\calC \subset\subset \mbox{Image}(\psi).
\label{tausmall}\eeq
Note that
$\calC$ consists of the set of points for which the solution of the semilinear wave equation \eqref{slw} depends solely on the data in the set $\calC_0 := \{ x\in \R^N :  \dist(x, \Gamma_0) < \tau \}$. 
We continue by defining 
\[
V := \psi^{-1}(\calC),
\]
\[
s_0 := \inf \{ y^0\in (-T_1,T_1) \ : \ [ \{y^0\} \times \T^n\times B_\nu(\rho_0) ] \cap V \ne\emptyset \}
\]
and
\[
V^*:= \{ y = (y^0, y')\in (s_0,T_1)\times \T^n\times B_\nu(\rho_0) \  : \ (s,y')\in V\mbox{ for some }s \ge y^0\}
\]
Thus $V^*$ is just $V$ ``extended downward'' in the timelike $y^0$ variable to   $s_0$.
For $s\in R$ we define
\[
V(s) := \{ y\in V : y^0<s \}, \quad\quad
V^*(s) := \{ y\in V^* : y^0<s \}.
\]
We further define
\begin{align*}
\partial_0 V (s)
&:=  
 \{ y\in \partial V (s)\ : \psi^0(y) = 0 \},\\
\partial_1 V (s)
&:= 
\{ y = (y^0,y') \in \partial V(s) \ :y^0 = s \}, \\
\partial_2 V (s)
&:= 
\partial V(s) \setminus \left( \partial_0 V (s)
\cup \partial_1 V (s).
\right)
\end{align*}
We will also write 
\begin{align*}
\partial_1V^*(s) &:= \{ y = (y^0,y') \in \partial V^*(s) \ :y^0 = s  \}\\
\partial_0 V &:=  \{ y\in \partial V\ : \psi^0(y) = 0 \}\\
W_0 
&:= \{ y'\in \T^n\times B_\nu(\rho_0) : (y^0, y')\in \partial_0 V
\mbox{ for some }y^0\}.
\end{align*}
Finally we define
\[
W_i(s) 
:= 
\{ y'\in \T^n\times B_\nu(\rho_0) : (y^0, y')\in \partial_i V (s)
\mbox{ for some }y^0 \}
\]
for $i=0,1,2$, and similarly $W^*_i(s)$.

The next lemma collects some geometric facts that we will need  about the sets defined above.

\begin{lemma}
\beq
(W_0(s) \setminus W_1(s)  ) \cap W_1^*(s) = \emptyset\quad\mbox{ for all } s.
\label{cone3}\eeq
In addition, there exists $s_1>0$ and $\rho_1>0$ such that 
\beq
\mbox{$(s_0, s_1)\times\T^n\times B_\nu(\rho_1)\subset V^*$}
\quad\quad\quad\mbox{ and }
\ \ 
\mbox{$\{ s_1 \}\times\T^n\times B_\nu(\rho_1)\subset V$}.
\label{s1.def}\eeq
\label{L.cone}\end{lemma}

\begin{proof}[proof of Lemma \ref{L.cone}]
To prove  \eqref{cone3},
fix $y'\in W_0(s)\setminus  W_1(s)$. The definitions imply that the line $\{  (y^0, y') : y^0\in \R\}$
intersects $\partial_0 V(s)$ and does not meet $\partial_1 V(s)$, so it must leave $\bar V$ at a point $(\sigma,y')$ with
$\sigma<s$. 
Arguments like those of Lemma \ref{b.def}
show that once the line has left $\bar V$, it cannot re-enter, as if it did, the timelike curve
$s\mapsto X(s) := \psi(s, y')$  (see Lemma \ref{b.def}) would intersect  $\partial^+ \calC  :=  \{ (t,x)\in \partial \calC: t > 0\}$
more than once, which is impossible.
Thus the line does not intersect $\bar V^*$ at any point $(y^0,y')$ with $y'>\sigma$, and so it cannot intersect
$\partial_1V^*(s)\subset \{( (y^0,s)\in \bar V^* : y^0 =  s\}$. Thus $y'\not\in W^*_1(s)$, proving \eqref{cone3}.

Next, the existence of $s_1, \rho_1>0$ satisfying  \eqref{s1.def} follows from the fact that the (Euclidean) distance from 
$\{ 0\}\times \T^n\times \{0\} = \psi^{-1}(\Gamma_0)$ to $\partial^+V := \partial V \setminus \partial_0V= \psi^{-1}(\partial^+\calC)$ is positive.
This last fact in turn is clear from the fact that the distance from $\Gamma_0$ to $\partial^+\calC$ is positive, together with the smoothness of $\psi$.
\end{proof}

Recall that $v_0:\T^n\times B_\nu(\rho_0)\cong \{0\}\times \T^n\times B_\nu(\rho_0)$ was defined in \eqref{v0.def}.
We extend $v_0$ to
$(-T_1,T_1) \times \T^n\times B_\nu(\rho)$
such that it is independent of $y^0$; this extended function is still denoted $v_0$.

The remainder of this section contains the proof of Proposition \ref{glocalest}.
In the proof,
when we want to distinguish between row vectors and column vectors (which one can think as vectors and covectors, respectively), we will write $\cvec{\xi}$ to denote a column vector, with components $\xi^\alpha$,
and  $\rvec{\xi}$  for a row vector, with components $\xi_\alpha$. 

\begin{proof}[Proof of Proposition \ref{glocalest}]
As in Proposition \ref{localest}, it suffices to prove the proposition for $v$ smooth in $\bar V$.

{\bf Step 1}. 
We define $v^*:V^*\to \R$ by
\beq
v^*(y) = \begin{cases}v(y)&\mbox{ if }y\in V\\
v_0 (y)&\mbox{ if }y\in V^*\setminus V.
\end{cases}
\label{vstar.def}\eeq
Since $v=v_0$ on $\bar V \cap (V^*\setminus V) = \partial_0V$, it is easy to see that $v^*$ is Lipschitz in $V^*$.
Note however that the derivatives of $v^*$ are in general discontinuous across $\partial_0 V$.

We define
\begin{align*}
\zeta_1(s) 
&:=
\dep \int_{  \partial_1 V^*(s)}  \  (1 +  (\yn)^2)\, e_\e(v^*)  \ dy' \ - 1\\
\zeta_2(s) 
&:= \calD( v^*(s); \rho_1/2)\\
\zeta_3(s) 
&:=
\dep \int_{ \partial_1 V^*(s) }
\left[  |D_\tau v^*|^2 + (\yn)^2 \, e_{\e,\nu}(v^*) \right] \ dy' .
\end{align*}
In view of \eqref{s1.def}, we can repeat word for word the arguments from the proof of
Proposition \ref{localest} to find that
\beq
\zeta_3(s) \le C\left(\zeta_1(s) + \zeta_2(s) +e^{-c/\e}\right)
\label{glocal2}\eeq
and
\[
\zeta_2(s) \le 2  \zeta_2(s_0) + C \int_{s_0}^s \zeta_3(\sigma)\ d\sigma
\]
for every $s\in [s_0 ,s_1]$.
And the definition of $s_0$ implies that $v^* = v_0$ on $\partial_1V^*(s_0) := \{s_0\}\times W_0$,
so that $\zeta_2(s_0) \le \zeta_0$ by \eqref{idata4}. Thus
\beq
\zeta_2(s) \le  C \zeta_0 + C\int_{s_0}^s \zeta_3(\sigma)\ d\sigma
\label{glocal1}\eeq
for every $s\in [s_0,s_1]$.

The remainder of the proof is devoted to the estimate of
$\zeta_1$. Since $v^*$ is smooth away from $\partial_0 V$ and 
(by Fubini's Theorem) $\partial_1 V^*(s) \cap \partial_0 V$ has $\calH^N$ measure $0$ 
for $\calL^1$ a.e. $s$,
the definition of $v^*$ implies that 
\beq
e_\e(v^*) 
= 
\begin{cases}
e_\e(v)&\mbox{ $\calH^N$ a.e. in  } \partial_1 V(s)\\
e_\e(v_0 )&\mbox{  $\calH^N$ a.e.  } \partial_1 V^*(s)\setminus \partial_1V(s)
\end{cases}
\label{good.s}\eeq
for a.e. $s$. Also, if $[\cdots]$ denotes an integrand that does not depend on the $y^0$
variable, then clearly $\int_{ \partial_1^*V(s)\setminus \partial_1V(s)}[\cdots] dy' \ = \ 
\int_{ W_1^*(s)\setminus W(s)}[\cdots] dy'$.
Thus 
\beq
\int_{  \partial_1 V^*(s)} \, (1 +  (\yn)^2)\, e_\e(v^*)  \ dy' =
\int_{  \partial_1 V(s)} \, (1 +  (\yn)^2)\, e_\e(v)  \ dy' +
\int_{ W_1^*(s)\setminus W_1(s)} \,  (1 +  (\yn)^2) e_\e(v_0) \ dy'
\label{z1.est0}\eeq
for a.e. $s$. 

{\bf Step 2.} 
We claim that for a.e. $s$, 
\begin{align}
\dep \int_{  \partial_1 V(s)} \, (1 +  (\yn)^2)\, e_\e(v)  \ dy' 
&\le
\dep \int_{  \partial_0 V(s)} \,  (1 +  (\yn)^2)\,(-n_0 e_\e(v) + n_i \vp^i)  \ \calH^N(dy)
\nonumber \\&\quad\quad\quad\quad\quad\quad\quad\quad\quad\quad\quad\quad\quad\quad
 +C \int_{s_0}^s \zeta_3(\sigma) d\sigma,
\label{z1.est1}\end{align}
where  $\rvec n(y)$ denotes the (Euclidean) outer unit normal at a point $y\in\partial V(s)$, thought of as a row vector with components $n_\alpha$, and $\vp^i$ is defined in \eqref{vp.def} and appears in the local energy estimate of Lemma \ref{L.eflux}.

{\bf Step 2.1} 
To prove \eqref{z1.est1} we will first integrate by parts and show that some of the boundary terms have a sign and hence can be discarded. (In this we basically follow the proof of Proposition \ref{localest}.) 
For this,  it is useful to  define $\tilde \calT_\e = \tilde \calT_\e(v)$  by
\beq
\tilde \calT^\alpha_{\e,\beta} := 
\delta^\alpha_\beta \left( \frac 1 2 g^{\gamma\delta}v_{y^\gamma} v_{y^\delta} + \frac 1{\e^2} F(v)\right) - 
g^{\alpha \gamma} v_{y^\gamma} v_{y^\beta}.
\label{tildeT}\eeq
Observe from the definitions that\footnote{In fact $\tilde T_\e$ is just the energy-momentum tensor for $u$ expressed in terms of the $y$-cordinates. 
The fact that the energy-momentum tensor is divergence-free, when written  in the
$y$ coordinates, takes the form $\partial_{y^\alpha} (\tilde\calT^\alpha_{\e,\beta} (v)\sqrt{-g}) = 0 \forall \beta$.  
One can use this fact to give a proof of Lemma \ref{L.eflux} slightly different from the one  presented earlier. }
\beq
\tilde \calT^0_{\e,0}(v) = e_\e(v)\quad\quad\mbox{ and }\ \  \tilde \calT^i_{\e,0}(v) = -\vp^i
\label{eT}\eeq
so that the conclusion of Lemma \ref{L.eflux} can be written
$\partial_{y^\alpha} \tilde \calT^\alpha_{\e,0} \le C (|D_\tau v|^2 + (\yn)^2|\nabla_\nu v|^2)$.

We now  compute
\begin{align}
\dep \int_{V(s)} \partial_{y^\alpha}  \left[ (1+  (\yn)^2)\tilde \calT^\alpha_{\e,0} \right]\ dy
&\le 
C \dep \int_{V(s)}  \left[ (|D_\tau v|^2 + (\yn)^2|\nabla_\nu v|^2) + y^N \tilde \calT^N_{\e,0}\right] \ dy\nonumber\\
&\le
C\dep \int_{V(s)}  (|D_\tau v|^2 + (\yn)^2|\nabla_\nu v|^2) \ dy\nonumber \\
&\le
C \int_{s_0}^s \zeta_3(\sigma) d\sigma.
\label{z1.e1b}\end{align}

On the other hand, we can  integrate by parts to rewrite the left-hand side
as an integral over $\partial V(s)$.
Then noting that $\rvec n(y) = (1,0,\ldots, 0)$ for $y\in \partial_1V(s)$, we find that
\begin{align*}
\dep \int_{V(s)}  \partial_{y^\alpha} \left[ (1+  (\yn)^2)\tilde \calT^\alpha_{\e,0} \right]
&\ = \ 
\dep \int_{\partial_1V(s)}  (1+  (\yn)^2) e_\e(v) dy' \\
&\quad\quad\quad\quad
 +\dep  \int_{\partial_0 V(s) }  \left[ (1+  (\yn)^2) n_\alpha \tilde \calT^\alpha_{\e,0} \right]  \ d\calH^N(y)\\
&\quad\quad\quad\quad
+ \dep \int_{\partial_2 V(s) }  \left[ (1+  (\yn)^2) n_\alpha \tilde \calT^\alpha_{\e,0} \right] \ d\calH^N(y).
\end{align*}
By combining this with \eqref{z1.e1b} and recalling \eqref{eT}, we see that our claim \eqref{z1.est1}
will follow if we can show that the last integral on the right-hand side is positive. 

{\bf Step 2.2.}  To do this we will show that 
\beq
\mbox{$n_\alpha(y) \  \tilde \calT^\alpha_{\e,0}(y) \ \ge \  0$  \ \ \ \ \  for a. e. $y\in\partial_2 V(s)$}.
\label{Vell}\eeq
To do this, we first check that 
\beq
g^{\alpha \beta} n_\alpha n_\beta = 0 \quad\quad\mbox{ a.e. }y \in \partial_2V.
\label{nDphi}\eeq
In fact, we will show that this holds at every $y\in \partial_2 V$ such that
$\partial C$ has a tangent plane at $x=\psi(y)$; this is a set of full measure.
Fix such a $y$ and let $\cvec w = (w^\alpha)$ be any (column) vector
tangent to $\partial \calC$ at $x$. Also, 
let $\rvec m(x)$ denote the (Euclidean) outer unit normal to $\calC$ at $x\in \partial \calC$, again thought of as a row vector with components $m_\alpha$. Writing $\phi = \psi^{-1}$ as usual, 
since $\phi$ maps $\partial \calC$ to $\partial V$, it
is clear that $D\phi(x)\, \cvec w$ is tangent to $\partial V$ at $\phi(x) = y$, which implies that
$\rvec n(y) D\phi(x) \, \cvec w = 0$.  Since this holds for all tangent vectors $\cvec w$ at $x$,
it follows that $\rvec n(y) \ D\phi(x)$ is parallel to the Euclidean unit  normal $\rvec m$ to $\partial C$ at $x$; that is
$\rvec n(y)D\phi(x) = \lambda \rvec m(x)$ for some $\lambda\in \R$.
And the form of $\calC$ implies that $\rvec m$ is a null vector, so that 
\[
0 = \lambda^2 \eta^{\alpha\beta}m_\alpha m_\beta \ = \ 
\lambda^2 \eta^{\alpha\beta}n_\gamma\phi^\gamma_\alpha m_\delta \phi^\delta_\beta =  
g^{\gamma \delta} n_\gamma n_\delta ,
\]
proving \eqref{nDphi}.
Note also that $n_0(y)>0$ for $y\in \partial_2 V$, and recall further that $F(u)\ge 0$.
Thus
\begin{align*}
n_\alpha \tilde \calT^\alpha_{\e,0}
&= 
\frac{n_0}{\e^2} F(u) +  \frac {n_0} 2 g^{\alpha\beta}v_{y^\alpha} v_{y^\beta} -  v_{y^0}  g^{\alpha\beta}n_{\alpha} v_{y^\beta} \\
&\ge 
\frac {n_0} 2 g^{\alpha\beta}v_{y^\alpha} v_{y^\beta} -  v_{y^0}  g^{\alpha\beta}n_{\alpha} v_{y^\beta} \\
&= 
\frac {n_0}2 g^{\alpha\beta}(Dv - \frac {v_{y^0}} {n_0}n)_\alpha (Dv - \frac {v_{y^0}} {n_0}n)_\beta
\end{align*}
using \eqref{nDphi}. If we write $\xi := Dv - \frac {v_{y^0}}{n_0} n$, then clearly $\xi_0 = 0$, which implies that
\[
g^{\alpha\beta}\xi_\alpha \xi_\beta = g^{ij}\xi_i\xi_j = a^{\alpha\beta} \xi_\alpha \xi_\beta  \ge 0.
\]
Thus we have proved \eqref{Vell}.

{\bf Step 3}. 
Next we note
that
\beq
-\int_{  \partial_0 V(s)} (1 +  (\yn)^2)\,n_0(y)\,  e_\e(v) (y) \calH^N(dy)
= 
\int_{ W_0(s)} (1 +  (\yn)^2)\, e_\e(v) (b(y'),y') \ dy,
\label{z1.areaf}\eeq
where we  recall that $\partial_0 V = \{ (b(y'), y') \ : \ y'\in W_0\}$, and hence
that $\partial_0 V(s) =  \{ (b(y'), y') \ : \ y'\in W_0(s)\}$. 
This is obvious, because the Euclidean outer unit normal to $V(s)$ is given by 
$\rvec n = (-1, \nabla b)/(1+|\nabla b|^2)^{1/2}$, with the minus sign appearing because
$V$ sits above the graph. Thus $-n_0(b(y'), y') = (1+|\nabla b(y')|^2)^{-1/2}$, and
then \eqref{z1.areaf} follows from a change of variables using the area formula.

{\bf Step 4}. 
Now we combine \eqref{z1.areaf} with \eqref{z1.est0}, \eqref{z1.est1} to find that
\[
\zeta_1(s) \le
 C \int_{s_0}^s \zeta_3(\sigma) \, d\sigma  + A + B,
\]
for a.e. $s\in [s_0, s_1]$,
where
\begin{align*}
A 
&:=
\dep \int_{W_0(s)}(1+(\yn)^2)( e_\e(v) - e_\e(v_0)) (b(y'), y') dy' 
+\dep \int_{  \partial_0 V(s)} (1 +  (\yn)^2)\, n_i \vp^i  \ d\calH^N, \\
B
&:=
\dep \int_{W_1^*(s)\setminus W_1(s)} (1+(\yn)^2) e_\e(v_0) dy' +  \dep\int_{W_0(s)}(1+(\yn)^2)e_\e(v_0)\ dy'
 -1.
\end{align*}
We have checked in  Lemma \ref{L.cone} that
$(W_1^*(s) \setminus W_1(s)  ) \cap W_0(s) = \emptyset$; this is equivalent to \eqref{cone3}. 
Thus
\[
B \le \dep \int_{W_0}(1+(\yn)^2)e_\e(v_0)\ dy'
 -1 \ \ \overset{\eqref{idata2}}\le \ \ \zeta_0.
\]
To estimate $A$, we differentiate the identity $v(b(y'), y') = v_0(y')$ to find that 
$v_{y^0} \nabla b + \nabla v = \nabla v_{0}$.
Thus $|D(v-v_0)| = |v_{y^0}(1, -\nabla b)| \le C |v_{y^0}|$ at points $(b(y'), y')\in \partial_0V$, using the control
over $\|\nabla b\|_\infty$ obtained in Lemma \ref{b.def}. 
It follows that at such points
\[
e_\e(v) - e_\e(v_0) = \frac 1 2 a^{\alpha\beta}(v-v_0)_{y^\alpha}(v+v_{0})_{y^\beta}
\le C \left( v_{y^0}^2 + |D_\tau v_0|^2 + |v_{y^0}| |\nabla_\nu v_0|\right).
\]
Similarly, using \eqref{coeffs3}, we see that  $|\vp^i| \le C(v_{y^0}^2 + |D_\tau v_0|^2 + (\yn)^2 |\nabla_\nu v_0|^2)$, so
\[
A \le \  \  C \dep \int_{\partial_0 V}  \left( v_{y^0}^2 + |v_{y^0}| \ |\nabla_\nu v_0|\right)
\  d\calH^N 
+ C  \dep \int_{W_0}  (|D_\tau v_0|^2 + (\yn)^2 |\nabla_\nu v_0|^2) \ dy'.
\]
Also, since $v_0(y') =v^*(s_0, y') $,
\[
\int_{W_0}  \e(|D_\tau v_0|^2 + (\yn)^2 |\nabla_\nu v_0|^2) \ dy' \ \le \zeta_3(s_0)
\overset{\eqref{glocal2}}{\le} C(\zeta_1(s_0) + \zeta_2(s_0) + e^{-c/\e})
 \overset{\eqref{idata2}, \eqref{idata3}}\le 
C \zeta_0.
\] 
Using this fact and \eqref{idata3} we conclude that $A\le C \zeta_0$, and hence that
\[ 
\zeta_1(s) \le C\int_{s_0}^s \zeta_3(\sigma)d\sigma + C  \zeta_0.
\]

{\bf Step 5}. The rest of the proof exactly follows that of Proposition \ref{localest}.
In the end we find that $\zeta_i(s_1)\le C \zeta_0$ for $i=1,2,3$, and in view of \eqref{s1.def}, these estimates immediately imply the conclusion of the proposition.

\end{proof}


\section{energy estimates, $k=2$}\label{S:vector_ee}

In this section, we prove energy estimates like  those from Sections \ref{S:energy1}, \ref{S:nonzero} above,
but now in the case $k=2$, so that we consider a vector-valued function $v: (-T_1,T_1)\times \T^n\times B_\nu(\rho)\to \R^2$
solving \eqref{v.eqn}, where  $B_\nu(\rho)\subset \R^2_\nu$ now denotes a $2$-dimensional ball, $\kappa_2$ is the constant chosen in \eqref{omega3.def}, $\dep = (\pi |\ln\e|)^{-1}$, and the nonlinearity in \eqref{slw} is  $f = \nabla F$ with $F:\R^2\to [0,\infty)$
satisfying \eqref{vectorF}.

The main results and proofs in this section are {strictly} analogous to Propositions \ref{localest} and \ref{glocalest}.
The chief difference is  that 
the ``defect confinement functional'' $\calD$ (discussed in the Introduction) has a quite
different form than in the case $k=1$. Thus, the arguments
we need to verify that the  desired 
properties \eqref{calD0}, \eqref{calD2} hold 
are quite different from, and more delicate than, their counterparts in the scalar case. 
Once suitable forms of these facts are established, we follow our earlier proofs with only cosmetic changes.

We will use machinery that relates the Jacobian and the Ginzburg-Landau energy. We will give precise statements of  the facts from the literature that we need, in the hope of rendering our arguments somewhat accessible to people who are not familiar with these results; see also the book \cite{sandserf} for a general reference on these topics.
The results we use (see Lemmas \ref{JSp1}, \ref{JSp2}, \ref{L.newJacest2}) are proved in the sources we cite for $F_{model}(u) = \frac 14 (|u|^2-1)^2$, but it is evident\footnote{In all the proofs we will cite,  easy truncation arguments are used to reduce to the
case of $u$ such that $|u|\le M$ a.e. for $M=2$ for example, and then \eqref{vectorF} implies that $\frac 1{(C \e)^2}F_{model}(u) \le \frac 1{\e^2}F(u) \le\frac 1{(\e/C)^2}F_{model}(u)$. It is then clear that results established for $F_{model}$ carry over to energy functionals that instead contain $F$, since everything we use is essentially unaffected if $\e$ is replaced by $C\e$ or $\e/C$.} from the proofs that they still apply to functions $F$ satisfying the assumptions \eqref{vectorF} that we impose here. 

For $v\in H^1( \T^n\times B_\nu(\rho); \R^2)$ we take $\calD$ to have the form (as when $k=1$)
\beq
\calD(v;\rho) :=
\int_{ \T^n} \calD_\nu(v(\yt');\rho)  \ d{\yt}',
\label{L1|||}\eeq
where $v(\yt')(\yn) = v(\yt', \yn)$.
And for $w = (w^1,w^2)\in H^1(B_\nu(\rho); \R^2)$, we define
\beq
\calD_\nu(w; \rho) := ||| J_\nu w - \pi \delta_0 |||_\rho
\label{calDnu.k2def}\eeq
where for  a measure $\mu$ on $B_\nu(\rho)$, 
\beq
||| \mu  |||_\rho := \sup \left\{ \int \omega(\yn) f(\yn) d\yn  \ : \omega \in C^2_c(B_\rho),  |\nabla \omega(\yn)| \le |\yn|^2, \| \omega \|_{W^{2,\infty}} \le 1\right\}.
\label{|||.def}\eeq
(Clearly $||| \cdot |||_\rho$ also makes sense for some distributions that are less regular than measures, but we will not need that here.)
Here we are using the notation 
$J_\nu w = \det \nabla_\nu w$. We will also write
${\bf J}_\nu w$ for the $2$-form ${\bf J}_\nu w = J_\nu w  \ d\yn$, where  $d\yn := dy^{\nu,1} \wedge dy^{\nu,2}$. Note that
\[
{\bf J}_\nu w
:=  d_\nu w^1 \wedge d_\nu w^2,\quad \mbox{ where } d_\nu w^i = \frac{\partial w^i}{\partial {y^{\nu,1} }} d y^{\nu,1} + \frac {\partial w^i}{\partial {y^{\nu,2}}} dy^{\nu,2} .
\]
(Recall that $y^{\nu, i} = y^{n+i}$.)

General results and heuristics about Jacobians and vortices (see for example  \cite{sandserf}), together with the definition of the
$|||\cdot |||_\rho$ norm,
suggest that  if $w:B_\nu(\rho)\to \R^2$ is a function possessing a single ``vortex of degree 1'' localized near some point in $B_\nu(\rho/2)$, then 
roughly speaking
\[
 \ \ ||| J_\nu w - \pi \delta_{0} |||_\rho \ \ \approx \ \ (\mbox{the distance from the origin to  the vortex })^3 \ {}
\]
(The cubic scaling on the right-hand side is related to the condition  $|\nabla \omega(\yn)| \le |\yn|^2$ imposed on test functions in the definition of $|||\cdot|||_\rho$.) Thus, the right-hand side of  \eqref{L1|||} is the average
of the above quantity over the tangential $\yt$ variables.

The first main result of this section parallels Proposition \ref{localest} above:

\begin{proposition}
Let $v:(-T_1,T_1)\times \T^n\times B_\nu(\rho_0)\to \R^2$ satisfy \eqref{v.eqn}, where $B_\nu(\rho)\subset \R^2_\nu$ and $f = \nabla F$ with  $F:\R^2\to \R$ satisfying \eqref{vectorF}.
Recalling that $\delta_\e = (\pi\ln\e|)^{-1}$, assume that there exist $s_1\in (-T_1,T_1)$, $\rho_1\in (0,\rho_0)$, and $\zeta_0\ge \dep$ such that
\beq
\dep  \int_{  \{s_1\}\times \T^n\times B_\nu(\rho_1)  }
 (1 + \kappa_2 |\yn|^2)\, e_\e(v) dy' - 1  \ \ \le \ \ \zeta_0
\label{vzeta0.h1}\eeq
\beq
\calD(v(0);\rho_1/2 ) 
 \ \ \le \ \ \zeta_0.
\label{vzeta0.h2}\eeq
Then there exists a constant $C$ such that 
\[   
\dep
\int_{  \{s_1+s\}\times \T^n\times B_\nu(\rho_1 - c_* s) }
 \left[   |D_\tau v|^2 +  |\yn|^2 ( |\nabla_\nu v|^2 + \frac 1{\e^2}F(v)) \right] \ dy' 
 \le  C \zeta_0
\]  
\[  
\dep \int_{  \{s_1+s\}\times \T^n\times B_\nu(\rho_1 - c_* s) }
 \  e_\e(v)(1+ \kappa_2|\yn|^2)
 dy' -  1 \ \le C \zeta_0
\]  
and
\[  
\calD(v(s); \rho_1/2) 
 \ \ \le \ \ C\zeta_0
\]  
for all $s\in [0, \rho_1/ 2c_*]$ such that $s_1+s < T_1$. Here $c_*$ is defined in \eqref{cstar.def}.
\label{vlocalest}\end{proposition}

As remarked earlier, there does not exist any initial data satisfying \eqref{vzeta0.h1}, \eqref{vzeta0.h2}
with $\zeta_0 \ll \dep$ when $k=2$, so the condition $\zeta_) \ge \dep$ is not restrictive.

The second main result of this section parallels Proposition \ref{glocalest}.

\begin{proposition}
Assume that $v:(-T_1, T_1)\times\T^n\times B_\nu(\rho_0)\to \R^2$ is a solution of \eqref{v.eqn} 
with data that satisfies \eqref{idata2}-\eqref{idata4} 
on the hypersurface $\{ (b(y'), y') : y'\in \T^n\times B_\nu(\rho_0)\}$, with $\zeta_0\ge \dep$
and with  $\calD$ as defined in \eqref{L1|||}.

Then there exists  some $s_1>0$ and $\rho_1>0$  for which $v$ satisfies the hypotheses \eqref{vzeta0.h1}, 
\eqref{vzeta0.h2}  of Proposition \ref{vlocalest}, with $\zeta_0$ replaced by $C\zeta_0$,
and such that in addition
\[
\dep 
\int_{\{y\in (-T_1, s_1)\times \T^n\times B_\nu(\rho_1): \psi^0(y)>0\}  } 
\left[ |D_\tau v|^2 + |\yn|^2( |\nabla_\nu v|^2 + \frac 1{\e^2}F(v)) \right] \ dy
\le
C \zeta_0.
\]
\label{P.vinitial}\end{proposition}

\subsection{variational stability estimates}

We start by establishing some properties relating the $|||\cdot|||_\rho$ norm of the Jacobian $Jv$ and the Ginzburg-Landau energy
$e_{\e,\nu}(v)$. These show will be used to show that $\calD(\cdot)$
 satisfies requirements \eqref{calD0} and \eqref{calD2} from the Introduction.

Our first result is analogous to Lemma \ref{L.mm2} and establishes a form of \eqref{calD0}. It is a straightforward consequence of the Jacobian machinery mentioned above. 

\begin{proposition}
For $\rho>0$ there exist constant $\kappa_3$ and $C$, both depending on $\rho$, such that if $w\in H^1(B_\nu(\rho); \R^2)$ and 
if
\beq
\calD_\nu(w; \rho)  \ = \  ||| J_\nu w - \pi \delta_0  |||_\rho  \ \le  \  \kappa_3
\label{Pv2.h}\eeq
then
\beq
|\ln \e|^{-1} \int_B e_{\e,\nu}(w) \  d\yn \ge \pi  - |\ln \e|^{-1}  C.
\label{Pv2.c}\eeq
\label{Pv2}\end{proposition}

The proof of Proposition \ref{Pv2} uses the following facts:

\begin{lemma}Suppose that $\e\in (0,1]$, that $w\in H^1(B_\nu(\rho); \R^2)$,  and that 
\[
\| J_\nu w - \pi \delta_0\|_{W^{-1,1}(B_\nu(\rho))} \le \frac \rho {10}. 
\]
Then
\[
\frac 1{|\ln \e|} \int_{B_\nu(\rho)} e_{\e,\nu}(w) \ d\yn \   \ge \  \pi- \frac C{|\ln\e|}.
\]
\label{JSp1}\end{lemma}

This follows for example from a much sharper estimate proved in \cite{jsp}, see Theorem 1.3.
A slightly different norm is used there in place of the $W^{-1,1}$ norm, 
but that result is easily seen to imply the one stated here.

\begin{lemma}
Suppose that $\e \in (0,1]$ and that $w\in H^1(B_\nu(\rho);\R^2)$ satisfies
\[
\frac 1{|\ln \e|}\int_{  B} e_{\e,\nu}(w) \ d\yn \ \   \le 3\pi/2. 
\]
Then there exists an integer $\ell\in \{ 0,\pm1\}$ and a point $\xi\in B$ such that
\[
\| J_\nu w - \pi \ell\delta_\xi\|_{W^{-1,1}(B_\nu(\rho))} \le C |\ln \e| \e^{1/4}. 
\]
\label{JSp2}\end{lemma}
This follows from Theorem 1.1 in \cite{jsp}.
Using the lemmas we give the 

\begin{proof}[proof of Proposition \ref{Pv2}]
{\bf 1}. Fix $w\in H^1(B_\nu(\rho);\R^2)$. We may assume that
\beq
\frac 1{|\ln \e|}\int_{ B_\nu(\rho)}  e_{\e,\nu}(w) \ d\yn \,  \ \le \  3\pi/2,
\label{GS.mayassume}\eeq
since otherwise \eqref{Pv2.c} is immediate.
So in view of Lemma \ref{JSp1} it suffices to show that there exists a constant
$\kappa_3(\rho)$ such that if \eqref{GS.mayassume} holds and $||| J_\nu w - \pi \delta_0 |||_\rho < \kappa_3$, then
\beq
\| J_\nu w - \pi\delta_0 \|_{ W^{-1,1}(B_\nu(\rho))} \le  \frac \rho{10}.
\label{s1}\eeq
In fact it suffices to show that there exists some $\e_0>0$ such that the above conclusion holds
if $\e\in (0,\e_0)$ in \eqref{GS.mayassume}, since we can arrange that \eqref{Pv2.c} holds for $\e>\e_0$
by choosing $C$ large enough.

Now
\eqref{GS.mayassume} and Lemma \ref{JSp2} imply that there exist an integer $\ell$ with $|\ell|\le 1$
and a point $\xi\in B_\nu(\rho)$ such that $\| J_\nu w - \pi \ell\delta_\xi\|_{W^{-1,1}(B_\nu(\rho))} \le C |\ln \e|  \e^{1/4}$.
Fix a  function $\omega_*\in C^2_c(B)$, with $|\nabla\omega_*(y)|\le |y|^2$ and $\| \omega_*\|_{W^{2,\infty}}\le 1$,
and such that $\omega_*(y) < \omega_*(0)$ if $y \ne 0$.
Then \eqref{Pv2.h} and the definition of the $||| \cdot  |||_\rho$ norm imply that 
\[
\int \omega_* \ J_\nu w  \ d\yn -  \pi \omega_*(0)  \ge -  \kappa_3.
\]
On the other hand, the estimate  $\| J_\nu w  - \pi \ell\delta_\xi\|_{W^{-1,1}(B)} \le C |\ln \e| \e^{1/4}$
implies that
\[
\int \omega_*\ J_\nu w \ d\yn -  \pi \ell \omega_*(\xi) \le  C \| \omega_*\|_{W^{1,\infty}}  |\ln \e| \e^{1/4} \le C   |\ln \e| \e^{1/4}.
\]
Thus
\beq
\ell \omega_*(\xi) \ge \omega_*(0) - \frac{ \kappa_3}\pi - C   |\ln \e| \e^{1/4}.
\label{om1}\eeq
Since $\omega_*(0)>0$, this implies that
$\ell = 1$  for all sufficiently small $\e>0$, if $\kappa_3$ is fixed small enough.
Then 
$\| Jw(\tau) - \pi\delta_\xi\|_{W^{-1,1}(B_\nu(\rho))} \le C |\ln \e| \e^{1/4}$, and as a result,
\begin{align*}
\| J(w(\tau)) - \pi\delta_0 \|_{ W^{-1,1}(B_\nu(\rho))} 
&\le  C |\ln \e| \e^{1/4} +
\pi  \|  \delta_\xi - \delta_0 \|_{ W^{-1,1}(B_\nu(\rho))} \\
&\le C |\ln \e| \e^{1/4}  +  \pi|\xi|
\end{align*}
where the last inequality follows immediately from the definition of the  $W^{-1,1}$ norm.
Since $\omega_*$ is continuous and achieves its maximum exactly at the origin, \eqref{om1} implies that if we fix $\kappa_3$ still smaller if necessary, then $\pi |\xi|<\rho/20$, and as a result \eqref{s1} holds, 
for all small $\e$.
\end{proof}

The second result about the $|||\cdot|||_\rho$ norm is analogous to Lemma \ref{L.mm2} and establishes a form of requirement \eqref{calD2}; in fact, the norm is designed exactly so that
an estimate of the form \eqref{Pv1.c1} holds. In the lemma we write $v$ as a function of $(y^0,\yn)\in \R\times \R^2_\nu$

\begin{proposition} 
Let  $v\in H^1((0,\tau)\times B_\nu(\rho) ;\R^2)$ for some $\rho, \tau>0$.
Then there exist positive constants $C,\alpha$, depending on $\rho$ but independent of $\tau$ and of $\e\in (0,1]$, such that
\begin{align}
||| J_\nu v(\tau,\cdot) - J_\nu v(0,\cdot) |||_\rho 
&\le 
C\dep 
\int_{(0,\tau)\times B_\nu(\rho)} 
(|\yn|^2 + \e^\alpha)( |\frac{|Dv|^2}2 + \frac 1{\e^2} F(v)) \ d\yn\, dy^0 
\nonumber\\
&\quad\quad\quad+
C \e^{\alpha}
 \left(1+ \int_{ \{ 0 \}\times B_\nu(\rho) } e_{\e,\nu}(v) d \yn 
+  \int_{ \{\tau\}\times B_\nu(\rho) } e_{\e,\nu}(v) d\yn \right).
\label{Pv1.c1}\end{align}
\label{Pv1}\end{proposition}

We believe that the $\e^\alpha$ in the first integral on the right-hand side of \eqref{Pv1.c1} could be removed with some work, but the estimate is false without the boundary  terms in the second 
line of \eqref{Pv1.c1}. In any case, all these terms will be negligible in our later arguments.

The proof of Proposition \ref{Pv1} requires the following Lemma.

\begin{lemma} There exist universal constants $C, \alpha>0$
such that, given any  $U\subset \R^3 = \R_{y^0}\times \R^2_\nu$,
and  $w \in H^1(U; \R^2)$, 
\begin{align}
\left| \int_U \omega\wedge {\bf J}w \right| 
& \le  \ 
\ \frac{ C }{|\ln\e|} \int_U  |\omega| \left( \frac{|Dw|^2}2+\frac {F(w)}{\e^2}\right) \nonumber \\
&\quad + C \e^\alpha  (1 +  \| D \omega\|_\infty )  \
\left( 1 + \|\omega\|_\infty + \int_{U} (|\omega|+1)\left( \frac{|Dw|^2}2+\frac {F(w)}{\e^2}\right)\right)
\label{newJacest2}\end{align}
for every compactly supported Lipschitz continuous $1$-form $\omega$ in $U$ and  every $\e\in (0,1]$.
Here ${\bf J}w$ denotes the $2$-form $dw^1 \wedge dw^2 = (w^1_{y^0} dy^0  + d_\nu w^1)
\wedge (w^2_{y^0}dy^0 + d_\nu w^2)$.
\label{L.newJacest2}\end{lemma}

This is Lemma 9 of \cite{j-bec}, with notation adapted to our setting.
In \eqref{newJacest2},  $Dw$ denotes as usual the gradient in all $3$ variables.

\begin{proof}[Proof of Proposition \ref{Pv1}]
{\bf 1}.
Fix  $v\in H^1((0,\tau)\times B_\nu(\rho);\R^2)$. In order to prove
\eqref{Pv1.c1}, we must estimate 
\[
\int_{B_\nu(\rho)} \omega [ J_\nu v(\tau, \yn)  - J_\nu v(0, \yn)] \ d\yn
\]
for an arbitrary $\omega\in C^\infty_c(B_\nu(\rho))$ such that $|\nabla \omega(y)| \le |y|^2$ and $\| \omega\|_{W^{2,\infty}}\le 1$.
We fix such a test function $\omega$, and we start by rewriting the above expression.
For this, let $\delta$ denote a positive number to be fixed below (not to be confused with $\dep$),
and define $V:(-\delta, \tau+\delta)\times B_\nu(\rho)\to \R^2$ by
\[
V(y^0,\yn) =\left\{
\begin{array}{ll} 
v(0,\yn)&\mbox{ if }-\delta <y^0\le 0,\\
v(y^0,\yn)&\mbox{ if } 0\le y^0 \le \tau,\\
v(\tau,\yn)&\mbox{ if } \tau \le y^0 \le \tau+\delta.\end{array}\right.
\]
Let $\chi \in C^\infty_c(-\delta,\tau+\delta)$ be a function such that
\[
\mbox{ $\chi(y^0) \equiv 1$ for $y^0\in [0,\tau]$,\ \ \ \ \ \ \ \ \ \  and
\ \ \ \ $\| \chi'  \|_{\infty } \le C(1+ \delta^{-1})$.}
\]
Since $J_\nu V(y^0)= J_\nu v(0)$  for $y^0\in (-\delta,0]$ and $J_\nu V(y^0) = J_\nu v(\tau)$ for $y^0\in [\tau, \tau+\delta)$, 
\begin{align}
\int_{B_\nu(\rho)  	} \omega [ J_\nu v(\tau, \yn)  - J_\nu v(0, \yn)] \ d\yn
&= 
-\int_{-\delta}^{\tau+\delta}   \chi'(y^0) (\int_{B_\nu(\rho)} \omega(\yn)  \, J_\nu V \, d\yn ) \ d y^0 \label{s2r}\\
&= 
-\int_{(-\delta, \tau+\delta)\times {B_\nu(\rho)}}   (\omega(\yn)\, \chi'(y^0)\, dy^0) \wedge {\bf J}V.\nonumber
\end{align}
We continue by observing that
\[
\omega(\yn) \chi'(y^0) \ dy^0 =  \omega(\yn) d \chi(y^0) = d[ \omega(\yn) \chi(y^0)] - \chi(y^0) d\omega(\yn).
\]
Also, since ${\bf J}V = d(V^1\wedge dV^2)$, it is clear that $d {\bf J}V = 0$, so that $d(\chi\omega) \wedge {\bf J}V = d( \chi\omega \wedge {\bf J}V)$,
and thus the right-hand side of \eqref{s2r} can be rewritten
\begin{align}
-\int_{(-\delta, \tau+\delta)\times {B_\nu(\rho)}}   (\omega\, \chi' \, dy^0) \wedge {\bf J}V
&= 
\int_{(-\delta, \tau+\delta)\times {B_\nu(\rho)}} 
 \chi \, d\omega \wedge {\bf J}V  \nonumber \\
&\quad\quad\quad\quad
-\int_{(-\delta, \tau+\delta)\times {B_\nu(\rho)}}  d(\chi\omega)  \wedge {\bf J}V\nonumber \\
&=
\int_{(-\delta, \tau+\delta)\times {B_\nu(\rho)}} 
\chi \, d\omega \wedge {\bf J}V.
\label{s3r}
\end{align}

{\bf 2}. Properties of $\omega$ and the choice of $\chi$ imply that 
\[
|\chi d\omega(y)| \le |\yn|^2,\quad\quad\quad\quad
\| D (\chi d\omega) \|_{\infty} \le C\delta^{-1}.
\]
It thus follows from Lemma \ref{L.newJacest2} that
\begin{align*}
\left| \int_{(-\delta, \tau+\delta)\times {B_\nu(\rho)}} \chi \, d\omega \wedge {\bf J}V \right|
&\le 
C  |\ln \e|^{-1} \int_{(-\delta, \tau+\delta)\times {B_\nu(\rho)}} |\yn|^2 \left( \frac 12 |DV|^2 + \frac 1{\e^2}F(V)\right) \ d\yn dy^0\\
&
C \e^\alpha (1+\delta^{-1}) \left( 1 + \int_{(-\delta, \tau+\delta)\times {B_\nu(\rho)}} 
\left( \frac 12 |DV|^2 + \frac 1{\e^2}F(V)\right)  d\yn dy^0 \right).
\end{align*}
We now fix $\delta := \e^{\alpha/2}$ and recall the definition of $V$ to find that
\begin{align*}
\left| \int_{(-\delta, \tau+\delta)\times {B_\nu(\rho)}} \chi \, d\omega \wedge {\bf J}V \right|
&\le 
 C  |\ln \e|^{-1} \int_{(0,\tau)\times {B_\nu(\rho)}} (|\yn|^2  + \e^{\alpha/2} )\left(\frac { |v_{y^0}|^2}{2}+ e_{\e,\nu}(v) \right)  d\yn dy^0 \\
&
+
C \e^{\alpha/2}
(1+ 
 \int_{ \{ 0 \} \times {B_\nu(\rho)}} e_{\e,\nu}(v) d\yn
+
 \int_{ \{ \tau \} \times {B_\nu(\rho)}} e_{\e,\nu}(v) d\yn).
\end{align*}
The conclusion of the Lemma now follows by recalling \eqref{s2r} and \eqref{s3r} and renaming $\alpha$.
 \end{proof}

\subsection{proof of Proposition \ref{vlocalest}}

Now we can give the 

\begin{proof}[Proof of Proposition \ref{vlocalest}] As in Proposition \ref{localest} it suffices to consider smooth solutions $v$.  

To simplify we will write $\calD(\cdot)$ and $||| \cdot |||$ instead of $\calD(\cdot; \rho_1/2)$ and 
$|||\cdot |||_{\rho_1/2}$.

{\bf Step 1}.
For simplicity we assume that $s_1=0$. We will use the notation $s_{\max} := \min \{ \rho_1/2c_*, T_1\}$ and 
\[
W_\nu(s) :=B_{\nu}(\rho_1-c_* s ),
\quad\quad\quad\quad
W(s) :=   \T^n \times W_\nu(s).
\]
We  define
\begin{align*}
\zeta_1(s) 
&:=
\dep \int_{  \{s\}\times W(s)} \,  (1 +  \kappa_2|\yn|^2)\,  e_\e(v) \ dy' - 1
\\
\zeta_2(s) 
&: = \calD(v(s)) \\
\zeta_3(s) 
&:=
\dep \int_{  \{s\}\times W(s) }
\left[ |D_\tau v|^2 + |\yn|^2e_{\e,\nu}(v) \right] \ dy' .
\end{align*}
(Recall that  $\kappa_2$ was fixed in \eqref{omega3.def} and that we took $\kappa_2=1$ for $k=1$.) 
We first claim that
\beq
\zeta_1(s)  \le   \zeta_0 + C \int_0^s \zeta_3(\sigma) d\sigma
\quad\quad
\quad\quad
\quad\quad \mbox{ for }0< s \le s_{max}.
\label{vzp.est}\eeq
Indeed, exactly as before we compute that
$\zeta_1'(s)= I_1 -c_* I_2$,
where
\begin{align*} 
I_1 &:=
\dep \int_ {  \{s\}\times W(s)}
 (1 +  \kappa_2|\yn|^2)\, \frac \partial {\partial y^0} e_\e(v) 
 \ dy' \\
I_2& = 
\dep 
\int_{\{s \}\times  \T^n \times \partial W_\nu(s) }
 (1 +   \kappa_2|\yn|^2)\, e_\e(v)  \ d\calH^{N-1}(y') .
\end{align*}
And exactly as before, 
in $I_1$ we use the differential inequality \eqref{eep.prime1} satisfied by the energy   and integrate by parts in the spatial variables. As before, our choice \eqref{cstar.def} of $c_*$  guarantees that the boundary term that arises,
involving  an integral over $ \{ s \} \times \T^n \times \partial W_\nu(s)$, is dominated by $- c_*I_2$.
This leads as before to the differential inequality
\[
\zeta_1' \le C\zeta_3.
\]
Since our assumption \eqref{vzeta0.h1} exactly states that
$\zeta_1(0) \le \zeta_0$, we conclude that   \eqref{vzp.est} holds.

{\bf Step 2}. Next, we estimate $\zeta_2$. It is clear that $|||\cdot |||$ is a norm, so that 
$\calD_\nu(v(s, \yt')) \le \calD_\nu(v(0,\yt')) + |||J_\nu v(s,\yt') - J_\nu v(0,\yt')|||$ for every $(s,\yt')$, by the triangle inequality. It follows that
\begin{align*}
\zeta_2(s)
&\le
\calD(v(0)) + \int_{\T^n} ||| J_\nu v(0, \yt')  - J_\nu v(s, \yt')||| d\yt'\\
&
\overset{\eqref{vzeta0.h2}, \eqref{Pv1.c1}}{\le}
\zeta_0+
C\dep \int_{\T^n} \int_{(0,s)\times B_\nu(\rho_1/2)}  |D_\tau v|^2 + (|\yn|^2+ \e^\alpha)e_{\e,\nu}(v)   d\yn dy^0  \ d{\yt}'\\
&\quad\quad+
C\e^\alpha
+ C \e^\alpha \int_{\T^n}\left( \int_{\{0\}\times B_\nu(\rho_1/2)}  e_{\e,\nu}(v) d\yn  \ + \ 
 \int_{\{s\}\times B_\nu(\rho_1/2)}  e_{\e,\nu}(v) d\yn   \right)
 \ d{\yt}'  . \ 
\end{align*}
Also, since $B_\nu(\rho_1/2)\subset W_\nu(s)$ for every $s\le \rho_0/2c_* $, the definitions yield
\[
\int_{\T^n}\int_{\{s\}\times B_\nu(\rho_1/2)}  e_\e(v) d\yn  \, d{\yt}' \le  C\dep^{-1}( \zeta_1(s) + 1) \le C|\ln \e|(\zeta_1(s)+1) ,
\]
and similarly for $s=0$.
By combining these and rearranging we find that if $0\le s \le s_{max}$, then
\beq
\zeta_2(s) \ \le  \zeta_0+\ C \int_0^s \left[ \zeta_3(\sigma)  +  \e^\alpha (\zeta_1(\sigma) + C)\right] \ d\sigma + C\e^{\alpha}
+ C\e^{\alpha/2}(\zeta_0 + \zeta_1(s) + C).
\label{vzeta2.est}\eeq

{\bf Step 3}.
Finally, we show (by {\em exactly} the same arguments as in the corresponding point of the proof of Proposition \ref{localest}) that 
\beq
\zeta_3(s) \le
C\left( \zeta_1(s)  \ + \   \zeta_2(s) \ + \  |\ln\e|^{-1}   \right)
\label{vzeta3.est}\eeq 
for every $s\in [0,s_{max}]$. 
We fix such an $s$, and we write $v(\cdot)$ instead of $v(s,\cdot)$.
It follows from  the definitions of $\zeta_1,\zeta_3$ and the choice  \eqref{omega3.def}
of $\kappa_2$  that
\beq
\zeta_1(s) \ge  c \  \zeta_3(s)
+
\ \  \dep \int_{ \{s\}\times W(\rho_1/2c_*)} \  e_{\e,\nu}(v)  \ dy' \ - \  1.
\label{vzeta2.e1}\eeq
We say that a point ${\yt}' \in \T^n$ is {\em good } if 
$
\calD_\nu(v(\yt'))  \le \kappa_3
$
and bad otherwise. Then Chebyshev's inequality and \eqref{L1|||} imply that
$
| \{  {\yt}'\in \T^n\ : {\yt}'\mbox{ is good } \}| \ge 1 - C \zeta_2(s)$,
and exactly as in \eqref{z3final}, but appealing to Proposition \ref{Pv2} instead of Lemma \ref{L.mm2},
we infer that
\[
\dep\
 \int_{\{ s \}\times W(\rho_1/2c_*)}   e_{\e,\nu}(v) \ d y' 
\ \ge \ 
\left( 1 -C\zeta_2(s)
 \right)(1- C|\ln\e|^{-1} ).
\]
Combining this inequality with \eqref{vzeta2.e1}, we obtain \eqref{vzeta3.est}.

{\bf Step 4}.
By combining the previous few steps, we see that
\[
\zeta_3(s) \le C \zeta_0 + C |\ln \e|^{-1} + C \int_0^s\zeta_3(\sigma) d\sigma + 
C\e^\alpha \int_0^s \int_0^\sigma \zeta_3(t)  \ dt \ d\sigma. 
\]
If we define $\zeta_4(s) := \zeta_3(s) +   \zeta_0+|\ln\e|^{-1} +\e^\alpha\int_0^s \zeta_3(\sigma)\ d\sigma$, it follows
(since $\zeta_0\ge \dep$) that
\[
\zeta_4(s) \le C \int_0^s\zeta_4(\sigma)\ d\sigma\quad\forall s\in [0,s_{max}] , \quad\quad \zeta_4(0) \le C \zeta_0.
\]
Gronwall's inequality then implies that
$\zeta_4(s)\le C\zeta_0$
for all $s\in [0,s_{max}]$. The conclusions of the proposition follow from this together with
\eqref{vzp.est} and \eqref{vzeta2.est}.
\end{proof}

\subsection{Proof of Proposition \ref{P.vinitial}} \label{S:vinitial}

Finally, we present the proof of Proposition \ref{P.vinitial}. We use notation such as $V^*(s), \partial_i V^*(s)$ and so on, from Section \ref{S:nonzero}.

\begin{proof}
As usual, we may assume by an approximation argument, relying on standard well-posedness theory for \eqref{v.eqn}, that $v$ is smooth on $\bar V$. Define $v^*$ as in \eqref{vstar.def}, and define
\begin{align*}
\zeta_1(s)
&=
\dep \int_{  \partial_1V^*(s) }
(1+\kappa_2|\yn|^2)  e_\e(v^*)
 dy' - 1 
\\
\zeta_2(s)
&=
\calD(v^*(s) ; \rho_1/2) \\
\zeta_3(s)
&=
\dep\int_{ \partial_1V^* (s) }
\left[    |D_\tau v^*|^2 +  |\yn|^2 e_{\e,\nu}(v^*) \right] \ dy' .
\end{align*}
We repeat exactly  the arguments of Proposition \ref{vlocalest} to find that
\[
\zeta_2(s) \ \le \ C \int_{s_0}^s \zeta_3(\sigma)  +  \e^\alpha (\zeta_1(\sigma) + C) \ d\sigma + C\e^\alpha
+ C\e^{\alpha/2}(\zeta_0 + \zeta_1(s) + C)
\]
and
\[
\zeta_3(s) \le
C\left( \zeta_1(s)  \ + \   \zeta_2(s) \ + \  |\ln\e|^{-1}   \right).
\]
To estimate $\zeta_1$, we argue as in the proof of Proposition \ref{glocalest}.
That is, we apply the divergence theorem to
\[
\int_{V(s)} \partial_{y^\alpha} \left[(1+ \kappa_2|\yn|^2)\tilde T^\alpha_{\e, 0}\right],
\]
where
$\tilde T^\alpha_{\e,\beta}(y) := 
\delta^\alpha_\beta ( \frac \e 2 g^{\gamma\delta}v_{y^\gamma}\cdot v_{y^\delta} + \frac 1\e F(v)) - 
\e \ g^{\alpha \gamma} v_{y^\gamma}\cdot v_{y^\beta}
$
and we rewrite, noting that $n_\alpha(y)\tilde T^\alpha_{\e,0}(y) \ge 0$ 
for a.e. $y\in \partial_2V(s)$  exactly as before. This eventually yields
\[
\zeta_1(s) \le
 C \int_{s_0}^s \zeta_3(\sigma) \, d\sigma  + A + B,
\]
for a.e. $s\in [s_0, s_1]$,
where
\begin{align*}
A 
&:=
\dep \int_{W_0(s)}(1+\kappa_2|\yn|^2)( e_\e(v) - e_\e(v_0)) (b(y'), y') dy'
+ \dep \int_{  \partial_0 V(s)} (1 + \kappa_2| \yn|^2)\, n_i \vp^i  \ d\calH^N, \\
 \\
B
&:=
\dep
\int_{[W_1^*(s)\setminus W_1(s)] \cup W_0(s) }  (1 + \kappa_2| \yn|^2) e_\e(v_0) dy'   -1.
\end{align*}
We proceed exactly as in the proof of Proposition \ref{glocalest},  using 
Lemmas \ref{b.def} and  \ref{L.cone}, the hypotheses \eqref{idata2} -- \eqref{idata4}, and elementary
arguments to
show that
$A \le C \zeta_0$ and $B \le C \zeta_0$
for a.e. $s\in [s_0, s_1]$, and hence that $\zeta_1 \le C \int_{s_0}^s \zeta_3(\sigma) d\sigma + C \zeta_0$.

The proof is now finished exactly as in Step 4 of the proof of Proposition \ref{vlocalest}.
\end{proof}


%

\section{Proof of Theorems  \ref{T1} and \ref{T2}}\label{S:iterate}

In this section we combine the estimates proved in the previous sections with standard
energy
estimates in the original $(t,x)$ variables, iterate, and harvest consequences, to 
complete the proofs of our main results. We mostly give a unified treatment of 
the cases $k=1,2$. 
To distinguish between the relevant energy densities in the $(t,x)$ and the $y$ variables, in this section we will
often use the notation $e_\e(u;\eta)$ and $e_\e(v;G)$, see \eqref{eep.def1} and the following discussion.

The following theorem assembles most of our main estimates and will easily imply Theorems \ref{T1} and \ref{T2}; in can be seen as the main result of this paper. In it, and throughout this section, when we write $C(\Gamma, T_0)$, it will denote a constant that may depend upon various choices  made in the construction   \eqref{psi.def} of the map $\psi$ that we use to change variables; these choices however are constrained only by $\Gamma$ and $T_0$.

\begin{theorem} Let $k=1$ or $2$, $n\ge 1$, and $N = n+k$. 

Let $\Gamma\subset (-T,T)\times \R^N$ be a smooth timelike Minkowski minimal surface of 
codimension $k$ satisfying our standing assumptions 
\eqref{Gamma.h1}- \eqref{Gamma.h3}.
Let $u:(-T,T)\times \R^N \to \R^k$ solve \eqref{slw} with initial
data satisfying assumpions \eqref{idata1}, \eqref{idata2} - \eqref{idata4}, for some $\zeta_0$
verifying \eqref{z0gdep}.

Given $T_0<T$, fix $T_1 \in (T_0, T)$ and $\rho_0>0$ so small that \eqref{rho0.1}, \eqref{T1rho} and
the conclusions of Proposition \ref{P.transformation} hold 
on $(-T_1,T_1)\times \T^n\times B_\nu(\rho_0)$.

Then there exists a constant $C(\Gamma, T_0)$ such that 
\beq
\delta_\e \int_{ [ (-T_0,T_0)\times \R^N] \setminus \calN }  \ \  \ e_\e(u;\eta) \ dx\, dt 
\le C \zeta_0 ,
\label{iterate.c1}\eeq
for $\calN = \mbox{image}(\psi) \cap [(-T_0,T_0)\times \R^N] $; 
and such that $v = u\circ \psi$ satisfies
\beq
\dep 
\int_{(-T_1,T_1)\times \T^n\times B_\nu(\rho_0)}  |D_\tau v|^2 + |\yn|^2( |\nabla_\nu v|^2 + \frac 1{\e^2}F(v))\ dy \le C \zeta_0,
\label{iterate.c2}\eeq
\beq
\dep \int_{(-T_1,T_1)\times \T^n\times B_\nu(\rho_0)}  \left[(1+ \kappa_2|\yn|^2) e_\e(v;G)\right] dy - 
\calH^{1+n}
( (-T_1,T_1)\times \T^n)
\le C \zeta_0,
\label{iterate.c3}\eeq
(for $\kappa_2$ as in \eqref{omega3.def}, with $\kappa_2=1$ when $k=1$)
and
\beq
\int_{(-T_1,T_1)\times \T^n}  \calD_\nu( v(\yt); \rho_1/2 ) \ d\yt 
 \ \le C \zeta_0  
\label{iterate.c4}\eeq
where $\calD_\nu$ was defined in \eqref{calDnu.k1def} for $k=1$ and \eqref{calDnu.k2def}
for $k=2$,
and $\rho_1$ was found in Lemma \ref{L.cone}.
Finally, 
\beq
\left\|\delta_\e  \calT_\e(u) - \calT(\Gamma) \right\|_{W^{-1,1}((-T_0,T_0)\times \R^N)} \le C \sqrt {\zeta_0}.
\label{iterate.c5}\eeq
\label{T3}\end{theorem}

The following lemma will be used  repeatedly.

\begin{lemma}
There exists a constant $C>0$, depending only $\Gamma, T_1, \rho_0$, such that 
\[
\frac 1C e_\e(u;\eta)(\psi(y)) \le \ e_\e(v;G)(y)  \ \le \ C e_\e(u;\eta)(\psi(y))
\]
and 
\[
\frac 1C \le |\det D\psi(y)| = \sqrt{-g(y)}   \le C 
\]
for all $y\in (-T_1,T_1)\times \T^n\times B_\nu(\rho_0)$.
\label{L.cuv1}\end{lemma}

\begin{proof}This is clear from the construction of the diffeomorphism $\psi$, see in particular \eqref{pos1}. 
\end{proof}

Next, we show that Theorems \ref{T1} and \ref{T2} follow from
directly from Theorem \ref{T3} and the above Lemma. The rest of this section
will then be devoted to the proof of Theorem \ref{T3}.

\begin{proof}[Proofs of Theorems \ref{T1} and \ref{T2}]
First we consider the scalar case, ie that of Theorem \ref{T1}. We define $\calN$
as in \eqref{calN.def}. Then as noted in Corollary \ref{L.d}, the function
$d$ defined by \eqref{d.inv} satisfies the eikonal equation \eqref{d1a} in
$\calN$ as required.

Let $u$ solve \eqref{slw} with the initial data given by 
Lemma \ref{L.z0}, in the case $k=1$, so that it satisfies the assumptions
of Theorem \ref{T3} with $\zeta_0 = C\e^2$, and in addition
\beq
\int_{\T^n\times B_\nu(\rho_0)} (v_0 - q(\frac {y^N} \e))^2 dy' \le C \e
\label{v000}\eeq
for $v_0$ as defined in \eqref{v0.def}.

Then conclusion \eqref{t1.c3} of Theorem \ref{T1} is exactly \eqref{iterate.c5}.

To prove \eqref{t1.c2}, we recall that $\calN\subset \psi ((-T_1,T_1)\times \T^n\times B_\nu(\rho_0))$, and
we use \eqref{d.inv} and Lemma \ref{L.cuv1} to estimate
\begin{align*}
&\dep \int_{\calN} d^2 e_\e(u;\eta) \,dx\,dt
\le
C\dep \int_{(-T_1,T_1)\times \T^n\times B_\nu(\rho_0)}( \yn)^2 e_\e(v;G) \,dx\,dt\\
&\quad\quad\quad\overset{\eqref{iterate.c3}}\le 
C \zeta_0 - \left[  
\dep \int_{(-T_1,T_1)\times \T^n\times B_\nu(\rho_0)}   e_\e(v;G) dy - 
\calH^{1+n}
( (-T_1,T_1)\times \T^n)
\right]
\end{align*}
Next, by using Lemma \ref{L.mm2} and arguing exactly as in the the proof of 
\eqref{z3.sts}, we see that
\begin{align}
&
2T_1 - \dep \int_{(-T_1,T_1)\times \T^n\times B_\nu(\rho_0)}   e_\e(v;G) dy 
\le C
\int_{(-T_1,T_1)\times \T^n} \calD_\nu(v(\yt))d\yt + C e^{-c/\e}.
\label{eab}\end{align}
The above inequalities and \eqref{iterate.c4} imply that $\dep \int_{\calN} d^2 e_\e(u;\eta) \,dx\,dt \le C \e^2$. By combining this with \eqref{iterate.c1}, we obtain \eqref{t1.c2}.

Finally, to prove \eqref{t1.c1}, note that for every $y'\in \T^n\times B_\nu(\rho_1)$ and every $y^0\in (-T_1,T_1)$,
\[
|v(y^0, y') - v_0(y')| = 
|v(y^0, y') - v(b(y'), y')| 
\le |y^0 - b(y')|^{1/2} \left(\int_{-T_1}^{T_1} |\partial_{y^0}v(s, y')|^2 ds\right)^{1/2}.
\]
Since $ |\partial_{y^0}v| \le |D_\tau v|$, we find by integrating that
\[
\int_{(-T_1,T_1)\times \T^n\times B_\nu(\rho_0)} |v(y^0, y') - v_0(y')|^2 \,dy 
\ \le \  
C \int_{(-T_1,T_1)\times \T^n\times B_\nu(\rho_0)}  |D_\tau v|^2 dy
\overset{\eqref{iterate.c2}}\le C\e
\]
using the fact that $\zeta_0\le C\e^2$. Then \eqref{v000} implies that 
\[
\int_{(-T_1,T_1)\times \T^n\times B_\nu(\rho_0)} |v(y) - q(y^N/\e)|^2 \,dy  \le C \e.
\]
By changing variables, using Lemma \ref{L.cuv1} and recalling \eqref{d.inv}, we obtain
\eqref{t1.c1}.

The proof of Theorem \ref{T2} is essentially the same, except that 
we do not make any claim about  $\int  |v(y^0, y') - v_0(y')|^2$, as the estimate
$\int |\partial_{y^0}v|^2 dy \le C$ is 
too weak to provide good control over this quantity. 

Otherwise, we follow the above proof. 
That is, we
let $u$ solve \eqref{slw} with the initial data given by 
Lemma \ref{L.z0}, in the case $k=2$, so that it satisfies the assumptions
of Theorem \ref{T3} with $\zeta_0 = C|\ln\e|^{-1}$.
Then conclusion \eqref{t2.c1} of Theorem \ref{T2} is exactly \eqref{iterate.c5}.
To prove \eqref{t2.c2}, it suffices in view of \eqref{iterate.c1} to prove that
\[
\int_{\calN} \dist(\cdot, \Gamma)^2 e_\e(u;\eta) \,dx\,dt \le C.
\]
Using Lemma \ref{L.cuv1} to change variables, and noting that $\dist(\psi(y), \Gamma)^2 \le C |\yn|^2$
(since the left-hand side is a smooth function of $y$ that vanishes when $\yn = 0$) it suffices to prove that
\[
\int_{(-T_1,T_1)\times\T^n\times B_\nu(\rho_0)} |\yn|^2  e_\e(v;G) \, dy \  \le  \  C.
\]
This exactly follows the proof of \eqref{t1.c2} in the case $k=1$ above. The estimate corresponding to \eqref{eab} has {\em exactly} the same form,
except that the last term on the right-hand side is now $C|\ln \e|^{-1}$; this is proved by arguing as before but using Proposition \ref{Pv2} in place of Lemma \ref{L.mm2}. 
\end{proof}

The proof of Theorem \ref{T3} will use some standard energy estimates that we now recall.

\begin{lemma}
Let $u:\R^{1+N}\to \R^k$ be a smooth, finite-energy solution of the semilinear wave equation \eqref{slw}. For any $a<b$ and any bounded Lipschitz function $\chi:\R^{1+N}\to \R$
\beq
\left|
 \int_{\{b\}\times \R^N} e_\e(u;\eta) \chi  \ dx -
 \int_{\{a\}\times \R^N} e_\e(u; \eta) \chi  \ dx
\right| \ \le \int_{(a,b) \times \R^N} e_\e(u;\eta) |D\chi| \ dx \ dt.
\label{ee1}\eeq
Also,  for any pair $a,b$ of real numbers and any open $A\subset\R^N$
\beq
\int_{\{b\}\times A_{|b-a|}} e_\e(u;\eta) \ dx \ \le \  \int_{\{a\}\times A} e_\e(u;\eta) \ dx,
\label{ee2}\eeq
where 
$
A_s:= \{x\in A : \dist(x,\partial A)>s\}$.
\label{standard.ee1}\end{lemma}

\begin{proof} 
Both conclusions are standard and follow from
the identity $\partial_t e_\e(u;\eta) = \nabla\cdot(u_t \nabla u)$, satisfied by solutions
of \eqref{slw};  integration by parts; and the elementary inequality $|u_t\nabla u|\le e_\e(u;\eta)$. For the second inequality, assuming for concreteness that $a<b$, 
it is easy to see that  the set $\{(t,x) : a<t<b, x\in A_{t-a} \}$ is a set of finite perimeter, so that the divergence theorem holds and there is no problem in justifying the standard argument.
\end{proof}

\begin{lemma}
Let $v:(-T_1,T_1)\times \T^n\times B_\nu(\rho_0)\to \R^k$ be a smooth solution of \eqref{v.eqn}. Then for any 
$-T_1\le a < b\le T_1$ and any $\chi\in W^{1,\infty}_0((-T_1,T_1)\times \T^n\times B_\nu(\rho_0))$,
\begin{align*}
&\left|
\int_{\{b\}\times \R^N} e_\e(v;G) \chi  \ dy' -
\int_{\{a\}\times \R^N} e_\e(v; G) \chi  \ dy'
\right|  
 \le C \int_{(a,b) \times \R^N}
 e_\e(v;G) \left (|\chi| +   |D\chi| \right)
 \ dy .
\end{align*}
\label{standard.ee2}\end{lemma}

\begin{proof}
Lemma \ref{L.eflux} implies that $\partial_{y^0} e_\e(v;G ) 
\le C e_\e(v;G) +\nabla\cdot \vp$,
and the positivity \eqref{A.def} of the matrix  $(a^{\alpha\beta})$, together with the definition \eqref{vp.def}
of $\vp$, 
implies that $|\varphi| \le C e_\e(v;G)$. The conclusion follows from these facts together with
integration by parts, exactly as in the previous lemma.
\end{proof}

Now we present the

\begin{proof}[proof of Theorem \ref{T3}]
We treat both cases $k=1$ and $2$ simultaneously. We may assume as usual that $u$ and $v = u\circ \psi$, are smooth.

It is convenient to define $\rho:(-T_0,T_0)\times \R^N\to [0,+\infty]$ by
\[
\rho(t,x) = \begin{cases}
|\yn|&\mbox{ if }(t,x) = \psi(y) \\
+\infty&\mbox{ if }(t,x)\not\in \calN = \mbox{image}(\psi).
\end{cases}
\]
Note that when $k=1$, $\rho(t,x) = |d(t,x)|$ for $(t,x)\in \calN$.

{\bf Step 1}. Given $u:(-T,T)\times \R^N\to \R^k$, $k=1$ or $2$, solving \eqref{slw}, we will say that $u$ is {\em controlled} on a set $W\subset  (-T_0,T_0)\times \R^N$ if 
there exists a constant $C$, depending on $W, \Gamma,\psi$, such that for any function $u$ satisfying the hypotheses of Theorem \ref{T3}, 
\[
\int_W e_\e(u;\eta)  \le C \zeta_0.
\]
(We will only say this about sets that are bounded away from $\Gamma$). If $W$ is an open set,
the integral is understood as $\int  \cdots dx dt$, and if $W$ is a subset of some $\{ t\}\times \R^N$
then it is understood as $\int\cdots dx$.

Similarly, for a set $W\subset (-T_1,T_1)\times \T^n\times B_\nu(\rho_0)$, 
we say that $v = u\circ \psi$ is  controlled on $W$
if there exists a constant $C = C(W, \Gamma, T_0)$ such that
\[
\dep \int_{W} \ \left[ |D_\tau v|^2 + |\yn|^2( |\nabla_\nu v|^2 + \frac 1{\e^2}F(v)) \right]\ \le C \zeta_0,
\]
Again $W$ may be either an open set or a subset of $\{s\}\times \T^n\times B_\nu(\rho_0)$ for some $s$,
and the integral is understood accordingly.

We make some easy remarks.
First, if $v$ is controlled on a set  $W$,
then since
\[
e_\e(v;G) \le C(\hat \rho) \left[|D_\tau v|^2 + |\yn|^2( |\nabla_\nu v|^2 + \frac 1{\e^2}F(v)) \right]
\quad\mbox{ whenever }|\yn|\ge \hat \rho,
\]
it follows that for any $\hat \rho\in (0, \rho_0)$,  $\int_{\{ y = (\yt, \yn) \in W :  |\yn|\ge \hat \rho \}}e_\e(v;G) \le C\zeta_0$.

As a result, Lemma \ref{L.cuv1} implies that  
if $A\subset \mbox{Image}(\psi)$ is bounded
away from $\Gamma$, then $u$ is controlled on $A$ if and only if $v$ is controlled on $\psi^{-1}(A)$,

Finally, we remark that for any $\hat \rho>0$, the assumptions \eqref{idata1}, \eqref{idata2}, \eqref{idata3} and  a change of variables imply that  
\beq
\mbox{ $u$ is controlled on } \ \{ (0,x)\in \R^{1+N} : \ \rho(0,x)\ge  \hat \rho\}
\label{ucont0}\eeq
with the implicit constants depending on $\hat \rho$ and $\psi$, or more precisely on
the behavior of $\psi$ on $\{ y \in (-T_1,T_1)\times \T^n\times B_\nu(\rho_0) : \psi^0(y)=0\}$.

{\bf Step 2}. We next claim that for $0 < s_+' < s_+\le T_1$ and $\rho' < \rho_1/2$,
\begin{align}
\mbox{ if $v$ is controlled}&\mbox{ on }
\{ y\in (-T_1, s_+)\times \T^n\times B_\nu({\rho_1}/2) : \psi^0(y)>0\}
\nonumber\\
&\mbox{then for every $t\in [0, s_+']$,  $u$ is controlled on }\{ (t,x) \ :  \ \rho(t,x) >\rho' \},
\label{cuv2}\end{align}
and this control is uniform for $t\in [0, s_+']$.
To prove this, we fix $\hat \rho>0$ so small that  $\hat \rho \le \rho'$ and
\begin{align*}
\{ (t,x) :0 < t < s_+', \rho(t,x) <\hat \rho \} 
&=
\{ \psi(y) : y\in (-T_1,T_1)\times \T^n\times B_\nu(\hat \rho) , \ \  0 < \psi^0(y) < s_+' \} 
\\
&\subset 
\{  \psi(y) : y\in  (-T_1, s_+)\times \T^n\times B_\nu(\hat \rho) : \psi^0(y)>0
 \}.
\end{align*}
The point is that if $|\yn|$ is small enough and $\psi^0(y) <s'_+$, then
$y^0< s_+$. Such a number $\hat \rho$ exists 
because $|\psi^0(y^0, {\yt}', \yn) - y^0| \le C |\yn|$; this
is an easy consequence of the definition of $\psi$.  

Then it follows from the control of $v$ and 
Step 1 that
\beq
\mbox{ $u$ is controlled on } \ \{ (t,x) :\   0 < t < s_+', \  \ \frac {\hat \rho}2 <\rho(t,x) < \hat \rho \ \} ,
\label{ucont1}\eeq
since this set is bounded away from $\Gamma$ and is contained, by the choice of $\hat \rho$,
in the image via $\psi$ of a set
on which we have assumed that $v$ is controlled. 

Now we fix a function $\chi\in C^\infty([0, s_+']\times \R^N)$
such that 
$\chi = 1$ wherever $\rho(t,x)> \hat \rho$,
and $\chi = 0$ where $\rho(t,x) \le \hat \rho/2$.
Then
we apply \eqref{ee1} with this choice of $\chi$
and with $a=0$ and $b \in (0, s_+')$, and using \eqref{ucont0} and \eqref{ucont1}, 
we find that $u$ is controlled on $\{ (b,x) \ : \rho(b,x) > \hat \rho \}$,
with implicit constants that are uniform for $b\in  (0, s_+']$.
Thus we have proved \eqref{cuv2}.

{\bf Step 3}. We next claim that
for $0 < s_+\le T_1$ as above,
\begin{align}
\mbox{ if $v$ is controlled}&\mbox{ on }
\{ y\in (-T_1, s_+)\times \T^n\times B_\nu({\rho_1}/2) : \psi^0(y)>0\}
\nonumber\\
&\mbox{ then $v$ is controlled on }
\{s_+\} \times \T^n\times  [B_\nu(\rho_1) \setminus B_\nu(\frac{\rho_1}2)] .
\label{cuv3}\end{align}
%
We first apply \eqref{cuv2}, with parameters $s_+'<s_+$ and $\rho'< \frac 12 \rho_1$
to be fixed below.
We then apply \eqref{ee2} with $ a =s_+'$ and  $|b|\le T$
to conclude that $u$ is controlled on 
\[
S(s_+', \rho') := 
 \{ (t,x) : \ |t|\le T, \  \dist(x,A(s_+',\rho')) >  |t - s_+'|\}
\]
where
\[
A(s_+',\rho') := \{ x \ : \rho(s_+',x) > \rho' \}.
\]
Then Step 1 implies that $v$ is controlled on $\psi^{-1}(S(s_+', \rho'))$.

Below we will show that we can fix
$s_+'<s_+$ and $\rho'>0$ such that
\beq
\psi\left(\{s_+\} \times \T^n\times [B_\nu(\rho_0) \setminus B_\nu(\frac{\rho_1}4)] \right)  \subset\subset
S(s_+', \rho'),
\label{cuv4}\eeq
by which we mean that some open neighborhood of $\psi(\cdots)$ is contained in $S(s_+',\rho')$.
For now we assume that we have selected $s_+'$ and $\rho'$ so that \eqref{cuv4} holds, and we
complete the proof of \eqref{cuv3}. Indeed, if \eqref{cuv4} holds, then clearly
\[
\{s_+\} \times \T^n\times \left[B_\nu(\rho_0)  \setminus B_\nu(\frac{\rho_1}4)\right] \subset\subset
\psi^{-1}(S(s_+', \rho')).
\]
and so there exists some $a< s_+$ such that 
$(a,s_+) \times \T^n\times \left[B_\nu(\rho_0) \setminus B_\nu(\frac{\rho_1}4)\right] \subset\subset
\psi^{-1}(S(s_+', \rho'))$.
Now we can find some
and some smooth nonnegative function $\chi$ such that 
\begin{align*}
&\chi = 
1 \mbox{ on }
\{s_+\} \times \T^n\times \left[\bar B_\nu(\rho_1)  \setminus B_\nu(\frac{\rho_1}2)\right], \mbox{ and }\\
&\mbox{spt}(\chi)
\subset
(a, T_1)\times \T^n\times \left[B_\nu(\rho_0) \setminus B_\nu(\frac{\rho_1}4)\right].
\end{align*}
Since $v$ is controlled in $\psi^{-1}(S(s_+', \rho'))$, \eqref{cuv3} follows from applying Lemma \ref{standard.ee2} with this choice of $\chi$ and $a$ and with $b=s_+$.

{\bf Step 4}. We next verify \eqref{cuv4}.
Since the sets $S(s,\rho)$ depend continuously on $s$ and $\rho$ in an obvious way,
\eqref{cuv4} will follow (for suitable $s_+'<s_+$ and $\rho'>0$)  if we can show that 
\beq
\psi\left(\{s_+\} \times \T^n\times (B_\nu(\rho_0) \setminus B_\nu(\frac{\rho_1}4)]\right)  \subset\subset
S(s_+, 0).
\label{cuv4a}\eeq
We will deduce this as a consequence of the following fact: if $\Sigma$ is a connected spacelike hypersurface 
in $1+N$-dimensional Minkowski space,  and if we define the solid light
cone with vertex $(t,x)$ to be 
\[
LC(t,x) := \{ (t', x') : |x-x'| \ge |t-t'| \}
\]
then $LC(t,x) \cap \Sigma = \{ (t,x)\}$, for every $(t,x)\in \Sigma$.

To reduce \eqref{cuv4a} to this geometric fact, we define 
\[
\Sigma := \psi\left(\{s_+\} \times \T^n\times B_\nu(\rho_0)\right).
\]
Clearly $\Sigma$ is a connected hypersurface. We claim that it is also spacelike.
To see this, recall (see \eqref{pos1}) that $(g_{ij})_{i,j=1}^N$ is positive definite; this implies
that $\psi^{-1}(\Sigma) = \{s_+\} \times \T^n\times B_\nu(\rho_0)$ is spacelike with respect to the $(g_{\alpha \beta})$ metric.
The claim then follows, since $\psi$ is an isometry between $(-T_1,T_1)\times \T^n\times B_\nu(\rho_0)$ with the $(g_{\alpha \beta})$ metric and $\mbox{Image}(\psi)\subset\R^{1+N}$ with the Minkowski metric in standard form $ds^2 = -dt^2 + (dx^1)^2 + \ldots + (dx^N)^2$.

Next note that the definition of $\psi$ and the choice \eqref{rho0.1} of $\rho_0$ 
imply that 
\[
\psi\left(\{s_+\} \times \T^n\times [B_\nu(\rho_0)  \setminus B_\nu(\frac{\rho_1}4)]\right)  \subset\subset
(-T,T)\times \R^N.
\]
Thus,  in order to prove \eqref{cuv4a} it suffices to show that the closure of
$\psi\left(\{s_+\} \times \T^n\times [B_\nu(\rho_0)  \setminus B_\nu(\frac{\rho_1}4)]\right)$
does not intersect
$[(-T,T)\times \R^N] \setminus S(s_+,0)$. 
However, by inspection of the definition of $S(s,\rho)$ 
one sees that
\begin{align*}
[(-T,T)\times \R^N] \setminus S(s_+,0) \ 
&\subset \ \cup_{  \{( x\in \R^N : \rho(s_+,x)=0\} } LC(s_+,x) \\
& \ = \ \cup_{  \{x\in \R^N : (s_+,x)\in \Gamma \} } LC(s_+,x).
\end{align*}
In addition, since $ \Gamma\cap (\{s_+)\times \R^N) = \psi(\{s_+\}\times \T^n\times\{0\})$,
it is clear that
\[
\Sigma 
\ \ \supset
\ \ 
\Gamma\cap (\{s_+)\times \R^N).
\]
Then the geometric fact mentioned above implies that
\begin{align*}
\Sigma \cap
\left([(-T,T)\times \R^N] \setminus S(s_+,0)\right) \ 
&\ \subset \ 
\ \cup_{  \{x : (s_+,x)\in \Gamma \} } \Sigma \cap LC(s_+,x) \\
&\ = \ 
\Gamma\cap (\{s_+)\times \R^N).
\\
&\ = \ \psi(\{s_+\}\times \T^n\times\{0\}).
\end{align*}
Since $\psi$ is injective, this implies that $\psi( \{s_+\}\times \T^n\times [B_\nu(\rho_0)\setminus \{0\}])$
does not intersect $[(-T,T)\times \R^N] \setminus S(s_+,0)$, completing the proof of \eqref{cuv4a}.

{\bf Step 5}. 
Next we introduce more terminology.
For a set $W\subset (-T_1,T_1)\times \T^n\times B_\nu(\rho_0)$, 
if there exists $W_\tau\subset (-T_1,T_1)\times \T^n$ such that 
\[
W_\tau\times B_\nu(\rho_1/2) \subset W \subset W_\tau\times B_\nu(\rho_0),
\]
then we say that $v$ is  {\em completely controlled} on $W$ if $v$ is controlled on $W$,
and addition there exists a constant $C(W,\psi)$ 
such that
\[
\dep \int_{W}  \  \left[(1+\kappa_2 |\yn|^2) e_\e(v;G)\right]  - 
\calH^{\dim W_\tau}
( W_\tau)
\le C \zeta_0,
\]
and
\[
\int_{W_\tau}\calD_\nu( v(\yt))\ \le C \zeta_0 .
\]
And as above, we allow $W$ to be either an open set or a subset of some $\{ y^0 = \mbox{const}\}$ slice, with the integral understood accordingly, and with $\dim W_\tau = 1+n$ in the first case and  $n$ in the second.

Now we establish estimates \eqref{iterate.c1} - \eqref{iterate.c4}.
First, by assumption $v=u\circ\psi$ satisfies the hypotheses of Proposition
\ref{glocalest} (if $k=1$) and 
Proposition \ref{P.vinitial} (if $k=2$) and these imply that 
\begin{align}
&\mbox{$v$ is controlled on 
$\{y\in (-T_1, s_1)\times \T^n\times B_\nu(\rho_1): \psi^0(y)>0\}$,  \ \ \ and }
\label{cuv5a}\\
&
\mbox{ $v$ is completely controlled on  $\{ s_1\} \times \T^n \times B_\nu(\rho_1)$.}
\label{cuv5b}
\end{align}
In particular, \eqref{cuv5b} implies that $v$ satisfies the hypotheses of Proposition
\ref{localest} ($k=1$) or Proposition \ref{vlocalest} ($k=2$), with $\zeta_0$ replaced by $C\zeta_0$. These propositions assert that 
\beq
\mbox{$v$ is completely controlled on $\{ s \} \times \T^n \times B_\nu(\rho_1/2)$,}\quad
\mbox{
for $s_1 \le s \le s_2  $.}
\label{cuv5c}\eeq
where $s_2 :=   \min\{ T_1, s_1+ \rho_1/2c_*\} $.
Then \eqref{cuv5a} and \eqref{cuv5c} imply that $v$ is controlled on 
$\{y\in (-T_1, s_2)\times \T^n\times B_\nu(\rho_1/2): \psi^0(y)>0\}$.
Next we invoke  \eqref{cuv3} to find that $v$ is controlled on 
$ \{ s_2\} \times \T^n \times [B_\nu(\rho_1) \setminus B_\nu(\rho_1/2)]$.
Hence, appealing again to \eqref{cuv5c}, we see that
 $v$ is completely controlled on  $ \{ s_2\} \times \T^n \times B_\nu(\rho_1)$.

Thus we can apply Proposition \ref{localest} or \ref{vlocalest} with $s_1$ replaced by $s_2$ and $\zeta_0$ multiplied by a suitable constant, but with the same fixed valued of $\rho_1$ used already. We can repeat this 
argument as necessary to find,  after a finite number of iterations, that
$v$ is completely controlled on $\{ s \} \times \T^n \times B_\nu(\rho_1/2)$ for $s_1 \le s \le T_1 $.
Since all our energy estimates are clearly valid backwards in the timelike variables, we can also iterate
Proposition \ref{localest} or \ref{vlocalest} backwards, starting from $s_1$ and arguing as above, to conclude
that
\beq
\mbox{$v$ is completely controlled on $\{ s \} \times \T^n \times B_\nu(\rho_1/2)$,}\quad
\mbox{$-T_1 \le s \le T_1 $.}
\label{cuv6a}\eeq
Since $T_1>T_0$, we deduce by applying \eqref{cuv2}  in both directions in the $t$ variable that
\beq
\mbox{ $u$ is controlled on $\{ (t,x) \in (-T_0,T_0)\times \R^N : \rho(t,x) \ge 
 \rho_1/4\} $. }
\label{cuv6b}\eeq
Using these and Lemma \ref{standard.ee1}, we can deduce (arguing as in the proof of \eqref{cuv2}, \eqref{cuv3}, that in fact $v$ is completely controlled on $(-T_1,T_1)\times \T^n\times B_\nu(\rho_0)$.
These estimates imply \eqref{iterate.c1} ---  \eqref{iterate.c4}.

{\bf 6}. It remains to prove \eqref{iterate.c5}. The point is that it essentially suffices to prove the same estimate
in the $y$ variables, in which  \eqref{iterate.c2} --- \eqref{iterate.c4} imply a great deal of information about the way
in which energy concentrates around $\Gamma$, which in these variables is $(-T_1,T_1)\times \T^n\times \{0\}$.
We will extract this information using Lemma \ref{L.mm2}
for the case $k=1$, and an estimate of Kurzke and Spirn \cite{ks} for $k=2$.

If $m = (m^\beta_\alpha)$ and $\calT = (\calT^\alpha_\beta)$, let us write
\[
\langle m, \calT \rangle := \int m^\beta_\alpha \ d\calT^\alpha_\beta.
\]
Then we must estimate
$
\langle m, \dep  \calT_\e(u) - \calT(\Gamma)\rangle
$
for $(m^\beta_\alpha)\in W^{1,\infty}((-T_0, T_0)\times \R^N )$ with compact support and with
$\| m \|_{W^{1,\infty}}\le 1$. To do this, let $\chi$ be a smooth function with support in
$\mbox{image}(\psi)$, and such that $\chi = 1$
on $\{ (t,x) : |t| < T_0 , \rho(t,x) < \rho_0/2\}$. Then 
\begin{align*}
\langle  m, \calT_\e(u) - \calT(\Gamma)\rangle 
\ = \ 
\langle (1-\chi ) m, \calT_\e(u) \rangle  \ + \ 
\langle  \chi m, \calT_\e(u) - \calT(\Gamma)\rangle 
\end{align*}
It is clear from the definition \eqref{emt1.def} of $\calT_\e$ that $|\calT^\alpha_{\e,\beta}(u)| \le C e_\e(u;\eta)$,
so that
\begin{align}
|\langle (1-\chi ) m, \calT_\e(u) \rangle| 
&\le 
\ \sum_{\alpha,\beta} \| m^\beta_\alpha\|_\infty 
\int_{\{(t,x)\in (-T_0,T_0)\times \R^N \ : \rho(t,x)  \ge \rho_0/2\}}
e_\e(u;\eta) \, dt\,dx \ \\
&\le C \zeta_0
\label{tens1}\end{align}
using \eqref{iterate.c1} and \eqref{iterate.c2} together with Lemma \ref{L.cuv1}.

{\bf 7}. Let us write $\bar m := \chi m$. Note that $\bar m$ is supported in $\mbox{image}(\psi)$, and  $\| \bar m \|_{W^{1,\infty}}\le C$.
We will write
\beq
\check m^\gamma_\delta(y) \ = \ 
\bar m^\beta_\alpha\circ \psi(y)  \ 
\psi^\alpha_{y_\delta}(y) \  \phi^\gamma_{x^\beta}\circ \psi(y) \sqrt{-g(y)},
\quad\quad\mbox{$\phi := \psi^{-1}$ as usual}.
\label{tens3}\eeq
Note that $\|\check m\|_{W^{1,\infty}} \le C$. 
We claim that
\begin{align}
\langle \bar m,\calT_\e(u)\rangle &= 
\int_{ (-T_1,T_1)\times \T^n\times B_\nu(\rho_0) } \check m^\gamma_\delta(y) \tilde\calT^\delta_{\e,\gamma}(v)(y) \ dy
\label{tens2}\\
\langle \bar m,\calT(\Gamma)\rangle &= 
\int_{ (-T_1,T_1)\times \T^n}  \check m^\gamma_\delta(\yt, 0)\  \tilde P_\gamma^{ \delta} \,  d \yt 
\label{tens4}\end{align}
where $\tilde \calT_{\e}(v)$ was defined\footnote{As remarked earlier $\tilde \calT_{\e}(v)$ is just the energy-momentum tensor for $u$ expressed in terms of the $y$ variables.} in \eqref{tildeT},
$\tilde P_\gamma^{\delta} = 1$ if $\delta = \gamma \in \{0,\ldots, n\}$ and $0$ otherwise.
These are
arguably obvious 
from the tensorial nature of the quantities involved. However, for the convenience of the reader,
we note that
the definitions \eqref{emt1.def} and \eqref{tildeT}  and \eqref{tens3} imply that
\[
\bar m^\beta_\alpha(t,x) \calT^\alpha_{\e,\beta}(u)(t,x) = 
\check m^\beta_\alpha(y) \tilde  \calT^\alpha_{\e,\beta}(v)(y)  (-g(y))^{-1/2}\quad\quad\mbox{ for }(t,x) = \psi(y).
\]
Then \eqref{tens2} follows, from a change of variables, noting that $ |\det D\psi| = \sqrt{-g}$, so that $dt\,dx = \sqrt{-g(y)} dy$.
%
To rewrite $\langle \bar m, \calT(\Gamma)\rangle$, note that  our proof of Lemma \ref{EMT2} (to which we refer 
for notation) showed that
\[
\langle \bar m, \calT(\Gamma)\rangle
= \  \int_{(-T, T)\times \T^n}
(\bar m^\beta_{\alpha}\circ H)  H^\alpha_{y^a} \eta_{\beta \delta} H^\delta_{y^b} \gamma^{ab} \sqrt{-\gamma} \ d\yt 
\]
where $H:(-T, T)\times \T^n\to (-T,T)\times\R^N$ is the given  map parametrizing $\Gamma$, see \eqref{Gamma.h1}, and with $a,b$ summed implictly from $0$ to $n$.
Since $\psi(\yt, 0) = H(\yt)$, we can rewrite the above integrand in terms of $\psi$ and $g$, and this leads to \eqref{tens4}.
For this it is useful to note that 
\[
g_{\alpha\beta}(\yt, 0) = \begin{cases}
\gamma_{\alpha\beta}(\yt) &\mbox{ if }\alpha, \beta \le n\\
\delta_{\alpha\beta}&\mbox{ if }\alpha, \beta >n\\
0&\mbox{ otherwise}.
\end{cases}
\]
and that
$
g^{ab} \ \psi^\delta_{y_b} \ \eta_{\delta \beta}=
\phi^a_\alpha\circ \psi \   \eta^{\alpha\gamma}\ \phi^b_\gamma\circ \psi\  \psi^\delta_{y_b}\  \eta_{\delta \beta}
= \phi^a_{x^\beta}\circ \psi
$ for $a\in \{0,\ldots, n\}$.

{\bf 8}.
We now apply our earlier estimates to control various terms in $\langle \bar m, \calT_\e(u)  - \calT(\Gamma)\rangle$,
represented as in \eqref{tens2}, \eqref{tens4} in terms of the $y$ coordinates. In these calculations,
we do {\em not} sum over  indices $\gamma$ and $\delta$ when they are repeated.

{\bf Case 1: $\gamma \ne \delta$,  $\delta\le n$.} When this holds, 
\[
|\tilde T^\delta_{\e, \gamma}| \overset{\eqref{tildeT}} =  |g^{\delta\alpha}v_{y^\alpha} v_{y^\gamma}|
\overset{\eqref{coeffs3}} \le  C (|D_\tau v|^2 + |\yn|^2 |\nabla_\nu v|^2).
\]
Thus in this case
\[
\dep \int_{(-T_1,T_1)\times \T^n\times B_\nu(\rho_0)} \check m^\gamma_\delta \ \tilde \calT^\delta_{\e, \gamma} \  \,dy
\overset{\eqref{iterate.c2}}\le C \| \check m\|_\infty  \zeta_0.
\]

{\bf Case 2: $\gamma \ne \delta$,  $\gamma\le n$.} In this case, we have the weaker estimate
\[
|\tilde T^\delta_{\e, \gamma}| \overset{\eqref{tildeT}} =  |g^{\delta\alpha}v_{y^\alpha} v_{y^\gamma}|
\overset{\eqref{coeffs3}} \le  C (|D_\tau v|^2+  |D_\tau v| |\nabla_\nu v|).
\]
So in this case we have
\begin{align*}
\dep \int_{(-T_1,T_1)\times \T^n\times B_\nu(\rho_0)} \check m^\gamma_\delta \ \tilde \calT^\delta_{\e, \gamma} \   \,dy
&\le
C\dep  \| \check m\|_\infty   (\| D_\tau v\|_2^2 + \| D_\tau v\|_2 \|\nabla_\nu v\|_2) \\
&
\overset{\eqref{iterate.c2}}\le C  \| \check m\|_\infty (\zeta_0 + \sqrt{\zeta_0}  \sqrt{\dep} \|\nabla_\nu v\|_2)\\
&\overset{\eqref{iterate.c3}} \le 
C \| \check m\|_\infty (\zeta_0 + \sqrt{\zeta_0}  \sqrt{ (\zeta_0 + 2T_1)}).
\end{align*}

{\bf Case 3: $\gamma = \delta \le n$.}
This is the only case in which $\langle m ,\calT(\Gamma)\rangle$ makes a nonzero contribution.

In this case, $|g^{\delta\alpha}v_{y^\alpha} v_{y^\delta}|
\overset{\eqref{coeffs3}} \le  C (|D_\tau v|^2+  |D_\tau v| |\nabla_\nu v|)$, so that
\begin{align*}
\tilde T^\delta_{\e, \delta} 
&\ \  \ \ \  = \ \  \  \ \  \ \frac 12 g^{\alpha\beta} v_{y_\alpha} v_{y^\beta} + \frac 1{\e^2}F(v) + O(|D_\tau v|^2+  |D_\tau v| |\nabla_\nu v|)\\
&\overset{\eqref{coeffs3}, \eqref{pos1}}=\ 
\frac 12 |\nabla_\nu v|^2 + \frac 1{\e^2}F(v) + O(|D_\tau v|^2+  |D_\tau v| |\nabla_\nu v|).
\end{align*}
Thus
\begin{align*}
\dep \int_{(-T_1,T_1)\times \T^n\times B_\nu(\rho_0)} \check m^\delta_\delta \ \tilde \calT^\delta_{\e, \delta}  \,dy
&= 
\dep \int_{(-T_1,T_1)\times \T^n\times B_\nu(\rho_0)} \check m^\delta_\delta \  e_{\e,\nu}(v) \   \,dy +
O( \| \check m\|_\infty \sqrt{\zeta_0}).
\end{align*}
The contribution to $\langle \bar m, \calT_\e(u) - \calT(\Gamma)\rangle$ from a summand 
 with $\delta = \gamma \le n$ is
 thus
\begin{align*}
& \int_{(-T_1,T_1)\times \T^n\times B_\nu(\rho_0)} \check m^\delta_\delta \ \dep  e_{\e,\nu}(v) \   \,dy \\
&
\quad\quad 
\quad\quad 
-   \int_{(-T_1,T_1)\times \T^n}m^\delta_\delta(\yt, 0 )  \   d\yt +
O( \| \check m\|_\infty\sqrt{\zeta_0}) \\
&= 
 \int_{(-T_1,T_1)\times \T^n\times B_\nu(\rho_0)}\left[ \check m^\delta_\delta(\yt,\yn) - \check m^\delta_\delta(\yt, 0)\right] \ \dep  e_{\e,\nu}(v) \   \,d y\\
&
\quad\quad 
\quad\quad 
-   \int_{(-T_1,T_1)\times \T^n}m^\delta_\delta(\yt, 0 )\left( 1 - \dep \int_{B_\nu(\rho)} e_{\e,\nu}(v)(\yt, \yn) d\yn\right)  \  d\yt +
O( \| \check m\|_\infty\sqrt{\zeta_0}) \\
&=: A + B + O( \| \check m\|_\infty\sqrt{\zeta_0}).
\end{align*}
To estimate $A$, note that $| \check m^\delta_\delta(\yt,\yn) - \check m^\delta_\delta(\yt, 0)| \le \| \check m\|_{W^{1,\infty}} |\yn| \le C |\yn|$, so that 
\[
|A| \le \left(\dep \int |\yn|^2   e_{\e,\nu}(v) \   \,d y)\right)^{1/2} \ \left(\dep \int   e_{\e,\nu}(v) \   \,d y)\right)^{1/2} \
\le C \sqrt {\zeta_0}
\]
after arguing as in Case 2 above to estimate  $\int e_{\e,\nu}(v) \le C$.

As for the other term,  since $\|\check m\|_\infty \le C$,
\begin{align*}
|B|
&\le C 
  \int_{(-T_1,T_1)\times \T^n} \left|\Theta_1(\yt) \right|  \  d\yt, \quad \quad\mbox{ for }
  \Theta_1(\yt):=   \dep \int_{B_\nu(\rho)} e_{\e,\nu}(v)(\yt, \yn)d\yn - 1 .
\end{align*}
Let us say that $\yt$ is {\em good} if $\calD_\nu(v(\yt)) \le \kappa_3$, where $\kappa_3$
is the constant from Lemma \ref{L.mm2}  and Proposition \ref{Pv2} for $k=1$ and  $k=2$ respectively.
A point will be called  {\em bad} if it is not good.
In particular, these results show that if $\yt$ is good, then
\[
\Theta_1(\yt)  \ge \ \  \left\{
\begin{array}{ll}
-C e^{-c/\e}&\mbox{ if }k=1\\
-C|\ln \e|^{-1}&\mbox{ if }k=2 
\end{array}\right\}\quad
\ge - C\zeta_0\ \ \ \mbox{ in both cases}
\]
since  $\dep \le \zeta_0$.
Thus, since clearly $\Theta_1(\yt) \ge -1$ everywhere, we see that
\[
|\Theta_1(\yt)|
\le \begin{cases}
\Theta_1(\yt) + C \zeta_0 &\mbox{ if $\yt$ is good}\\
\Theta_1(\yt) + 2 &\mbox{ if $\yt$ is bad}
\end{cases}
\]
Thus we compute
\begin{align*}
|B| 
&\le  \ C \int_{ good\   pts} (\Theta_1(\yt) + C \zeta_0 ) d\yt + C  \int_{ bad\  pts} (\Theta_1(\yt) + 2 ) d\yt
\\
&\le
C \int_{  (-T_0,T_0)\times \T^n } \Theta_1(\yt) d\yt + C \zeta_0 + 2\calH^{1+n}(\{  \yt \in (-T_0,T_0)\times \T^n \ : \ \yt \mbox{ is bad}.\})
\end{align*}
To conclude the estimate, we note that \eqref{iterate.c3} implies that $ \int_{  (-T_0,T_0)\times \T^n } \Theta_1(\yt) d\yt \le C \zeta_0$,
and \eqref{iterate.c4} together with Chebyshev's inequality implies that 
\[
\calH^{1+n}(\{  \yt \in (-T_0,T_0)\times \T^n \ : \ \yt \mbox{ is bad}\})
\le C \int_{(-T_0,T_0)\times \T^n} \calD_\nu(v(\yt, \cdot)) d\yt \le C \zeta_0.
\]
Thus $|B| \le C \zeta_0$.

{\bf Case 4: $\gamma, \delta  > n$.}

Here we consider the cases $k=1,2$ separately.

{\bf k=1}: Here the assumption of Case 4 reduces to $\gamma = \delta  = N$, and we using \eqref{k1coeffs}
we see that 
\[
\tilde \calT^N_{\e,N} \ = \ \frac 12 \sum_{a,b=0}^n g^{ab}v_{y^a}v_{y^b}  - \frac 12 (v_{y^N})^2  + \frac 1{\e^2}F(v)
= - \frac 12 (v_{y^N})^2  + \frac 1{\e^2}F(v) + O(|D_\tau v|^2).
\]
We will write
\[
\Theta_2(\yt) :=  \dep \int_{B_\nu(\rho_0)} \left|  v_{y^N}^2 - \frac 1{\e^2}F(v)\right| d\yn,
\]
and we will now say that $\yt\in (-T_1,T_1)\times \T^n$ is {\em good} if 
\beq
\calD_\nu(v(\yt)) \le \kappa_3,\quad\quad\quad\mbox{ and  in addition }
\Theta_1(\yt) \le \kappa_4
\label{good2}\eeq
for $\kappa_3,\kappa_4$ found in Lemma \ref{L.mm2}. Then Lemma \ref{L.mm2} implies that
if $\yt$ is good, then
\[
\Theta_2(\yt) \le C \sqrt {|\Theta_1(\yt)| + \zeta_0}.
\]
Thus  using H\"older's inequality
\begin{align*}
\int_{good\   pts}\Theta_2(\yt) d\yt 
&\le  \sqrt {\ C \int_{ good\   pts}( |\Theta_1(\yt)| + C \zeta_0 ) d\yt }
\\
&\le
C\sqrt{\zeta_0}
\end{align*}
using estimates from Case 3 above.
And if $\yt$ is bad, then
\[
\Theta_2(\yt) \le C\left( 1 + \Theta_1(\yt)\right)
\]
so that
\begin{align*}
\int_{bad\   pts}\Theta_2(\yt) d\yt 
&\le  \sqrt {\ C \int_{ bad\   pts}( \Theta_1(\yt) + 1) d\yt }
\\
&\le
C\sqrt{\zeta_0} +
+ C \left(\calH^{1+n}(\{  \yt \in (-T_0,T_0)\times \T^n \ : \ \yt \mbox{ is bad}\})\right)^{1/2}\\
&
\le C \sqrt{\zeta_0}
\end{align*}
where at the end we used Chebyshev's inequality with \eqref{iterate.c3}, \eqref{iterate.c4}.
Hence
\[
|\dep \int_{(-T_1,T_1)\times \T^n\times B_\nu(\rho_0)} \check m^N_N \ \tilde \calT^N_{\e,N}  \,dy|
\
\le \ C\int_{ (-T_0,T_0)\times \T^n} \Theta_2(\yt) d \yt + O(\zeta_0) \le C \sqrt{\zeta_0}
\]

{\bf k=2}: We claim that when $k=2$, 
\beq
|\dep \int_{(-T_1,T_1)\times \T^n\times B_\nu(\rho_0)} \check m^\delta_\gamma \ \tilde \calT^\gamma_{\e,\delta}  \,dy|
\le  C \sqrt{\zeta_0}
\label{k2fest}\eeq
if $\delta, \gamma \in \{N-1,N\}$. This will complete the proof of  \eqref{iterate.c5}.
To prove \eqref{k2fest}, we first note that Proposition \ref{P.transformation} implies that 
\begin{align}
\left(\begin{array}{ll}
\tilde\calT^{N-1}_{\e, N-1}&\tilde\calT^{N-1}_{\e, N}\\
\tilde\calT^{N}_{\e, N-1}&\tilde\calT^{N}_{\e, N}
\end{array}\right)
&=
\left( \begin{array}{cc} 
 \frac 12( |v_{y^N}|^2 - |v_{y^{N-1}}|^2) + \frac 1{\e^2}F(v)
& - v_{y^{N-1}}\cdot v_{y^N}\\
- v_{y^{N-1}}\cdot v_{y^N}
&\frac 12(- |v_{y^N}|^2 + |v_{y^{N-1}}|^2) + \frac 1{\e^2}F(v)
\end{array}\right) \nonumber \\
&+ O(|D_\tau v|^2 + |\yn|^2 |\nabla_\nu v|^2).
\label{explicitT}\end{align}
At this point we need Theorem 1 from Kurzke and Spirn \cite{ks}, which implies that 
if $w\in H^1(B_\nu(\rho_0), \R^2)$ and 
\beq
\calD_\nu(w) \le \kappa_3,\quad\quad\quad\mbox{ and } \quad\quad
\Theta_1(\yt) \le \frac 32
\label{good3}\eeq
then 
\begin{align*}
&
\left| \dep \int_{B_\nu(\rho_0)} 
\left( 
\begin{array}{cc} 
|w_{y^{N-1}}|^2 	&  w_{y^{N-1}}\cdot w_{y^N}\\
w_{y^{N-1}}\cdot v_{y^N}	& |w_{y^N}|^2 
\end{array}
\right) 
d\yn
-
\left(
\begin{array}{ll}1&0
\\0&1\end{array}
\right)
\right|
\\
&
\quad\quad\quad\quad\quad\quad\quad\quad\quad\quad\quad\quad
 \le C \left(\dep  \int_{B_\nu(\rho_0)} e_{\e,\nu}(w) d\yn - 1 + C\dep\right)^{1/2}.
\end{align*}
(The main hypothesis of the Kurzke-Spirn estimate is \eqref{s1}, and we have shown in the proof of Proposition \ref{Pv2} that this follows from \eqref{good3}.)
Accordingly, we will continue to say (exactly parallel to the case $k=1$, see \eqref{good2}) say that $\yt\in (-T_1,T_1)\times \T^n$ is {\em good} if  $v(\yt, \cdot)\in H^1(B_\nu(\rho_0);\R^2)$ satisfies \eqref{good3}. As usual, a point that is not good is said to be {\em bad}.
It follows as usual from Chebyshev's inequality and \eqref{iterate.c3}, \eqref{iterate.c4} that
\[
\calH^{1+n}(\{  \yt \in (-T_0,T_0)\times \T^n \ : \ \yt \mbox{ is bad}\}) \le C \zeta_0.
\]
The Kurzke-Spirn inequality implies that if $\yt$ is good, then 
\[
\frac \dep 2\int_{B_\nu(\rho)}  |\nabla_\nu v(\yt,\yn)|^2 d\yn
\ge 1 - C \left(\Theta_1(\yt) + C \dep\right) ^{1/2}.
\]
Thus for a good point $\yt$, 
\begin{align*}
\dep\int_{B_\nu(\rho_0)} \frac 1{\e^2}F(v)(\yt, \yn) d\yn &=
\Theta_1(\yt) + \left(1 - \frac \dep 2\int_{B_\nu(\rho)}|\nabla_\nu v(\yt,\yn)|^2 d\yn\right)\\
&\le
\Theta_1(\yt) +  C\left(\Theta_1(\yt) + C \dep\right) ^{1/2}.
\end{align*}
Similarly, the Kurzke-Spirn estimate also implies that if $\yt$ is good, then
\[
\dep \left|\int_{B_\nu(\rho_0)} \  ( |v_{y^N}|^2 - |v_{y^{N-1}}|^2)  \  d\yn \right|
+
\dep \left|\int_{B_\nu(\rho_0)} \  v_{y^N}\cdot v_{y^{N-1}}  \  d\yn \right|
\le C(\Theta_1(\yt) + C \dep)^{1/2}
\]
Combining these and recalling \eqref{explicitT}, we see that if $\yt$ is good, then for $\gamma,\delta\in \{N-1,N\}$, 
\beq
\left| \dep   \int_{B_\nu(\rho_0)}\calT^\delta_{\e,\gamma}(\yt, \yn)  d\yn\right| \le  C(\Theta_1(\yt) + C \dep)^{1/2}
+\Theta_3(\yt)
\label{good4}\eeq
where $\Theta_3(\yt) := \dep \int_{B_\nu(\rho_0)} (|D_\tau v|^2 + |\yn|^2 |\nabla_\nu v|^2) d\yn$.
Also, if $\yt$ is bad, then \eqref{explicitT} implies that
\[
\dep \int_{B_\nu(\rho_0)} \left|\calT^\delta_{\e,\gamma}(\yt, \yn) \right| d\yn \le  C(\Theta_1(\yt) + 1).
\]
This last fact together with the estimate of the size of the bad set and \eqref{iterate.c3} implies that
\[
\left|
\dep \int_{(-T_1,T_1)\times \T^n\times B_\nu(\rho_0)} \check m^\delta_\gamma \ \tilde \calT^\gamma_{\e,\delta}  \,dy\right|
\le |A| +|B| +  
C \zeta_0,
\]
where
\begin{align*}
A 
&:=
\dep \int_{good\ pts \in (-T_1,T_1)\times \T^n} \int_{B_\nu(\rho)}\check m^\delta_\gamma(\yt, 0) \ \tilde \calT^\gamma_{\e,\delta}  \,d\yn d\yt\\
B 
&:=
\dep \int_{good\ pts \in (-T_1,T_1)\times \T^n} \int_{B_\nu(\rho)} [ \check m^\delta_\gamma(\yt, \yn) -\check m^\delta_\gamma(\yt, 0) ]\ 
\tilde \calT^\gamma_{\e,\delta}  \,d\yn d\yt\\
\end{align*}
From  \eqref{good4}, H\"older's inequality, \eqref{iterate.c2}, and \eqref{iterate.c3}, we see that
\[
|A| \le \| \check m\|_\infty  \int_{good\ pts \in (-T_1,T_1)\times \T^n}
 C[(\Theta_1(\yt) + C \dep)^{1/2}
+\Theta_3(\yt)] d\yt \ \le \ C\sqrt{\zeta_0}.
\]
And since $\|\hat m\|_{W^{1,\infty}}\le C$,
\[
|B|\le C\dep \int_{(-T_1,T_1)\times \T^n\times B_\nu(\rho_0)} |\yn| e_{\e, \nu}(v) \ dy.
\]
We have shown in Case 3 above that the right-hand side above is bounded by $C\sqrt{\zeta_0}$, so we
find that $|A| + |B| \le C\sqrt{\zeta_0}$. We therefore have proved
\eqref{k2fest}, and \eqref{iterate.c5} follows.
\end{proof}

\section{Appendix} \label{S:geom.lemmas}

In this appendix we give the proof of Proposition \ref{P.transformation}.
Recall that we have defined $G = D\psi^T \, \eta \, D\psi$, where 
$\psi(\yt,\yn) := \ H(\yt) + \sum_{i=1}^k \bar \nu_i(\yt)\,  y^{n+i}$.  Here $H$ is the given
parametrization of the minimal surface $\Gamma$, and the vectors $\{ \bar\nu_i\}$
form an orthonormal frame for the normal bundle of $\Gamma$, see \eqref{nu1}.
The proposition asserts certain properties of $g := \det G$ and $(g^{ij}) = G^{-1}$.
We will use the following lemma to read off properties of $G^{-1}$ from 
those of $G$.

\begin{lemma}
Let $M$ be a matrix written in block form as 
\[
M = \left( \begin{array}{ll}A&B\\C&D\end{array}\right)
\]
(where the blocks need not be of equal size, ie
$A\in M^{n\times n}, B\in M^{n\times m}, C\in M^{m\times n}$ and $D\in M^{m\times m}$
for some $m,n$.
Assume that
\beq
N = 
\left( \begin{array}{cc}(A- BD^{-1}C)^{-1} &-A^{-1}B(D- CA^{-1}B)^{-1}\\
-D^{-1}C(A- BD^{-1}C)^{-1}&(D- CA^{-1}B)^{-1}\end{array}\right)
\label{sym.block.inv}\eeq
is well-defined. Then $N = M^{-1}$.
\end{lemma}

This is proved by simply verifying that that $MN = I$.

\begin{proof}[Proof of Propositions \ref{P.transformation} and \ref{k1.trans}]
We will think of $\nu(\yt)$ as a $(1+N)\times k$ matrix with columns $\bar \nu_i$, $i=1,...,k$, 
and of $\yn$ as a $k\times 1$ vector,
so that $\nu \yn := \sum_{i=1}^k \bar \nu_i(\yt)\,  y^{n+i}$.

{\bf Step 1}. 
To start,   note  that $\nabla_\nu \psi(y) = \nu(\yt)$, so the choice \eqref{nu1} of $\nu$ implies that
$G$ can be written in block form as 
\[
G = \left(
\begin{array}{cc}
G_{\tau\tau}	 & G_{\tau\nu} \\
G_{\nu\tau} & I_k
\end{array}\right)
\]
where
\[
G_{\tau\tau} :=
D_\tau \psi^T \, \eta \, D_\tau \psi\in M^{(1+n)\times (1+n)}
\ \ \ \ \ \mbox{ and }
G_{\tau\nu} = G_{\nu\tau}^T := D_\tau (\nu  \yn)^T \, \eta  \nu \in M^{(1+n)\times k}
\]
and  $I_k$ denotes the $k\times k$ identity matrix. 
Observe that
\beq
|G_{\tau\nu}| \le C |\yn| \quad\mbox{ when } k\ge 2\quad\quad 
\mbox{ and }
G_{\tau\nu} \equiv 0 \quad\mbox{ for } k=1.
\label{tnests}\eeq
The second assertion above follows from differentiating the identity 
$(\nu^T \eta \nu)(\yt) \equiv 1$.
Since $\psi(\yt,0) = H(\yt)$, we can write $G_{\tau\tau}(\yt,0)$ in block form as
\[
G_{\tau\tau}(\yt, 0)
= 
 \left(
\begin{array}{cc}
H_{y_0}^T \eta H_{y_0}	 &H_{y_0}^T \eta \nabla H \\
\nabla H^T \eta H_{y_0} & \nabla H^T\eta \nabla H
\end{array}\right) 
= 
 \left(
\begin{array}{cc}
-1+ |h_{y^0}|^2 	 &0 \\
0& \nabla  h^T  \nabla h
\end{array}\right) 
\]
where we have used \eqref{Gamma.h1}, \eqref{Gamma.h2}. It then follows from \eqref{Gamma.h3} and the
smoothness of $H$
that $G_{\tau\tau}(\yt, 0)$ is invertible, with uniformly bounded inverse, for $\yt\in [-T_1, T_1]\times \T^n$.
It follows by continuity that that $G_{\tau\tau}(y)$ is invertible, with uniformly bounded inverse, for
$y\in  [-T_1, T_1]\times \T^n\times B_\nu(\rho_0)$, if $\rho_0$ is chosen small enough.

{\bf Step 2}. Next, we note that $G_{\tau\tau}(y) = G_{\tau\tau}(\yt,0) + O( |\yt|)$, and we use \eqref{sym.block.inv} and \eqref{tnests} to find that, taking $\rho_0$ smaller if necessary,
$G(y)$ is invertible for $y\in   [-T_1, T_1]\times \T^n\times B_\nu(\rho_0)$, with
\begin{align}
G^{-1}(y) \ 
&=
 \ \left( \begin{array}{cc}(G_{\tau\tau}(\yt, 0) + O(|\yn|))^{-1} &O(|\yn|)\\
O(|\yn|)&(I_k- O(|\yn|^2))^{-1}\end{array}\right)
\nonumber\\
&
 = \ 
 \ \left( \begin{array}{cc}G_{\tau\tau}^{-1}(\yt, 0)  &0\\
0&I_k  \end{array}\right) 
 + \ 
  \ \left( \begin{array}{cc} O(|\yn|)&O(|\yn|)\\
O(|\yn|)&  O(|\yn|^2)\end{array}\right).
\label{ginv}\end{align}
We have used (more than once) the fact that 
$G_{\tau\tau}^{-1}(\yt,0)$ is uniformly bounded, which implies that $|(G_{\tau\tau}+ A)^{-1} - G_{\tau\tau}^{-1}| \le C|A|$ for $A$ sufficiently small, with a uniform constant $C$.

From \eqref{ginv} and \eqref{Gamma.h3}, we easily conclude that \eqref{pos1}, \eqref{coeffs3}, and the first estimate of \eqref{coeffs1} hold. Moreover, if $k=1$ then, in view of \eqref{tnests},
\[
G^{-1}(y) \ =  \ \left( 
\begin{array}{cc}G_{\tau\tau}(y)^{-1}  &0\\
0&I_k
\end{array}\right)
=
G^{-1}(y) \ =  \ \left( 
\begin{array}{cc}G_{\tau\tau}(\yt, 0)^{-1} + O(|\yn|)  &0\\
0&I_k
\end{array}\right)
\]
from which we infer \eqref{k1coeffs} and \eqref{pos1a}.

To establish the second conclusion of \eqref{coeffs1}, we differentiate the identity $G^{-1}G = I$
to find that 
\[
G^{-1}_{y_0} = - G^{-1} \, G_{y_0} \, G^{-1}.
\]
Our earlier expression for $G$ implies that 
\[
G_{y_0} = \left(
\begin{array}{cc}
G_{\tau\tau, y_0}	 &O(|\yn|) \\
O(|\yn|) & 0
\end{array}\right)
\]
and one can readily check that this implies 
that $|g^{\alpha\beta}_{y^0}\xi_\alpha\xi_\beta| \le C(|\xi_\tau|^2 + |\yn|^2\, |\xi_\nu|^2)$,
which completes the proof of \eqref{coeffs1}.

{\bf Step 3} It remains to establish \eqref{coeffs2}.  
To do this, fix $\zeta \in C^\infty_0([-T_1,T_1]\times \T^n;\R^{k})$, and for $\sigma\in \R{}$
define
\begin{equation}
f(\sigma) = \int_V \left ( -\det(D_\tau H_\sigma^T \ \eta \ DH_\sigma) \right )^{1/2} d\yt, 
\label{def_sigma}
\end{equation}
where
$$
H_\sigma(\yt) = H(\yt) + \sigma \nu(\yt) \zeta(\yt) \ = \ \psi(\yt, \sigma\zeta(\yt)).
$$
Note that for $\sigma$ small, $H_\sigma$ parametrizes a surface $\Gamma_\sigma$ that is a small variation of the original surface $\Gamma$.  
Because $\Gamma$ is a Minkowski minimal surface, it follows that
$f'(0)=0$. 
We will show that this yields the conclusion of the lemma.

Thinking of $D \zeta$ as a $k\times (1+n)$ matrix, a direct computation yields
\[
DH_\sigma(\yt) 
= 
D_\tau  \psi(\yt, \sigma \zeta(\yt)) + \sigma \nu(\yt) D\zeta(\yt)\\
\]
It then follows from \eqref{nu1}
that $DH_\sigma^T \  \eta \  DH_\sigma$ has the form
\[
[DH_\sigma^T \,  \eta \,  DH_\sigma](\yt) \ = \ 
[D_\tau \psi^T \, \eta \, D_\tau  \psi](\yt, \sigma \zeta(\yt))  + \sigma^2 B(\yt).
\]
for some matrix $B(\yt)$ that depends smoothly on $\yt$.
Since
\[
\left. \frac d{d\sigma} \det( A(\sigma) + \sigma^2 B)\right|_{\sigma=0} = 
\left. \frac d{d\sigma} \det A(\sigma) \right|_{\sigma=0}
\]
if $A(\sigma)$ are square matrices depending smoothly on a real parameter $\sigma$,
it follows that
\beq
\frac d{d\sigma} \det(DH_\sigma^T \ \eta \ DH_\sigma) (\yt)
= 
\frac d{d\sigma}  \det  ( D_\tau \psi^T \ \eta \  D_\tau  \psi)(\yt, \sigma \zeta(\yt)) 
\label{reduce1}\eeq
at $\sigma = 0$.

{\bf  Step 4}. We next note that
\begin{equation}
\det  ( D_\tau \psi^T \ \eta \  D_\tau  \psi)(\yt, \sigma \zeta(\yt))
=
\det  (  D\psi^T \ \eta \  D \psi)(\yt, \sigma \zeta(\yt)) + O(\sigma^2).
\label{reduce2}\end{equation}
Indeed, this follows by rather easy linear algebra considerations from the
fact that
\beq
D\psi^T \ \eta \  D \psi(\yt,\sigma \zeta(\yt))
= 
\left(
\begin{array}{ll}
D_\tau \psi^T \ \eta \  D_\tau  \psi &O(\sigma)\\
O(\sigma)&I_k+O(\sigma^2)
\end{array}\right).
\label{G.form}\eeq
By combining \eqref{reduce1}  and \eqref{reduce2} that
\[
\left. \frac d{d\sigma} \det(DH_\sigma^T \ \eta \ DH_\sigma) (\yt)\right|_{\sigma=0}
= 
\left.  \frac d{d\sigma} g(\yt,\sigma\zeta (\yt)) \right|_{\sigma=0}
=
\nabla_\nu g(\yt,0) \cdot \zeta
\]
at $\sigma = 0$.  Also, it follows from \eqref{Gamma.h3} and continuity that $\det(DH_\sigma^T \ \eta \ DH_\sigma) (\yt)$
and $g(\yt,\sigma\zeta(\yt))$ are bounded away from $0$ for $\zeta$ small enough, and hence
\[
\left. \frac d{d\sigma} \left ( -\det(DH_\sigma^T \ \eta \ DH_\sigma) (\yt) \right )^{1/2}
\right|_{\sigma=0}
= 
\nabla_\nu \sqrt{-g} \cdot \zeta 
%
\] 
Thus the identity $f'(0) = 0$
reduces to
\[
0 
= 
\int    \nabla_\nu \sqrt{-g} \cdot \zeta \ \   d\yt
\]
Since $\zeta$ is arbitrary, we
conclude that $   \nabla_\nu \sqrt{-g}  = 0$. This fact and \eqref{ginv} imply
the required estimate \eqref{coeffs2}.
\end{proof}



%

\end{document}